\documentclass[11pt]{article}

\usepackage[T1]{fontenc}
\usepackage{lmodern}
\usepackage[margin=0.72in]{geometry}
\usepackage{amsfonts}
\usepackage{amsthm}
\usepackage{aliascnt}
\usepackage{graphicx,epstopdf}

\usepackage{mathrsfs}
\usepackage{mathtools}
\usepackage{bm}
\usepackage{multirow}
\usepackage{color}
\usepackage{hyperref}
\usepackage{xr-hyper} 
\usepackage{xcolor}

\newcommand{\revblue}[1]{#1} 
\newcommand{\Gint}{{\mathring{G}}}

\usepackage{caption, subcaption}%
\usepackage{array}
\usepackage{booktabs}
\usepackage{tabularx}

\usepackage{xparse} 
\usepackage{stmaryrd} 
\usepackage{rotfloat}
\usepackage{cleveref}

\usepackage{setspace}
\usepackage{enumerate}
\usepackage{amssymb}
\usepackage{longtable}
\usepackage{placeins}
\usepackage{afterpage}

\allowdisplaybreaks

\usepackage{tikz} 
\usetikzlibrary{decorations.pathreplacing,calc,arrows.meta,positioning}

\usepackage{hyperref}

\usepackage{accents}

\usepackage[most]{tcolorbox} 

\ifpdf
  \DeclareGraphicsExtensions{.eps,.pdf,.png,.jpg}
\else
  \DeclareGraphicsExtensions{.eps}
\fi

\theoremstyle{plain}
\newtheorem{theorem}{Theorem}[section]
\newaliascnt{proposition}{theorem}
\newtheorem{proposition}[proposition]{Proposition}
\aliascntresetthe{proposition}
\newaliascnt{lemma}{theorem}
\newtheorem{lemma}[lemma]{Lemma}
\aliascntresetthe{lemma}
\newaliascnt{corollary}{theorem}
\newtheorem{corollary}[corollary]{Corollary}
\aliascntresetthe{corollary}
\newaliascnt{assumption}{theorem}

\aliascntresetthe{assumption}
\theoremstyle{remark}
\newaliascnt{remark}{theorem}
\newtheorem{remark}[remark]{Remark}
\aliascntresetthe{remark}
\theoremstyle{definition}
\newaliascnt{expl}{theorem}

\aliascntresetthe{expl}
\numberwithin{equation}{section}

\crefname{theorem}{theorem}{theorems}
\Crefname{theorem}{Theorem}{Theorems}
\crefname{proposition}{proposition}{propositions}
\Crefname{proposition}{Proposition}{Propositions}
\crefname{lemma}{lemma}{lemmas}
\Crefname{lemma}{Lemma}{Lemmas}
\crefname{corollary}{corollary}{corollaries}
\Crefname{corollary}{Corollary}{Corollaries}
\crefname{assumption}{assumption}{assumptions}
\Crefname{assumption}{Assumption}{Assumptions}
\crefname{remark}{remark}{remarks}
\Crefname{remark}{Remark}{Remarks}
\crefname{expl}{example}{examples}
\Crefname{expl}{Example}{Examples}
\crefname{section}{section}{sections}
\Crefname{section}{Section}{Sections}
\crefname{subsection}{section}{sections}
\Crefname{subsection}{Section}{Sections}
\crefname{subsubsection}{section}{sections}
\Crefname{subsubsection}{Section}{Sections}
\crefname{equation}{equation}{equations}
\Crefname{equation}{Equation}{Equations}
\crefname{figure}{figure}{figures}
\Crefname{figure}{Figure}{Figures}
\crefname{table}{table}{tables}
\Crefname{table}{Table}{Tables}

\allowdisplaybreaks
\allowbreak
\makeatletter
\def\widebreve{\mathpalette\wide@breve}
\def\wide@breve#1#2{\sbox\z@{$#1#2$}%
	\mathop{\vbox{\m@th\ialign{##\crcr
				\kern0.08em\brevefill#1{0.8\wd\z@}\crcr\noalign{\nointerlineskip}%
				$\hss#1#2\hss$\crcr}}}\limits}
\def\brevefill#1#2{$\m@th\sbox\tw@{$#1($}%
	\hss\resizebox{#2}{\wd\tw@}{\rotatebox[origin=c]{90}{\upshape(}}\hss$}
\makeatletter

\newcommand{\dt}{\Delta t}
\newcommand{\dx}{\Delta x}
\newcommand{\dy}{\Delta y}

\newcommand{\email}[1]{\href{mailto:#1}{#1}}
\newenvironment{keywords}{\par\small\noindent\textbf{Keywords. }}{\par\normalsize}
\newenvironment{MSCcodes}{\par\small\noindent\textbf{MSC2020. }}{\par\normalsize}

\usepackage{xcolor}
\usepackage[normalem]{ulem} 

\newcommand{\reviewmode}{1}

\usepackage{CJKutf8}
\ifnum\reviewmode=1

\newcommand{\cmmDel}[1]{{\color{red!70!black}\sout{#1}}}

\newcommand{\cmmNote}[1]{{\color{orange!90!black}\footnotesize~[\textit{CMAME: }\begin{CJK*}{UTF8}{gbsn}#1\end{CJK*}]}}
\else

\newcommand{\cmmDel}[1]{}

\newcommand{\cmmNote}[1]{}
\fi

\date{}

\title{Invariant domain preservation for hybrid point-value and cell-average discretizations of hyperbolic equations on general meshes
	\thanks{{This work was partially supported by Science Challenge Project (No.~TZ2025007),  Shenzhen Science and Technology Program (Nos.~JCYJ20250604144300001, RCJC20221008092757098, and JCYJ20240813095709013), 
	National Natural Science Foundation of China (Nos.~12401525, 12471382, and 12171227), and 
	the Guangdong Basic and Applied Basic Research Foundation (2024A1515012329).}}}

\author{Shengrong Ding
	\thanks{School of Science, Sun Yat-Sen University, Shenzhen, Guangdong 518107, China. 
  (\email{dingshr7@mail.sysu.edu.cn}, \email{cuishm3@mail.sysu.edu.cn}).}
\and 
Shumo Cui \footnotemark[2]
\and 
R\'{e}mi Abgrall 
\thanks{Institute of Mathematics, University of Z\"{u}rich, 8057 Z\"{u}rich, Switzerland. 
	(\email{remi.abgrall@math.uzh.ch}). }
\and
Kailiang Wu 
\thanks{Department of Mathematics and Shenzhen International Center for Mathematics, Southern University of Science and Technology, Shenzhen 518055, China. 
	(\email{wukl@sustech.edu.cn}).}
}

\usepackage{amsopn}

\usepackage{todonotes}

\ifpdf
\hypersetup{
  pdftitle={Invariant domain preservation for hybrid point-value and cell-average discretizations of hyperbolic equations on general meshes},
  pdfauthor={Shengrong Ding, Shumo Cui, R\'emi Abgrall, and Kailiang Wu},
  hidelinks
}
\fi

\begin{document}

\maketitle


\begin{abstract}
This paper presents a unified invariant-domain-preserving (IDP) framework for hybrid discretizations of hyperbolic conservation laws, including active flux and PAMPA (point-average-moment polynomial-interpreted) methods, in which cell averages are updated conservatively while cell-boundary point values evolve under a possibly non-conservative operator. This framework provides a nontrivial generalization of the one-dimensional methodology in [Abgrall, Jiao, Liu, and Wu, SIAM J. Sci. Comput., 47(6):A3536--A3565, 2025] to general multidimensional polygonal meshes and broader classes of hybrid discretizations. The main computational challenge is to preserve admissibility for these two coupled sets of states under a single local stability condition, without adding evolved degrees of freedom or relying on post-update repairs.
For the point-value update, we introduce an admissibility transform based on a barrier--Legendre map for convex admissible interiors described by concave constraints, and prove that the inverse map is globally defined and Lipschitz continuous on the relevant sets. 
For the conservative cell-average update, we establish a structural obstruction theorem: the single-state continuous physical flux built from admissible boundary traces alone cannot provide a conservative IDP guarantee, for any prescribed Courant--Friedrichs--Lewy (CFL) number, when internal reconstruction values are uncontrolled. This identifies the missing local control that must be supplied by an additional IDP flux mechanism. 
To realize this mechanism explicitly, we combine cell average decompositions (CAD), geometric quasilinearization, and local \emph{a priori} scaling to construct admissible, generally discontinuous trace states before flux evaluation; the scaling preserves the cell average and remains inactive whenever the relevant trace and CAD control values are already admissible. 
Under an explicit trace-based CFL condition, with constants determined by local CAD weights and trace-state wave-speed bounds, the coupled hybrid update preserves the prescribed invariant domain. Concrete third-order schemes are developed on triangular, Cartesian, convex quadrilateral, and general convex polygonal meshes. Numerical results for scalar equations and compressible Euler flow benchmarks demonstrate the designed order of accuracy for smooth solutions and the strict preservation of physical admissibility.
\end{abstract}




\section{Introduction}

High-order discretizations of hyperbolic conservation laws must preserve the admissible set, or invariant domain, of the continuous governing equations. Throughout this paper, we distinguish an open admissible interior $\Gint$ and its closed counterpart $G=\overline{\Gint}$. For the compressible Euler equations, for instance, $\Gint$ is characterized by strict positivity of density and pressure. Once a numerical state leaves the closed admissible set $G$, hyperbolicity may be lost and the numerical simulation may break down \cite{shu2016high,guermond2016invariant}. Invariant domain preservation is therefore a structural stability requirement for high-order methods, especially in regimes involving shocks, near-vacuum states, or strong nonlinear waves.

For classical finite volume and (continuous or discontinuous) finite element discretizations, rigorous invariant-domain-preserving (IDP) frameworks have been extensively developed \cite{zhang2010maximum,zhang2010positivity,guermond2016invariant,guermond2017invariant,guermond2018second,xu2014parametrized,zhang2011b}. For a systematic overview of high-order IDP methods for hyperbolic and related systems, see \cite{Wu2026IDP}. In this paper, we consider \emph{hybrid} schemes that dynamically couple a conservative update of cell averages with a possibly non-conservative evolution of point values located on the mesh skeleton. Representative examples include Active Flux (AF) and its recent generalizations \cite{eymann2011AF1,eymann2013AF3,abgrall2023combination,abgrall2025EulerAF,barsukow2025HOrder3}, such as the triangular-mesh AF discretization for compressible flows \cite{abgrall2025Tri}, as well as point-average-moment polynomial-interpreted (PAMPA) and virtual element method (VEM) formulations \cite{abgrall2023combination,abgrall2024Polygonal}. In these methods, conservative cell averages provide the finite-volume foundation, while shared point values on the mesh skeleton carry directional information and naturally support compact high-order reconstructions with continuous traces across cell interfaces. This continuity is one of the attractive structural features of such schemes, especially on complex geometries. At the same time, it raises a distinct IDP challenge for the conservative update: a continuous boundary trace is a single shared state, whereas an invariant-domain flux argument typically requires sufficient local one-sided control to ensure admissibility of the cell-average update.

Establishing theoretical admissibility for such hybrid methods therefore poses a distinct challenge. The Zhang--Shu framework of maximum-principle and positivity-preserving high-order schemes \cite{zhang2010maximum,zhang2010positivity,zhang2011b,shu2016high} provides a mature methodology for enforcing admissibility once a conservative high-order update has been cast into an appropriate convex form, but the hybrid setting introduces a new difficulty. Cell averages and point values are advanced as different types of unknowns while remaining coupled through the same reconstruction and flux evaluation. The cell average is a conservative quantity, whereas the skeleton point values evolve through a pointwise, generally non-conservative operator. Hence, admissibility cannot be reduced to a single convex update or to a standard finite-volume argument: one must control two coupled representations of the solution without sacrificing either conservation or the point-value evolution. This tension between continuous traces, conservative fluxes, and uncontrolled internal reconstruction values is the source of the continuous-trace obstruction characterized in \Cref{sec:avg_idp}.

Several stabilization strategies for AF and PAMPA have been developed either \emph{a posteriori} or for particular models and mesh families. Examples include monolithic convex limiting \cite{kuzmin2020monolithic,kuzmin2022bound,bolm2026invariant,abgrall2024Polygonal}, convex blending techniques closely related to monolithic convex limiting \cite{kuzmin2020monolithic,abgrall2025bound,duan2025active}, and nonlinear bound-preserving limiters for Cartesian Active Flux discretizations \cite{chudzik2021cartesian}. A systematic presentation of these property-preserving numerical frameworks can be found in \cite{kuzmin2012flux,kuzmin2024property}.  See also \cite{abgrall2026robust} for a robust triangular-mesh DG-formulation PAMPA scheme with bound preservation, oscillation elimination, and boundary treatments. These techniques are computationally effective. From the viewpoint of the conservative cell-average update, convex limiting or blending can be regarded as an alternative way of introducing an IDP numerical-flux contribution: when the limiter is active, the final conservative flux is no longer the single-state continuous physical flux, even if the high-order trace variable itself is continuous. Our goal is complementary: to provide an \emph{a priori} mechanism that identifies, prior to each hybrid update, which local quantities must be controlled so that conservative cell averages and non-conservative point values remain admissible under distinct but compatible arguments.

Motivated by the one-dimensional PAMPA formulation in \cite{abgrall2025novel}, 
we reveal a structural obstruction in the multidimensional case. We prove that admissible continuous boundary traces, used directly through the single-state physical flux while internal reconstruction values remain uncontrolled, do not by themselves provide a conservative IDP guarantee for any prescribed CFL number. In this sense, boundary continuity provides the correct trace compatibility but not the local one-sided control, or numerical dissipation, required to stabilize the cell-average update. This motivates the use of an additional IDP flux mechanism in the conservative update; in this paper, we realize it via locally admissible, generally discontinuous trace inputs together with a Riemann-type numerical flux.

The resulting strategy addresses the two components of the hybrid update via separate admissibility mechanisms and then couples them. The non-conservative point evolution is handled in a mapped space, so that transformed variables are never evaluated outside the admissible set. The conservative cell-average update is controlled through cell average decomposition (CAD), geometric quasilinearization (GQL)---originating from the positivity-preserving or bound-preserving schemes for compressible MHD in \cite{wu2017admissible,wu2018positivity,Wu2023Geometric}---and local \emph{a priori} scaling, which is adapted from \cite{zhang2010maximum,zhang2010positivity}. This scaling is conservative, cellwise, and applied prior to numerical flux evaluation. It introduces no additional evolved degrees of freedom, requires no global nonlinear solves or mesh-wide corrections, and does not rely on postprocessing limiting. It remains inactive on cells whose trace and CAD control states already satisfy the admissibility constraints.
The resulting CFL condition depends only on local CAD weights and trace-state wave-speed bounds, making the admissibility requirements explicit rather than hidden in implementation parameters.

The main contributions of this work are summarized as follows:
\begin{itemize}
    \item We formulate an \emph{a priori} admissibility principle for hybrid point-value/cell-average schemes, in which conservative cell averages and point states are controlled by distinct but compatible mechanisms in a coupled update. For a forward Euler stage, an explicit CFL condition ensures preservation of both components of the hybrid admissible set. Invariant domain preservation under strong-stability-preserving (SSP) Runge--Kutta methods then follows via convex-stage embedding; see \Cref{thm:hybrid_stage_IDP} and \Cref{cor:ssp_hybrid_stage}.

    \item We construct a model-independent barrier--Legendre map for convex admissible interiors defined by finitely many concave constraints. Its inverse is characterized by a strictly convex minimization problem, is globally defined and Lipschitz continuous on the transformed space; see \Cref{thm:Psi_constructible} in Appendix~\ref{sec:proof_Psi_constructible}.

    \item We formulate a multidimensional point-value evolution in transformed variables. Directional derivatives are evaluated using admissible current point states, without probing the transform outside the admissible set. We establish consistency for smooth solutions and mapped-back point admissibility for both the exact inverse and practical admissibility-preserving maps; see \Cref{sec:pt_uncond} and \Cref{cor:hybrid_stage_practical}.

    \item We demonstrate that, in the multidimensional hybrid reconstruction setting, continuous traces used through the single-state physical flux alone cannot ensure conservative invariant-domain updates under a uniform CFL condition; see \Cref{thm:necessity_discont}. This obstruction already manifests for scalar linear advection. Our realization of the required IDP flux mechanism employs numerical fluxes with local discontinuous admissible trace states produced by CAD-compatible \emph{a priori} scaling; its relation to convex limiting and flux blending is clarified in \Cref{rem:convex_limiting_trace}. 
The scaling preserves the cell average, acts as the identity on already admissible CAD data, and yields an explicit CFL condition whose constants depend only on local CAD weights and trace-state wave-speed bounds; see \Cref{thm:B}. This yields constructive third-order schemes on triangular, Cartesian, convex quadrilateral, and general convex polygonal meshes; see \Cref{sec:geom}.
\end{itemize}

The methodology developed in this paper decouples the admissibility requirements for point values and cell averages, and then combines them into a unified hybrid IDP framework. It also identifies the limitation of continuous traces in the conservative update and provides a local admissible-trace construction with explicit CFL constants. The remainder of the paper is organized as follows. Section~\ref{sec:setting} presents the abstract hybrid framework and the main admissibility results. Section~\ref{sec:pt_uncond} develops the transformed point-value operator. Section~\ref{sec:avg_idp} investigates the conservative cell-average update and the continuous-trace obstruction. Section~\ref{sec:geom} details constructive CAD realizations. Section~\ref{sec:numerics} presents numerical benchmarks, and the appendices contain the constructive transform theorem, implementation details, and several technical proofs.

\section{Framework and main results}
\label{sec:setting}

This section outlines the mathematical setting and presents the main hybrid IDP theorems. We first specify the model-dependent ingredients: the invariant domain, the one-dimensional invariant-domain flux condition, and the wave-speed bound. We then state the mesh-dependent CAD and GQL assumptions, introduce the point transform and the conservative update, and combine them in a coupled hybrid theorem for a single forward Euler stage, which forms the building block of strong-stability-preserving (SSP) Runge--Kutta schemes. Throughout this section, $\Gint$ denotes the admissible interior used for exact point transforms, while the closed set $G=\overline{\Gint}$ is used for conservative cell-average admissibility.

We consider the $d$-dimensional system of hyperbolic conservation laws
\begin{equation}\label{eq:cl}
	\partial_t u + \nabla\cdot F(u)=0, \qquad (x,t)\in \Omega\times \mathbb{R}_+,
\end{equation}
where $t$ denotes time, $x := (x_1,\ldots,x_d)^\top$ is the coordinate vector, 
$\Omega \subseteq \mathbb{R}^d$ is the spatial domain, $u:\Omega\times\mathbb{R}_+\to\mathbb{R}^m$ denotes the vector of conservative variables, and $F(u)=(f_1(u),\ldots,f_d(u))^\top\in(\mathbb{R}^m)^d$ represents the flux function with $\nabla\cdot F(u):=\sum_{i=1}^d\partial_{x_i} f_i(u)$.
Let $A_i(u) \coloneqq Df_i(u) \in \mathbb{R}^{m \times m}$ denote the Jacobian matrices of the flux components $f_i(u)$.
Throughout the paper, we focus on $d=2$ for geometric constructions on convex polygonal meshes, while the theoretical framework remains dimension-independent.

A fundamental requirement for physical realizability and numerical stability of the nonlinear hyperbolic system \eqref{eq:cl} is the preservation of an admissible set. Two representative benchmark models are considered throughout this paper and introduced here in the open/closed notation used later. For a scalar conservation law,
\begin{equation}\label{G:Scalar}
	\Gint_{\mathrm{sc}}=(U_{\min},U_{\max}),\qquad G_{\mathrm{sc}}=\overline{\Gint_{\mathrm{sc}}}=[U_{\min},U_{\max}],
\end{equation}
where $U_{\min}=\min_x u_0(x)$ and $U_{\max}=\max_x u_0(x)$, see \cite{zhang2010maximum}. For the compressible Euler equations, with conservative variables $u=(\rho,\rho v,E)^\top$, velocity $v\in \mathbb{R}^d$, and pressure $p=(\gamma-1)\bigl(E-\frac{1}{2}\rho|v|^2\bigr)$,
\begin{equation}\label{G:Euler}
	\Gint_{\mathrm{Euler}}=\left\{u=(\rho,\rho v,E)^\top:\ \rho(u)>0,\ \mathcal E(u):=E-\frac{1}{2}\rho|v|^2>0\right\},\quad G_{\mathrm{Euler}}=\overline{\Gint_{\mathrm{Euler}}},
\end{equation}
where $\Gint_{\mathrm{Euler}}$ is convex because $\mathcal E$ is concave in $u$ on $\{\rho>0\}$, see \cite{zhang2010positivity}. In the general framework below, these two model cases are denoted uniformly by $\Gint$ and $G$.

We consider hybrid schemes that evolve \emph{cell averages} and \emph{cell-boundary point values} through two coupled operators: a conservative flux-divergence update for cell averages and a non-conservative evolution operator for point values.

\subsection{Geometry and degrees of freedom}
\label{sec:geom_dofs}
Let $\mathcal{T}_h$ be a conforming partition of $\Omega\subset\mathbb{R}^2$ into open, non-overlapping \emph{convex} polygonal cells $K$.
We denote by $\mathcal{E}_h$ the set of all edges in the mesh $\mathcal{T}_h$.
We denote by $|K|$ the area and by $h_K$ the diameter of $K$, and set
$h:=\max_{K\in\mathcal{T}_h} h_K$.
Let $\mathcal{E}_K$ be the set of edges bounding $K$. For each $e\in\mathcal{E}_K$, denote by $|e|$ the length of $e$ and by $n_{K,e}$ the outward unit normal to $e$ relative to $K$. For an edge $e$ shared by two adjacent cells (i.e., $e = \partial K \cap \partial K_e$), we always have $n_{K_e,e} = -n_{K,e}$.

Hybrid methods advance two families of unknowns. For each $K\in\mathcal{T}_h$, the cell average $\bar{u}_K(t)\in\mathbb{R}^m$ approximates
\[
\bar{u}_K(t)\approx \frac{1}{|K|}\int_K u(x,t)\,\textrm{d}x.
\]
On the mesh skeleton, a finite set $\Sigma_h$ of evolution nodes is defined, where each node $\sigma\in\Sigma_h$ has coordinate $x_\sigma$ and carries a point value $u_\sigma(t)$ approximating $u(x_\sigma,t)$. For a given cell $K$, let $\Sigma_K:=\{\sigma\in\Sigma_h:\ x_\sigma\in\partial K\}$.

The hybrid state space is
\[
\mathcal{V}_{\mathrm{avg}}
:=\prod_{K\in\mathcal{T}_h}\mathbb{R}^m,\qquad
\mathcal{V}_{\mathrm{pt}}
:=\prod_{\sigma\in\Sigma_h}\mathbb{R}^m,\qquad
\mathcal{V}_h := \mathcal{V}_{\mathrm{avg}}\times\mathcal{V}_{\mathrm{pt}}.
\]
A discrete hybrid state is denoted by $u_h=(\bar{u}_h,u_\Sigma)$ with
$\bar{u}_h=\{\bar{u}_K\}_{K\in\mathcal{T}_h} \in \mathcal{V}_{\mathrm{avg}}$ and $u_\Sigma=\{u_\sigma\}_{\sigma\in\Sigma_h} \in \mathcal{V}_{\mathrm{pt}}$.
Given sets $G_{\mathrm{avg}},G_{\mathrm{pt}}\subset\mathbb{R}^m$, we define the discrete admissible set
\begin{equation}\label{eq:discreteGh_general}
	G_h(G_{\mathrm{avg}},G_{\mathrm{pt}})
	:=\Bigl\{u_h\in\mathcal{V}_h:\ \bar{u}_K\in G_{\mathrm{avg}} ~~ \forall K\in\mathcal{T}_h, ~~~ 
	u_\sigma\in G_{\mathrm{pt}} ~~ \forall\sigma\in\Sigma_h\Bigr\}.
\end{equation}
The subsequent analysis accommodates different admissible domains for the conservative and point-value updates: the conservative update is analyzed on the admissible set enforced by the limiter $\Pi_K$, whereas the point-value update takes values either in $\Gint$ under the exact inverse or in the admissible range of a practical admissibility-preserving map.
A hybrid scheme is called \emph{invariant-domain-preserving (IDP)} if
$u_h^0\in G_h(G_{\mathrm{avg}},G_{\mathrm{pt}}) \Rightarrow u_h^n\in G_h(G_{\mathrm{avg}},G_{\mathrm{pt}})$ for all time levels $n$.

Each cell is equipped with a local reconstruction space $\mathbb{W}_K$ (e.g., $(\mathbb{P}^k(K))^m$ or a richer Active Flux/PAMPA approximation space). 
The reconstruction operator acts on all hybrid degrees of freedom: its domain is the hybrid state space $\mathcal V_h=\mathcal V_{\mathrm{avg}}\times\mathcal V_{\mathrm{pt}}$ introduced above, the same space on which the discrete admissible set is defined in \eqref{eq:discreteGh_general}. 
We denote by
\[
\mathcal{R}_K:\mathcal{V}_h\longrightarrow \mathbb{W}_K,\qquad
u_K := (\mathcal{R}_K u_h)|_K,
\]
a (possibly overdetermined) reconstruction operator satisfying the compatibility
constraints
\begin{subequations}\label{eq:recon_constraints}
	\begin{align}
		\frac{1}{|K|}\int_K u_K(x)\,\textrm{d}x &= \bar{u}_K, \label{eq:recon_avg}\\
		u_K(x_\sigma) &= u_\sigma,\qquad \forall \sigma\in\Sigma_K. \label{eq:recon_pt}
	\end{align}
\end{subequations}
Only the compatibility conditions \eqref{eq:recon_constraints} and the smooth-data accuracy of $\mathcal{R}_K$ are required in the subsequent analysis.

For each edge \(e\in\mathcal E_K\), we fix a positive quadrature rule
\begin{equation}\label{eq:edge_quad}
 \frac{1}{|e|}\int_e \varphi(x)\,\mathrm ds
 \approx \sum_{\nu=1}^{N_q}\omega_{e,\nu}\,\varphi(x_{e,\nu}),
 \qquad \omega_{e,\nu}>0,\quad \sum_{\nu=1}^{N_q}\omega_{e,\nu}=1,
\end{equation}
where \(x_{e,\nu}\in e\) denote the quadrature nodes on edge \(e\),
\(\omega_{e,\nu}\) are the corresponding weights, and \(N_q\) is chosen
such that the quadrature is exact for the trace on \(e\) of every
\(\varphi\in \mathbb W_K\), componentwise. We assume $x_{e,\nu}\in\{x_\sigma\}_{\sigma\in\Sigma_h}$.

\subsection{Invariant domains and the GQL half-space representation}
\label{sec:GQL}

We define the open admissible interior
\begin{equation}\label{eq:G_open_strict}
	\Gint := \Bigl\{u\in\mathbb{R}^m:\ g_j(u)>0,\ \forall j\in \mathcal{J}\Bigr\},
\end{equation}
and its closed counterpart
\begin{equation}\label{eq:G_def_g}
	G := \overline{\Gint}
	= \Bigl\{u\in\mathbb{R}^m:\ g_j(u)\ge 0,\ \forall j\in \mathcal{J}\Bigr\}.
\end{equation}
The point-transform formulation is established on $\Gint$, the conservative and limiter conditions on $G$, and the floating-point relaxations are deferred to Appendix~\ref{sec:IDP_imple}. Each $g_j$ is assumed to be concave and continuously differentiable on an open domain $U_j\subset\mathbb{R}^m$ containing $G$ together with its active boundary $S_j$ defined below.

A key technical challenge in preserving invariant domains is the possible nonlinearity
of the constraints \eqref{eq:G_def_g}, such as the positivity of internal energy in \eqref{G:Euler} for compressible Euler flows.
To address this, we employ the \emph{geometric quasilinearization (GQL)} framework of Wu and Shu \cite{Wu2023Geometric}, which converts the nonlinear admissibility constraints into an equivalent family of linear constraints.

For each constraint $j\in\mathcal{J}$, we set
$G_j:=\{u:\ g_j(u)\ge 0\}$ and denote
\[
S_j := \partial G\cap \partial G_j,
\]
i.e., the portion of $\partial G$ where the $j$-th constraint is active.
For any $u_j^\ast\in S_j$, define the normal
\[
n_j^\ast := \nabla g_j(u_j^\ast).
\]
We assume the non-degeneracy condition $n_j^\ast \neq 0$ (which is naturally satisfied for concave constraints $g_j$ provided $\text{int}(G)\neq \emptyset$), ensuring that each linearized constraint defines a proper supporting half-space.
The GQL half-space representation of $G$ is
\begin{equation}\label{eq:G_star}
	G^\ast
	:=\Bigl\{u\in\mathbb{R}^m:\ (u-u_j^\ast)\cdot n_j^\ast \ge 0,\ 
	\forall u_j^\ast\in S_j,\ \forall j\in\mathcal{J}\Bigr\}.
\end{equation}

\begin{proposition}[GQL representation \cite{Wu2023Geometric}]\label{prop:GQL}
	Assume each $g_j$ in \eqref{eq:G_def_g} is concave, $C^1$ on the open domain $U_j$ containing $G\cup S_j$, and satisfies the non-degeneracy condition on $S_j$. Then the convex set $G$ is equivalent to its GQL representation:
	\[
	G = G^\ast.
	\]
\end{proposition}

For the compressible Euler admissible set \eqref{G:Euler}, the corresponding closed admissible set admits the following half-space representation.

\begin{proposition}[GQL representation for the Euler system \cite{Wu2023Geometric}]
	For the compressible Euler equations, the closed admissible set $G$ corresponding to \eqref{G:Euler} has the GQL representation
	\begin{equation} \label{G*:Euler}
		G^* = \{u = (\rho, \rho v, E)^\top : (u - u_j^*) \cdot n_j^* \ge 0 \quad \forall u_j^* \in {S}_j, \ j=1, 2\}
	\end{equation}
	where the generating sets ${S}_1$ and ${S}_2$ and their associated normals are given as follows. 
	For $j=1$ (the density constraint), one may take ${S}_1=\{u_1^*=\bm 0\in\mathbb R^m\}$ and $n_1^*=e_1$, the first standard basis vector in $\mathbb R^m$.
	For $j=2$ (the pressure constraint), the set is
	\[
	{S}_2 = \left\{ u_2^* = \left(\rho^*, \rho^* v^*, \frac{1}{2}\rho^*|v^*|^2\right)^\top : \rho^* > 0, v^* \in \mathbb{R}^d \right\},
	\]
    with the corresponding normals $n_2^* = \left( \frac{1}{2}|v^*|^2, -v^*, 1 \right)^\top$.
	
	Equivalently, $G^*$ can be explicitly written in the form 
	\begin{equation} \label{G*:Euler2}
		G^* = \{u = (\rho, \rho v, E)^\top : {u} \cdot n_1^* \ge 0, ~  (u - u^*) \cdot {n}^* \ge 0 ~ \forall u^* \in S \}
	\end{equation}
	with $S=S_2$ and $n^* = n_2^*$.
\end{proposition}

\subsection{Admissibility transform}
\label{sec:Psi}
We denote the admissibility transform and its inverse by
\begin{equation}\label{eq:point_update_w}
	\Psi:\Gint\longrightarrow \mathbb{R}^m,\qquad u\longmapsto w=\Psi(u),
	\qquad
	\Psi^{-1}:\mathbb{R}^m\longrightarrow \Gint,\qquad w\longmapsto u=\Psi^{-1}(w).
\end{equation}

\begin{itemize}
	\item (A$\Psi$1) \emph{Global bijection.}
	$\Psi:\Gint\to\mathbb{R}^m$ is bijective, and its inverse $\Psi^{-1}:\mathbb{R}^m\to \Gint$ is globally defined.
	\item (A$\Psi$2) \emph{Regularity.} $\Psi$ is $C^1$ on $\Gint$ and
	$D\Psi(u)$ is invertible for all $u\in \Gint$.
	Moreover, $\Psi$ and $\Psi^{-1}$ are locally Lipschitz continuous.
	\item (A$\Psi$3) \emph{Compatibility with the PDE.} For any smooth solution
	$u(x,t)\in \Gint$ of \eqref{eq:cl}, the transformed variable $w=\Psi(u)$ satisfies the
	equivalent quasilinear non-conservative form
	\begin{equation}\label{eq:w_quasilinear}
		\partial_t w + \sum_{i=1}^d J_i(u)\,\partial_{x_i}w = 0,
		\qquad
		J_i(u):= D\Psi(u)\,A_i(u)\,\bigl(D\Psi(u)\bigr)^{-1}.
	\end{equation}
\end{itemize}
\begin{remark}[Constructive admissibility transform]\label{rem:Psi_constructible_appendix}
For convex admissible interiors defined by finitely many concave constraints, Appendix~\ref{sec:proof_Psi_constructible} constructs a barrier--Legendre admissibility transform satisfying (A$\Psi$1)--(A$\Psi$3), along with the variational characterization of the inverse map and a global Lipschitz bound; see \Cref{thm:Psi_constructible}.
\end{remark}

\begin{remark}[Exact and practical point maps]\label{psi1}
The exact theoretical formulation relies on the barrier--Legendre map of \Cref{thm:Psi_constructible} and its exact inverse $\Psi^{-1}:\mathbb R^m\to \Gint$. In practical implementations, a forward map $\mathcal T:U_{\mathrm{fwd}}\to\mathbb R^m$ is paired with a globally defined admissibility-preserving map $\mathcal B:\mathbb R^m\to G_{\mathrm{pt}}$, where the convex point-value admissible target satisfies $G_{\mathrm{pt}}\subseteq U_{\mathrm{fwd}}$. Explicit pairs for scalar conservation laws and the compressible Euler equations are listed in Appendix~\ref{sec:transform_pairs} and summarized in \Cref{tab:transform-summary}.
\end{remark}

\subsection{Hybrid update operator}
\label{sec:abstract_operator}
The abstract hybrid update stage is formulated componentwise as follows.  The point update is performed in the transformed variables and then mapped back via an admissibility-preserving map $\mathcal B$, whereas the cell-average update remains conservative:
\begin{equation}\label{eq:hybrid_operator}
	\begin{cases}
		u_\sigma^{n+1} = \mathcal{L}_{\mathrm{pt}}(u_h^n)_\sigma, & \forall\sigma\in\Sigma_h,\\[2mm]
		\bar{u}_K^{n+1} = \mathcal{L}_{\mathrm{avg}}(u_h^n)_K, & \forall K\in\mathcal{T}_h,
	\end{cases}
	\quad
	u_h^{n+1}=\mathcal{L}_h(u_h^n):=
	\bigl(\mathcal{L}_{\mathrm{avg}}(u_h^n),\mathcal{L}_{\mathrm{pt}}(u_h^n)\bigr).
\end{equation}
\begin{equation}\label{eq:pt_update_w}
	w_\sigma^{n+1} = w_\sigma^n + \Delta t\,\mathcal{D}^{\mathrm{pt}}_\sigma(u_h^n),
	\qquad u_\sigma^{n+1} := \mathcal B(w_\sigma^{n+1}).
\end{equation}

In \eqref{eq:admissible_recon}, the coefficients $\lambda_{K,e}$ denote the positive boundary weights of a cell average decomposition for the cell $K$. They are introduced abstractly in Assumption~(A3), specifically in \eqref{eq:CAD}--\eqref{eq:edge_ratio}, and their concrete third-order realizations are detailed in Section~\ref{sec:geom}. Thus, $\lambda_{K,e}$ is a fixed CAD boundary weight associated with the edge $e$ of $K$.

A reconstruction $\hat{u}_K^n\in\mathbb{W}_K$ is said to be \emph{$G$-admissible} if
\begin{subequations}\label{eq:admissible_recon}
\begin{equation}\label{eq:915}
	\hat{u}_K^n(x_{e,\nu})\in G,\ \forall e\in\mathcal{E}_K,\ \forall \nu,
\end{equation}
and when $\sum_{e\in\mathcal{E}_K} \lambda_{K,e} < 1$,
\begin{equation}\label{eq:921}
	\hat{u}_K^* := \dfrac{\bar{u}_K^n-\sum_{e\in\mathcal{E}_K}\lambda_{K,e}\bar{\hat u}_{K,e}^n}{1-\sum_{e\in\mathcal{E}_K} \lambda_{K,e}} 
	= 
	\dfrac{\bar{u}_K^n-\sum_{e\in\mathcal{E}_K}\lambda_{K,e}\sum_{\nu=1}^{N_q} \omega_{e,\nu}\hat{u}_K^n(x_{e,\nu})}{1-\sum_{e\in\mathcal{E}_K} \lambda_{K,e}} 
	\in G,
\end{equation}
\end{subequations}
where $\bar{\hat u}_{K,e}^n:=\frac{1}{|e|}\int_e \hat{u}^n_K \, \textrm{d}s$ denotes the average of $\hat{u}^n_K$ over the edge $e$ of cell $K$. These are the same edge weights that appear in the CAD identity \eqref{eq:gcd_point}.

\begin{remark}[Interpretation of the auxiliary CAD state $\hat u_K^*$]
\label{rem:aux_cad_state}
The quantity $\hat u_K^*$ in \eqref{eq:921} is an auxiliary CAD control state. In general, it cannot be interpreted as a point value of the reconstruction on $K$. Such an interpretation is valid only in the special case where the CAD is exact on the full reconstruction space $\mathbb W_K$ and the residual interior contribution in the CAD is represented by a single point evaluation. For example, on a triangular element the identity
\[
\bar\phi_K = \frac16\sum_{e\in\mathcal E_K}\bar\phi_e + \frac12\,\phi(x_K^c),
\qquad \forall\phi\in\mathbb P^2(K),
\]
with exact edge averages and barycenter $x_K^c$, yields $\hat u_K^*=\hat u_K^n(x_K^c)$ when $\mathbb W_K=\mathbb P^2(K)$.

The situation is different when the CAD is required only on $(\mathbb P^k(K))^m$ and $(\mathbb P^k(K))^m\subsetneq\mathbb W_K$. In that case, \eqref{eq:921} defines a linear functional on the available data, not necessarily an evaluation functional on $\mathbb W_K$. To see this, take the reference triangle $K=\{(x,y):x\ge0,\ y\ge0,\ x+y\le1\}$, use the same $\mathbb P^2$ CAD, and choose $\mathbb W_K=\mathbb P^3(K)$. The first moments force any point representing the residual functional to be $x_K^c=(1/3,1/3)$; however, for $\phi(x,y)=x^3$,
\[
2\left(\bar\phi_K-\frac16\sum_{e\in\mathcal E_K}\bar\phi_e\right)=\frac1{30}\neq \frac1{27}=\phi(x_K^c).
\]
Thus, the reduced CAD does not define a point value on the enriched space. This distinction has no adverse impact on the IDP limiter or on the accuracy-preservation argument: the conservative analysis uses only the condition $\hat u_K^*\in G$, and the high-order perturbation estimate for the reduced-CAD alternative limiter is proved in \Cref{lem:PiK_props2}(iii) and summarized in \Cref{cor:limiter_small_smooth}.
\end{remark}

For a shared edge $e = \partial K \cap \partial K_e$, the interior and exterior traces at the quadrature nodes are defined by:
\[
\hat{u}_{K,e,\nu}^n:=\hat{u}_K^n(x_{e,\nu}),\qquad \hat{u}_{K_e,e,\nu}^n:=\hat{u}_{K_e}^n(x_{e,\nu}),\qquad e=K\cap K_e.
\]
The conservative forward Euler update is
\begin{equation}\label{eq:avg_update_FE}
	\bar{u}_K^{n+1}
	= \bar{u}_K^n
	- \frac{\Delta t}{|K|}\sum_{e\in\mathcal{E}_K} |e|\sum_{\nu=1}^{N_q}
	\omega_{e,\nu}\,
	\widehat{F}\bigl(\hat{u}_{K,e,\nu}^n,\hat{u}_{K_e,e,\nu}^n;\,n_{K,e}\bigr),
\end{equation}
where $K_e$ denotes the neighboring cell across $e$ (or represents a boundary state if $e\subset\partial\Omega$).

\subsection{Standing assumptions}
\label{sec:assumptions}

We employ the following modular assumptions. Assumptions (A1)--(A3) concern the mesh and CAD geometry, (A4) is needed only for the exact point-transform result, and (A5) is the conservative flux admissibility condition. Boundary conditions are handled through prescribed boundary data, rather than through a separate structural assumption: they supply the exterior traces and boundary directional derivatives required by the point-value operator. The periodic, inflow and exact-data, outflow, and reflective treatments used in the numerical simulations are detailed in Appendix~\ref{sec:boundary_closure}.

\begin{itemize}
	\item(A1) \emph{Mesh regularity and positive quadrature.}
	The mesh $\mathcal{T}_h$ is conforming and uniformly shape-regular;
	the number of edges per cell is uniformly bounded.
	All edge quadrature rules \eqref{eq:edge_quad} have strictly positive weights.
	
	\item(A2) \emph{Convex invariant domain and GQL.}
	The invariant domain $G\subset\mathbb{R}^m$ is convex and admits the description \eqref{eq:G_def_g} with concave, non-degenerate $C^1$ constraints; hence $G=G^\ast$ holds in the sense of \Cref{prop:GQL}.
	
\item(A3) \emph{Feasible cell average decomposition (CAD).}
	Let $\mathbb{V}_K \in \{(\mathbb{P}^{k}(K))^m, \mathbb{W}_K\}$. For each cell $K$, there exist boundary weights $\{\lambda_{K,e}\}_{e\in\mathcal{E}_K}$, internal weights $\{\beta_{K,s}\}_{s=1}^{S_K}$, and internal nodes $\{x^*_{K,s}\}_{s=1}^{S_K}\subset \overline{K}$ such that
	\begin{equation}\label{eq:CAD}
		\bar{\phi}_K
		:=\frac{1}{|K|}\int_K \phi(x)\,\textrm{d}x
		= \sum_{e\in\mathcal{E}_K}\lambda_{K,e}\,\bar{\phi}_{K,e}
		+ \sum_{s=1}^{S_K}\beta_{K,s}\,\phi(x_{K,s}^\ast),
		\qquad \forall \phi\in \mathbb{V}_K,
	\end{equation}
	where $\bar{\phi}_{K,e}:=\frac{1}{|e|}\int_e \phi\,\textrm{d}s$. The weights satisfy
	\[
	\lambda_{K,e}>0,\quad \beta_{K,s}>0,\qquad
	\sum_{e\in\mathcal{E}_K}\lambda_{K,e}+\sum_{s=1}^{S_K}\beta_{K,s}=1.
	\]
	Assume that the edge quadrature rule \eqref{eq:edge_quad} is exact for traces of
functions in \(\mathbb W_K\) on each edge, with exactness understood
componentwise for vector-valued functions. 
In particular, it is exact for
traces of functions in \(\mathbb V_K\subseteq \mathbb W_K\). 
    Then
	\begin{equation}\label{eq:gcd_point}
		\bar{\phi}_K
		=
		\sum_{e\in\mathcal{E}_K}\lambda_{K,e}\sum_{\nu=1}^{N_q}\omega_{e,\nu}\,\phi(x_{e,\nu})
		+\sum_{s=1}^{S_K}\beta_{K,s}\,\phi(x_{K,s}^*),
		\qquad \forall \phi\in \mathbb{V}_K.
	\end{equation}
We define the edge ratio parameter
\begin{equation}\label{eq:edge_ratio}
	\mu_K := \min_{e\in\mathcal{E}_K} \frac{\lambda_{K,e}\,|K|}{|e|}>0,
	\qquad
	\mu := \min_{K\in\mathcal{T}_h} \mu_K > 0.
\end{equation}

		\item(A4) \emph{Exact point evolution in $w$-space.}
		The mapping $\Psi$ satisfies (A$\Psi$1)--(A$\Psi$3), the point update is performed in the $w$-variables as in \eqref{eq:pt_update_w}, and the point states are recovered through the exact inverse $\Psi^{-1}:\mathbb R^m\to \Gint$.

	\item(A5) \emph{Weak Lax--Friedrichs splitting.}
		Under the GQL framework, the weak half-space formulation of the Lax--Friedrichs splitting principle is assumed to hold; see \cite{Wu2023Geometric}.
		For every state $u\in G$, every constraint $j\in\mathcal{J}$, and every active boundary state $u_j^\ast\in S_j$, there exists a wave-speed bound $\tilde\alpha(u,n)\ge 0$ and a vector function $\zeta(u_j^\ast)\in\mathbb{R}^d$ such that for any unit vector $n\in\mathbb{R}^d$,
	\begin{equation}\label{eq:weakLF2}
		\alpha\,(u-u_j^\ast)\cdot n_j^\ast \pm \big(F(u)\cdot n\big)\cdot n_j^\ast
		\ \ge\ \mp \zeta(u_j^\ast)\cdot n
		\qquad \forall \alpha\ge \tilde\alpha(u,n).
	\end{equation}
\end{itemize}

For boundary edges, the exterior trace states and the directional derivatives required by the point-value update are supplied by the
prescribed boundary conditions. In the analysis presented below, these point-operator boundary treatments are required to be consistent with the physical boundary data and locally Lipschitz continuous with respect to the interior trace and gradient data. The exterior states entering the conservative numerical flux are specified separately by the physical boundary conditions, or by admissible approximations thereof, and are required to belong to \(G\). Appendix~\ref{sec:boundary_closure} describes this decoupling and verifies these properties for the treatments used in the numerical simulations.

\begin{proposition}[Verification of Assumption~(A5) for the scalar invariant domain]\label{prop:A5_scalar}
	For scalar conservation laws with $G_{\mathrm{sc}}=[U_{\min}, U_{\max}]$, Assumption~(A5) holds for all $u \in G_{\mathrm{sc}}$ and unit normals $n$ with
    \begin{equation}\label{eq:Roe_speed}
		\tilde{\alpha}(u, n) = \max\{a(u,n,u_1^*), a(u,n,u_2^*)\}
        \qquad 
        \zeta(u_j^*) = -n_j^* F(u_j^*)
	\end{equation}
    with
    \[
		a(u,n,u_j^*)=
		\begin{cases} 
			\left| \frac{\left(F(u) - F(u_j^*)\right) \cdot n}{u - u_j^*} \right|, & \text{if } u \neq u_j^*, \\[2mm]
			|DF(u_j^*) \cdot n|, & \text{if } u = u_j^*,
		\end{cases}
	\]
	where $u_1^* = U_{\min}$ and $u_2^* = U_{\max}$ are the boundary states of $\partial G_{\mathrm{sc}}$ with state-space inward normals $n_1^* = 1$ and $n_2^* = -1$, respectively. 
\end{proposition}
\begin{proof}
	The proof is provided in Appendix~\ref{sec:weakLF_scalar}.
\end{proof}

By the mean-value theorem, $a(u,n,u_j^*)$ defined in \eqref{eq:Roe_speed} corresponds exactly to the magnitude of the projected flux Jacobian at some intermediate state $u_{j,\xi}$ within the segment between $u$ and $u_j^*$:
\[
a(u,n,u_j^*) = |DF(u_{j,\xi}) \cdot n|= |A_n(u_{j,\xi})| = \bigg|\sum_{i=1}^d n_i A_i(u_{j,\xi})\bigg|, \quad n=(n_1,\cdots,n_d).
\]
Consequently, $\tilde{\alpha}(u,n)=\max\{|A_n(u_{1,\xi})|,|A_n(u_{2,\xi})|\}$ is naturally bounded by the maximum local characteristic speed.

\begin{proposition}[Verification of Assumption~(A5) for the Euler invariant domain]\label{prop:A5_Euler}
	For the compressible Euler equations with $\Gint_{\mathrm{Euler}}$, which represents the admissible set defined by the strict positivity of density and pressure, Assumption~(A5) holds for all interior states $u \in \Gint_{\mathrm{Euler}}$ and unit normals $n$ with
	\begin{equation}\label{eq:Euler_speed}
	\tilde\alpha(u,n)=|v\cdot n|+c_s,\qquad c_s=\sqrt{\gamma p/\rho},
	\qquad \zeta\equiv \bm0.
	\end{equation}
\end{proposition}

\begin{proof}
	A direct half-space verification is provided in Appendix~\ref{sec:weakLF_Euler}.
\end{proof}

	\Cref{prop:A5_Euler} does not cover Tadmor's minimum principle for the specific entropy \cite{Tadmor1986}. If the compressible Euler invariant domain is strengthened by adding a lower bound on the specific entropy, Assumption~(A5) must be re-verified for the enlarged set of constraints. In general, this requires a larger wave-speed estimate $\tilde\alpha(u,n)$ and a modified boundary-correction function $\zeta$ in \eqref{eq:weakLF2}.

\subsection{Main theorems}
\label{sec:main_theorems}

The following main theorems are stated before the detailed proofs to highlight the structure of the hybrid argument. The conservative cell-average theorem provides the IDP CFL condition, while the point update is closed either by the global inverse or by an admissibility-preserving practical map.  The coupled theorem combines these mechanisms to yield the IDP property without post-processing corrections.

\begin{theorem}[IDP property for the hybrid forward Euler update]\label{thm:hybrid_stage_IDP}
Assume that (A1)--(A5) hold and that the prescribed boundary conditions provide the exterior states and boundary directional derivatives required by the point-value operator, as specified in Appendix~\ref{sec:boundary_closure}. Let $u_h^n\in G_h(G,\Gint)$, let $u_K^n=\mathcal R_K(u_h^n)$ be the associated reconstruction, and let $\hat u_K^n=\Pi_K(u_K^n;\bar u_K^n)$ satisfy \eqref{eq:admissible_recon}. Consider one forward Euler stage of the hybrid update defined by \eqref{eq:hybrid_operator} and \eqref{eq:pt_update_w}:
\[
 u_h^{n+1}=\mathcal L_h(u_h^n)=\bigl(\mathcal L_{\mathrm{avg}}(u_h^n),\,\mathcal L_{\mathrm{pt}}(u_h^n)\bigr),
\]
with the point update computed via the exact inverse $\mathcal B=\Psi^{-1}$. If the CFL condition \eqref{eq:CFL_abstract} holds, then
\[
 \bar u_K^{n+1}\in G \quad \forall K\in\mathcal T_h,
 \qquad
 u_\sigma^{n+1}\in \Gint \quad \forall \sigma\in\Sigma_h,
\]
and hence 
$
 u_h^{n+1}\in G_h(G,\Gint).
$ 
\end{theorem}

\begin{theorem}[Cell-average admissibility]\label{thm:B}
Assume that (A1), (A2), (A3), and (A5) hold. Let $\hat{u}_K^n=\Pi_K(u_K^n;\bar{u}_K^n)$ be a $G$-admissible reconstruction satisfying \eqref{eq:admissible_recon}. For the conservative forward Euler update \eqref{eq:avg_update_FE} employing the Lax--Friedrichs flux \eqref{eq:LF_flux}, there exists an explicit CFL condition of the form
\begin{equation}\label{eq:CFL_abstract}
	\Delta t \le \min_{K\in\mathcal{T}_h} \frac{\mu_K}{2\,\alpha_K^n},
	\qquad
	\alpha_K^n := \max_{e\in\mathcal{E}_K}\max_{\nu=1,\ldots,N_q}
	\tilde{\alpha}\bigl(\hat{u}_{K,e,\nu}^n,\,n_{K,e}\bigr),
\end{equation}
such that the updated cell averages satisfy
\[
\bar{u}_K^{n+1} \in G,\qquad \forall K\in\mathcal{T}_h.
\]
In particular, if each boundary edge $e\subset\partial\Omega$ is assigned an exterior state satisfying $\hat{u}_{K_e,e,\nu}^n\in G$, the boundary cell averages also satisfy $\bar{u}_K^{n+1}\in G$ under the same CFL condition \eqref{eq:CFL_abstract}.
\end{theorem}

\begin{remark}[Sufficient nature of the CFL condition]\label{rem:CFL_stage_conditional} 
The CFL condition \eqref{eq:CFL_abstract} is a sufficient criterion for the conservative forward Euler stage. It ensures that each CAD sub-state in \eqref{eq:uK_LF_decomp} remains in $G$, so their convex combination $\bar{u}_K^{n+1}$ also lies in $G$.
\end{remark}

\begin{corollary}[IDP property for SSP Runge--Kutta methods]\label{cor:ssp_hybrid_stage}
Under the hypotheses of \Cref{thm:hybrid_stage_IDP} or \Cref{cor:hybrid_stage_practical}, consider an SSP Runge--Kutta method whose stages are convex combinations of forward Euler steps with admissible substep sizes and with the update structure \eqref{eq:hybrid_operator}--\eqref{eq:pt_update_w}. Then every stage preserves the corresponding hybrid admissible set. In particular, the classical third-order SSP Runge--Kutta method satisfies the same conclusion under the forward-Euler conservative CFL condition.
\end{corollary}

\begin{remark}[Positivity on open interior and vacuum avoidance]\label{rem:positivity_open}
\Cref{thm:B} is also valid for the open admissible interior $\Gint$. For the compressible Euler equations, applying \Cref{thm:B} strictly within $\Gint$ requires the wave-speed bounds $\tilde\alpha$ to remain strictly bounded away from infinity, which amounts to avoiding vacuum states ($\rho=0$ or $p=0$). In practical simulations, vacuum states are avoided by enforcing a numerical cutoff $\rho\ge\varepsilon_\rho>0$ and $p\ge\varepsilon_p>0$ as described in Appendix~\ref{sec:IDP_imple}, which ensures that wave speeds are well defined throughout the domain.
\end{remark}


\section{Point-value update in transformed variables: operator construction and mapped admissibility}\label{sec:pt_uncond}

This section develops the non-conservative point-value update in transformed variables, establishes its invariant-domain-preserving (IDP) property, and verifies its consistency for smooth solutions.

\subsection{Transformed non-conservative formulation}\label{sec:pt_transformed}

Recall the hyperbolic system \eqref{eq:cl} and the admissibility transform $w=\Psi(u)$ introduced in \Cref{sec:Psi}. Under property~(A$\Psi$3), the transformed variable $w$ satisfies the non-conservative system \eqref{eq:w_quasilinear} for smooth solutions taking values in $\Gint$. 
For any direction $q:=(q_1,\ldots, q_d)\in\mathbb{R}^d$, we have
\begin{equation}\label{eq:Aq}
	J_q(u) := \sum_{i=1}^d q_i\,J_i(u),
	\quad 
	A_q(u):=\sum_{i=1}^d q_i\,A_i(u),
	\quad
	J_q(u)=D\Psi(u)\,A_q(u)\,\bigl(D\Psi(u)\bigr)^{-1}.
\end{equation}
In two space dimensions ($d=2$), for any orthonormal basis $(n,\tau)$, we have
\[
\sum_{i=1}^2 J_i(u)\,\partial_{x_i} w
= J_n(u)\,\partial_n w + J_\tau(u)\,\partial_\tau w.
\]
This directional decomposition serves as the basis for the geometry-consistent multidimensional upwind splitting developed below.

\subsection{Multidimensional upwind splitting in transformed variables}\label{sec:pt_gmdu}

We consider $d=2$, convex polygonal cells (see \Cref{fig:Mesh1}), and the third-order discretization $k=2$.
\subsubsection{Local geometry at a mesh skeleton point}\label{sec:pt_geometry_sigma}

Let $\sigma\in\Sigma_h$ be a mesh skeleton point with coordinates $x_\sigma$.
Denote by $\mathcal{E}_\sigma$ the set of edges incident to $x_\sigma$.
For each pair $(\sigma,e)$ with $e\in\mathcal{E}_\sigma$, we specify:
\begin{itemize}
	\item a unit tangent $\tau_{\sigma,e}$ along $e$ oriented \emph{away} from $x_\sigma$;
	\item a unit normal $n_{\sigma,e}$ such that $\{n_{\sigma,e},\tau_{\sigma,e}\}$
	is a right-handed orthonormal basis.
\end{itemize}
For an interior edge $e$ shared by two cells, there exist unique cells
$K^{-}_{\sigma,e}$ and $K^{+}_{\sigma,e}$ such that $n_{\sigma,e}$ points from
$K^{-}_{\sigma,e}$ to $K^{+}_{\sigma,e}$, i.e.,
\[
n_{\sigma,e} = n_{K^{-}_{\sigma,e},e} = -\,n_{K^{+}_{\sigma,e},e}.
\]
	For a boundary edge $e\subset\partial\Omega$, the prescribed boundary conditions provide the exterior trace and the required directional derivatives; explicit periodic, inflow, exact-data, outflow, and reflective formulas are detailed in Appendix~\ref{sec:boundary_closure}.

For any cell $K$ and any vertex $\sigma\in\partial K$, let $\theta_{K,\sigma}\in(0,\pi)$
be the interior angle of $K$ at $\sigma$ (convexity ensures positivity).
For any vertex $\sigma$ and any incident edge $e\in\mathcal{E}_\sigma$, define
\[
\theta_{\sigma,e} :=
\begin{cases}
	\frac12\bigl(\theta_{K^-_{\sigma,e},\sigma}+\theta_{K^+_{\sigma,e},\sigma}\bigr), & \text{if $e$ is an interior edge},\\[0.4em]
	\frac12\theta_{K^-_{\sigma,e},\sigma}, & \text{if $e\subset\partial\Omega$},
\end{cases}
\]
and 
\[
\Theta_\sigma := \sum_{e\in\mathcal{E}_\sigma}\theta_{\sigma,e},
\qquad
\omega_{\sigma,e} := \frac{\theta_{\sigma,e}}{\Theta_\sigma}.
\]
By construction, $\omega_{\sigma,e}\ge 0$ and $\sum_{e\in\mathcal{E}_\sigma}\omega_{\sigma,e}=1$ for every vertex $\sigma$.

\begin{remark}[Mid-edge points]\label{rem:pt_midpoint_weight}
	If $\sigma$ is the midpoint of an interior edge $e$, then $\mathcal{E}_\sigma=\{e\}$ and we set $\omega_{\sigma,e}=1$. For the third-order discretizations employed here, the edge trace and its tangential derivative are independent of which adjacent cell is selected on $e$, so no auxiliary cell is required in \eqref{eq:pt_midpoint_scheme}.
\end{remark}

\begin{figure}[!htb]
	\centering
	\begin{subfigure}{0.48\textwidth}
		\begin{center}
			\begin{tikzpicture}[scale=1.6]
				\draw [ultra thick] (0,0) -- (1.5,0.2);
				\draw [ thick] (1.5,0.2) -- (0.4,1.5) -- (0,0);
				\draw [ thick] (0,0) -- (-1.2,0.8) -- (0.4,1.5);
				\draw [ thick] (0,0) -- (-1.0,-0.9) -- (-1.2,0.8);
				\draw [ thick] (0,0) -- (0.5,-1.3) -- (-1.0,-0.9);
				\draw [ thick] (1.5,0.2) -- (0.5,-1.3);

				\filldraw [blue] (0,0) circle (1pt);
				\filldraw [blue] (0.75,0.1) circle (1pt);	
				\filldraw [blue] (1.5,0.2) circle (1pt);	
				\filldraw [blue] (0.4,1.5) circle (1pt); %
				\filldraw [blue] (0.95,0.85) circle (1pt);
				\filldraw [blue] (0.2,0.75) circle (1pt); 
				\filldraw [blue] (-1.2,0.8) circle (1pt); %
				\filldraw [blue] (-0.6,0.4) circle (1pt);
				\filldraw [blue] (-0.4,1.15) circle (1pt);
				\filldraw [blue] (-1.0,-0.9) circle (1pt); %
				\filldraw [blue] (-0.5,-0.45) circle (1pt); 
				\filldraw [blue] (-1.1,-0.05) circle (1pt); 
				\filldraw [blue] (0.5,-1.3) circle (1pt); %
				\filldraw [blue] (0.25,-0.65) circle (1pt); 
				\filldraw [blue] (-0.25,-1.1) circle (1pt); 
				\filldraw [blue] (1,-0.55) circle (1pt); 
				
				\draw[ultra thick, densely dotted, black] (0,0) -- (0.42,0.36);
				\draw[ultra thick, densely dotted, black] (0,0) -- (0.48,-0.276);
				\draw[ultra thick, black, -] (0.18,0) arc (0:45:0.18);
				\draw[ultra thick, black, -] (0.18,0) arc (0:-35:0.18);
				\draw [black] (0.42,0.36) node[right] {$\theta_2$};
				\draw [black] (0.48,-0.276) node[right] {$\theta_1$};
				
				\draw [black] (1.1,0.14) node[below] {$e$};			
				\draw [black] (0.5,1.2) node[below] {$K^-_{\sigma,e}$};
				\draw [black] (0.55,-0.9) node[above] {$K^+_{\sigma,e}$};

				\draw [black] (0.1,0.2) node[left] {$x_\sigma$};
				
				\draw [black] (-0.8,-0.2) node[above] {$K^{\mathrm{ext}}_{\sigma,e}$};
				
				\draw[ultra thick, black, ->] (0.85,0.25) -- (1.225,0.3);
				\draw [black] (1,0.32) node[above, rotate=5] {$\tau_{\sigma,e}$};
			\end{tikzpicture}
		\end{center}
		\caption{Triangular mesh.}
		\label{fig:Mesh1_a}
	\end{subfigure} 
	\quad
	\begin{subfigure}{0.48\textwidth}
		\begin{center}
			\raisebox{0.5cm}{
				\begin{tikzpicture}[scale=1.6]
					\draw [ultra thick] (0,0) -- (1.2,0);
					\draw [ thick] (1.2,0) -- (1.5,0.8) -- (0,1) -- (0,0);
					\draw [ thick] (0,0) -- (-1.2,-0.2) -- (-1,0.9) -- (0,1);
					\draw [ thick] (0,0) -- (0,-1) -- (-1,-1.2) -- (-1.2,-0.2);
					\draw [ thick] (0,-1) -- (1.2,-1.2) -- (1.2,0);
					
					\filldraw [blue] (0,0) circle (1pt);
					\filldraw [blue] (1.2,0) circle (1pt);	
					\filldraw [blue] (0.6,0) circle (1pt);	
					\filldraw [blue] (-1.2,-0.2) circle (1pt);	
					\filldraw [blue] (-0.6,-0.1) circle (1pt);	
					\filldraw [blue] (0,0.5) circle (1pt);
					\filldraw [blue] (1.35,0.4) circle (1pt);
					\filldraw [blue] (-1.1,0.35) circle (1pt);
					\filldraw [blue] (1.5,0.8) circle (1pt);
					\filldraw [blue] (0.75,0.9) circle (1pt);
					\filldraw [blue] (0,1) circle (1pt);
					\filldraw [blue] (-0.5,0.95) circle (1pt);
					\filldraw [blue] (-1,0.9) circle (1pt);
					\filldraw [blue] (0,-0.5) circle (1pt);
					\filldraw [blue] (1.2,-0.5) circle (1pt);
					\filldraw [blue] (-1.1,-0.7) circle (1pt);
					\filldraw [blue] (1.2,-1.2) circle (1pt);
					\filldraw [blue] (0.6,-1.1) circle (1pt);
					\filldraw [blue] (0,-1) circle (1pt);
					\filldraw [blue] (-0.5,-1.1) circle (1pt);
					\filldraw [blue] (-1,-1.2) circle (1pt);

					\draw[ultra thick, densely dotted, black] (0,0) -- (0.4, 0.4);
					\draw[ultra thick, densely dotted, black] (0,0) -- (0.4,-0.4);
					\draw[ultra thick, black, -] (0.15,0) arc (0:45:0.15);
					\draw[ultra thick, black, -] (0.15,0) arc (0:-45:0.15);
					\draw [black] (0.4, 0.4) node[right] {$\theta_2$};
					\draw [black] (0.4, -0.4) node[right] {$\theta_1$};
					
					\draw [black] (0.9,0) node[below] {$e$};			
					\draw [black] (0.3,0.9) node[below] {$K^-_{\sigma,e}$};
					\draw [black] (0.3,-1) node[above] {$K^+_{\sigma,e}$};
					
					\draw [black] (0,0.1) node[left] {$x_\sigma$};
					
					\draw [black] (-0.7,0.5) node[above] {$K^{\mathrm{ext}}_{\sigma,e}$};
					
					\draw[ultra thick, black, ->] (0.75,0.15) -- (1.1,0.15);
					\draw [black] (1,0.2) node[above] {${\bm \tau}_{\sigma,e}$};
				\end{tikzpicture}
			}
		\end{center}
		\caption{Quadrilateral mesh.}
		\label{fig:Mesh1_b}
	\end{subfigure} 
	
	\begin{subfigure}{0.48\textwidth}
		\begin{center}
			\begin{tikzpicture}[scale=1.6]
				\coordinate (P1) at (0,0);
				\coordinate (P2) at (0.8,-0.4);
				\coordinate (P3) at (1.4,0.2);
				\coordinate (P4) at (1,1);
				\coordinate (P5) at (0,0.8);
				\coordinate (P6) at (-1.0,1);
				\coordinate (P7) at (-1.4,0.2);
				\coordinate (P8) at (-0.8,-0.4);
				\coordinate (P9) at (-0.5,-1.2);
				\coordinate (P10) at (0.5,-1.2);
				\coordinate (P1_2) at ($(P1)!0.5!(P2)$);
				\coordinate (P2_3) at ($(P2)!0.5!(P3)$);
				\coordinate (P3_4) at ($(P3)!0.5!(P4)$);
				\coordinate (P4_5) at ($(P4)!0.5!(P5)$);
				\coordinate (P1_5) at ($(P1)!0.5!(P5)$);
				\coordinate (P5_6) at ($(P5)!0.5!(P6)$);
				\coordinate (P6_7) at ($(P6)!0.5!(P7)$);
				\coordinate (P7_8) at ($(P7)!0.5!(P8)$);
				\coordinate (P1_8) at ($(P1)!0.5!(P8)$);
				\coordinate (P8_9) at ($(P8)!0.5!(P9)$);
				\coordinate (P9_10) at ($(P9)!0.5!(P10)$);
				\coordinate (P10_2) at ($(P10)!0.5!(P2)$);
				
				\draw [ultra thick]  (P1) -- (P2);
				\draw [ thick]  (P2) -- (P3) -- (P4) -- (P5) -- (P1);
				\draw [ thick]  (P5) -- (P6) -- (P7) -- (P8) -- (P1);
				\draw [ thick]  (P8) -- (P9) -- (P10) -- (P2);
				
				\filldraw [blue] (P1) circle (1pt);
				\filldraw [blue] (P2) circle (1pt);
				\filldraw [blue] (P3) circle (1pt);
				\filldraw [blue] (P4) circle (1pt);
				\filldraw [blue] (P5) circle (1pt);
				\filldraw [blue] (P6) circle (1pt);
				\filldraw [blue] (P7) circle (1pt);
				\filldraw [blue] (P8) circle (1pt);
				\filldraw [blue] (P9) circle (1pt);
				\filldraw [blue] (P10) circle (1pt);
				\filldraw [blue] (P1_2) circle (1pt);
				\filldraw [blue] (P2_3) circle (1pt);
				\filldraw [blue] (P3_4) circle (1pt);
				\filldraw [blue] (P4_5) circle (1pt);
				\filldraw [blue] (P1_5) circle (1pt);
				\filldraw [blue] (P5_6) circle (1pt);
				\filldraw [blue] (P6_7) circle (1pt);
				\filldraw [blue] (P7_8) circle (1pt);
				\filldraw [blue] (P1_8) circle (1pt);
				\filldraw [blue] (P8_9) circle (1pt);
				\filldraw [blue] (P9_10) circle (1pt);
				\filldraw [blue] (P10_2) circle (1pt);

				\draw[ultra thick, densely dotted, black] (0,0) -- (0.6,0.28);
				\draw[ultra thick, densely dotted, black] (0,0) -- (0,-0.6);
				\draw[ultra thick, black, -] (0.18,0) arc (0:30:0.18);
				\draw[ultra thick, black, -] (0.18,0) arc (0:-95:0.18);
				\draw [black] (0.6,0.3) node[right] {$\theta_2$};
				\draw [black] (0,-0.6) node[right] {$\theta_1$};
				
				\draw [black] (0.5,-0.3) node[below] {$e$};			
				\draw [black] (0.8,0.9) node[below] {$K^-_{\sigma,e}$};
				\draw [black] (0.3,-1.2) node[above] {$K^+_{\sigma,e}$};
				
				\draw [black] (0,0.1) node[left] {$x_\sigma$};
				
				\draw [black] (-0.8,0.5) node[above] {$K^{\mathrm{ext}}_{\sigma,e}$};
				
				\draw[ultra thick, black, ->] (0.55,-0.1) -- (0.85,-0.25);
				\draw [black] (0.8,-0.17) node[above, rotate=-10] {${\bm \tau}_{\sigma,e}$};
			\end{tikzpicture}
		\end{center}
		\caption{Pentagonal mesh.}
		\label{fig:Mesh1_c}
	\end{subfigure} 
	\quad
	\begin{subfigure}{0.48\textwidth}
		\begin{center}
			\begin{tikzpicture}[scale=1.6]
				\coordinate (P1) at (0,0);
				\coordinate (P2) at (0.8,0);
				\coordinate (P3) at (1.2,0.6);
				\coordinate (P4) at (0.8,1.2);
				\coordinate (P5) at (0,1.2);
				\coordinate (P6) at (-0.4,0.6);
				\coordinate (P7) at (-1.1,0.7);
				\coordinate (P8) at (-1.5,0.2);
				\coordinate (P9) at (-1.2,-0.4);
				\coordinate (P10) at (-0.5,-0.5);
				\coordinate (P11) at (-0.1,-1.2);
				\coordinate (P12) at (0.7,-1.3);
				\coordinate (P13) at (1.2,-0.7);
				\coordinate (P1_2) at ($(P1)!0.5!(P2)$);
				\coordinate (P2_3) at ($(P2)!0.5!(P3)$);
				\coordinate (P3_4) at ($(P3)!0.5!(P4)$);
				\coordinate (P4_5) at ($(P4)!0.5!(P5)$);
				\coordinate (P5_6) at ($(P5)!0.5!(P6)$);
				\coordinate (P1_6) at ($(P1)!0.5!(P6)$);
				\coordinate (P6_7) at ($(P6)!0.5!(P7)$);
				\coordinate (P7_8) at ($(P7)!0.5!(P8)$);
				\coordinate (P8_9) at ($(P8)!0.5!(P9)$);
				\coordinate (P9_10) at ($(P9)!0.5!(P10)$);
				\coordinate (P1_10) at ($(P1)!0.5!(P10)$);
				\coordinate (P10_11) at ($(P10)!0.5!(P11)$);
				\coordinate (P11_12) at ($(P11)!0.5!(P12)$);
				\coordinate (P12_13) at ($(P12)!0.5!(P13)$);
				\coordinate (P13_2) at ($(P13)!0.5!(P2)$);
				
				\draw [ultra thick]  (P1) -- (P2);
				\draw [ thick]  (P2) -- (P3) -- (P4) -- (P5) -- (P6) -- (P1);
				\draw [ thick]  (P6) -- (P7) -- (P8) -- (P9) -- (P10) -- (P1);
				\draw [ thick]  (P10) -- (P11) -- (P12) -- (P13) -- (P2);
				
				\filldraw [blue] (P1) circle (1pt);
				\filldraw [blue] (P2) circle (1pt);
				\filldraw [blue] (P3) circle (1pt);
				\filldraw [blue] (P4) circle (1pt);
				\filldraw [blue] (P5) circle (1pt);
				\filldraw [blue] (P6) circle (1pt);
				\filldraw [blue] (P7) circle (1pt);
				\filldraw [blue] (P8) circle (1pt);
				\filldraw [blue] (P9) circle (1pt);
				\filldraw [blue] (P10) circle (1pt);
				\filldraw [blue] (P1_2) circle (1pt);
				\filldraw [blue] (P2_3) circle (1pt);
				\filldraw [blue] (P3_4) circle (1pt);
				\filldraw [blue] (P4_5) circle (1pt);
				\filldraw [blue] (P5_6) circle (1pt);
				\filldraw [blue] (P1_6) circle (1pt);
				\filldraw [blue] (P6_7) circle (1pt);
				\filldraw [blue] (P7_8) circle (1pt);
				\filldraw [blue] (P8_9) circle (1pt);
				\filldraw [blue] (P9_10) circle (1pt);
				\filldraw [blue] (P1_10) circle (1pt);
				\filldraw [blue] (P10_11) circle (1pt);
				\filldraw [blue] (P11_12) circle (1pt);
				\filldraw [blue] (P12_13) circle (1pt);
				\filldraw [blue] (P13_2) circle (1pt);

				\draw[ultra thick, densely dotted, black] (0,0) -- (0.3,0.7);
				\draw[ultra thick, densely dotted, black] (0,0) -- (0.2,-0.6);
				\draw[ultra thick, black, -] (0.18,0) arc (0:70:0.18);
				\draw[ultra thick, black, -] (0.18,0) arc (0:-75:0.18);
				\draw [black] (0.3,0.7) node[right] {$\theta_2$};
				\draw [black] (0.2,-0.6) node[right] {$\theta_1$};
				
				\draw [black] (0.6,0) node[below] {$e$};			
				\draw [black] (0.8,1) node[left] {$K^-_{\sigma,e}$};
				\draw [black] (0.5,-1.3) node[above] {$K^+_{\sigma,e}$};
				
				\draw [black] (-0.1,0) node[left] {$x_\sigma$};
				
				\draw [black] (-0.9,0.2) node[above] {$K^{\mathrm{ext}}_{\sigma,e}$};
				
				\draw[ultra thick, black, ->] (0.45,0.15) -- (0.8,0.15);
				\draw [black] (0.7,0.2) node[above] {${\bm \tau}_{\sigma,e}$};
			\end{tikzpicture}
		\end{center}
		\caption{Hexagonal mesh.}
		\label{fig:Mesh1_d}
	\end{subfigure}   
	\caption{Representative local edge geometry on polygonal meshes. Blue dots denote evolution points.}
	\label{fig:Mesh1}
\end{figure}
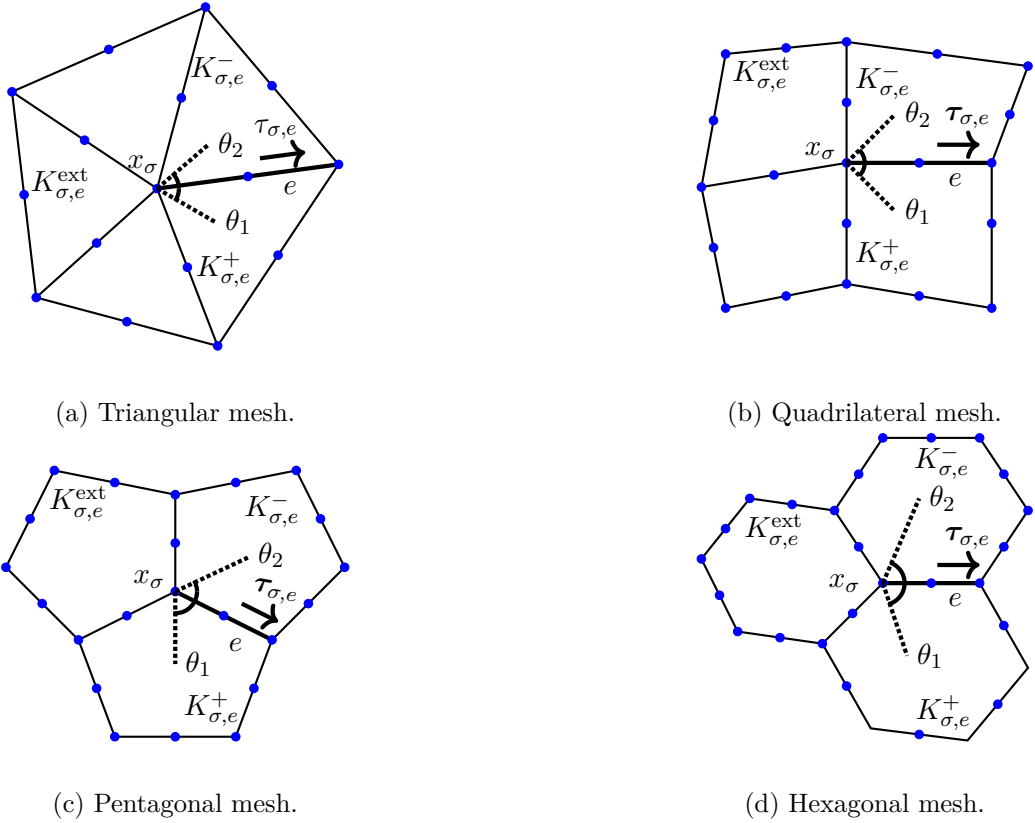

\subsubsection{Lax--Friedrichs splitting matrices in transformed space}\label{sec:pt_lf}

For any $\sigma\in\Sigma_h$ with $u_\sigma\in G$, we choose $\alpha_\sigma>0$ such that
\begin{equation}\label{eq:pt_alpha_sigma}
	\alpha_\sigma \ \ge\
	\max_{e\in\mathcal{E}_\sigma}
	\max\Big\{\varrho \big(J_{n_{\sigma,e}}(u_\sigma)\big),\ \varrho\big(J_{\tau_{\sigma,e}}(u_\sigma)\big)\Big\},
\end{equation}
where $\varrho(A)$ denotes the spectral radius of a matrix $A$. We define the Lax--Friedrichs splitting matrices by
\begin{equation}\label{eq:pt_split}
	J_{n_{\sigma,e}}^{\pm}(u_\sigma):=\frac12\Big(J_{n_{\sigma,e}}(u_\sigma)\pm \alpha_\sigma I\Big),
	\qquad
	J_{\tau_{\sigma,e}}^{\pm}(u_\sigma):=\frac12\Big(J_{\tau_{\sigma,e}}(u_\sigma)\pm \alpha_\sigma I\Big).
\end{equation}

\begin{lemma}[Characteristic sign structure]\label{lem:pt_char_sign}
	Assume local hyperbolicity on $G$. If $\alpha_\sigma\ge \varrho(J_q(u_\sigma))$ for some
	direction $q$, then for any eigendecomposition $J_q(u_\sigma)=R\Lambda R^{-1}$ (where $R$ is the right eigenvector matrix and $\Lambda$ is the diagonal eigenvalue matrix), it holds that
	$J_q^\pm(u_\sigma)=R\Lambda^\pm R^{-1}$ with $\Lambda^\pm=\frac12(\Lambda\pm\alpha_\sigma I)$.
	Consequently,
	$\Lambda^+$ has nonnegative diagonal entries and $\Lambda^-$ has nonpositive diagonal entries.
\end{lemma}

\begin{proof}
	Since $|\lambda_k|\le\alpha_\sigma$ for all eigenvalues $\lambda_k$ of $J_q(u_\sigma)$, it holds that
	$(\lambda_k+\alpha_\sigma)/2\ge 0$ and $(\lambda_k-\alpha_\sigma)/2\le 0$.
\end{proof}

\subsubsection{Directional derivatives without evaluating the transform outside admissible set}\label{sec:pt_derivatives}

High-order reconstructions in $u$ may generate intermediate internal values that
are not necessarily in $G$. To keep the point operator \emph{intrinsic}, we do \emph{not}
reconstruct a polynomial in $w$ by evaluating $\Psi$ at such intermediate values.
Instead, we compute transformed directional derivatives via the chain rule evaluated at the \emph{admissible} point state $u_\sigma$.

Let $u_h\in\mathcal{V}_h$ and let $u_K:=\mathcal{R}_K(u_h)\in\mathbb{W}_K$ be the
reconstruction satisfying \eqref{eq:recon_constraints}. Assume that $u_K$ is differentiable
at $x_\sigma$. In the exact smooth setting, we define the directional derivative of the transformed variable $w$ at $x_\sigma$ along a unit direction $q \in \mathbb{R}^2$ by
\begin{equation}\label{eq:pt_chain_rule}
	\partial_q w_{K,\sigma}
	\ :=\
	D\Psi(u_\sigma)\,\Big(q\cdot \nabla u_K(x_\sigma)\Big).
\end{equation}
For a practical forward map $\mathcal T:U_{\mathrm{fwd}}\to\mathbb R^m$ that is differentiable at the current state $u_\sigma$, the definition carries over by replacing $D\Psi(u_\sigma)$ with $D\mathcal T(u_\sigma)$.

\begin{proposition}[Intrinsic evaluation of transformed directional derivatives]\label{prop:pt_intrinsic_derivative}
Assume either the exact smooth-transform setting with $u_\sigma\in \Gint$, or a practical forward map $\mathcal T:U_{\mathrm{fwd}}\to\mathbb R^m$ that is differentiable at $u_\sigma\in U_{\mathrm{fwd}}$. Then the intrinsic directional derivative defined by \eqref{eq:pt_chain_rule} in the exact case, and by the corresponding formula with $D\mathcal T(u_\sigma)$ in the practical case, is well-defined and depends only on the forward-domain state $u_\sigma$ and the reconstructed gradient $\nabla u_K(x_\sigma)$. In particular, no evaluation of the forward transform or its derivative at intermediate reconstruction values outside the forward domain is required.
\end{proposition}

\begin{proof}
In both cases, the transform enters only through its derivative evaluated at the admissible point state $u_\sigma$. The formula therefore requires only the forward-domain value $u_\sigma$ and the reconstructed gradient $\nabla u_K(x_\sigma)$ at $x_\sigma$, and never probes intermediate reconstruction values.
\end{proof}


\begin{proposition}[Consistency of the intrinsic directional derivative in $w$]\label{prop:A3862}
	Let $K$ be a cell and let $u$ be a smooth solution in a neighborhood of $K$. Let $u_h\in\mathcal V_h$ and let $u_K:=\mathcal R_K(u_h)\in\mathbb W_K$ be the reconstruction employed in \eqref{eq:pt_chain_rule}. Assume that:
    \begin{enumerate}[(i)]
        \item For a fixed integer $p\ge 1$, the reconstruction $u_K$ satisfies
        	\begin{equation}\label{eq:recon_smooth_bounds_prop}
        		\|u_K-u\|_{L^\infty(\bar{K})} \le C\,h^{p+1},
        		\qquad
        		\|\nabla u_K-\nabla u\|_{L^\infty(\bar{K})} \le C\,h^{p},
        	\end{equation}
        	where $C>0$ is independent of the mesh size $h$. 
        \item The exact solution values and the point states under consideration remain in a convex open set $G_\delta$ with \revblue{$\overline{G_\delta}$ compact, $\overline{G_\delta}\subset \Gint$, and $\operatorname{dist}(\overline{G_\delta},\partial G)>0$}, and $\Psi\in C^2(\overline{G_\delta})$.
    \end{enumerate}
    Setting $w:=\Psi(u)$, for any $\sigma \in  \Sigma_K$ and any unit direction $q\in\mathbb R^2$, 
	there exists a constant \revblue{$C'=C'(G_\delta,\Psi,u,C)>0$}, independent of $h$ \revblue{but possibly dependent on the compact interior set $G_\delta$, and hence on its distance from $\partial G$}, such that
	\begin{equation}\label{eq:chain_rule_consistency_estimate}
		\big\|\partial_q w_{K,\sigma} - q\cdot\nabla w(x_\sigma)\big\|_2 \le C'\,h^{p}.
	\end{equation}
\end{proposition}

\begin{proof}
	By the chain rule for the exact smooth solution $w=\Psi(u)$,
	\[
		q\cdot\nabla w(x_\sigma)=D\Psi\bigl(u(x_\sigma)\bigr)\bigl(q\cdot\nabla u(x_\sigma)\bigr).
	\]
	Adding and subtracting $D\Psi\bigl(u(x_\sigma)\bigr)\bigl(q\cdot\nabla u_K(x_\sigma)\bigr)$ yields
	\begin{align*}
		\partial_q w_{K,\sigma} - q\cdot\nabla w(x_\sigma)
		=&
		\Bigl(D\Psi(u_\sigma)-D\Psi\bigl(u(x_\sigma)\bigr)\Bigr)\bigl(q\cdot\nabla u_K(x_\sigma)\bigr)
		+
		D\Psi\bigl(u(x_\sigma)\bigr)\bigl(q\cdot(\nabla u_K-\nabla u)(x_\sigma)\bigr).
	\end{align*}
	The second term on the right-hand side is bounded via \eqref{eq:recon_smooth_bounds_prop}:
	\[
		\Big\|D\Psi\bigl(u(x_\sigma)\bigr)\bigl(q\cdot(\nabla u_K-\nabla u)(x_\sigma)\bigr)\Big\|_2
		\le \|D\Psi\|_{C^0(\overline{G}_\delta)}\,\|\nabla u_K-\nabla u\|_{L^\infty(\bar{K})}
		\le C_2 h^p.
	\]
    For the first term, the convexity of $G_\delta$ and the compactness of $\overline{G_\delta}$ guarantee that $D\Psi$ is Lipschitz continuous on $\overline{G_\delta}$. Thus, there exists a constant \revblue{$L=L(G_\delta,\Psi)>0$}, independent of $h$, such that
	\[
		\|D\Psi(u_\sigma)-D\Psi(u(x_\sigma))\|_2 \le L\,\|u_\sigma-u(x_\sigma)\|_2.
	\]
	By the reconstruction constraint \eqref{eq:recon_pt}, $u_\sigma=u_K(x_\sigma)$ for any $\sigma \in \Sigma_K$, which implies
	\[
		\|u_\sigma-u(x_\sigma)\|_2
		=\|u_K(x_\sigma)-u(x_\sigma)\|_2
		\le \|u_K-u\|_{L^\infty(\bar{K})}
		\le Ch^{p+1}.
	\]
	Since $u$ is smooth, $\|\nabla u\|_{L^\infty(\bar{K})}$ is uniformly bounded, and \eqref{eq:recon_smooth_bounds_prop} implies that $\|q\cdot\nabla u_K\|_{L^\infty(\bar{K})}\le C''$ for sufficiently small $h$. Therefore,
	\[
		\Big\|\bigl(D\Psi(u_\sigma)-D\Psi(u(x_\sigma))\bigr)\bigl(q\cdot\nabla u_K(x_\sigma)\bigr)\Big\|_2
		\le L\,\|u_\sigma-u(x_\sigma)\|_2\,\|q\cdot\nabla u_K\|_{L^\infty(\bar{K})}
		\le C_1 h^{p+1}.
	\]
	Combining the two estimates yields
	\[
		\|\partial_q w_{K,\sigma} - q\cdot\nabla w(x_\sigma)\|_2
		\le C_2 h^p + C_1 h^{p+1}
		\le C' h^p,
	\]
	for sufficiently small $h$, which completes the proof of \eqref{eq:chain_rule_consistency_estimate}.
\end{proof}

\subsubsection{Definition of the semi-discrete point operator}\label{sec:pt_operator}

\smallskip\noindent\textbf{Extension element for tangential upwinding at vertices.}
If $\sigma$ is a \emph{vertex} and $e\in\mathcal{E}_\sigma$, we define the extension point
$x_{\sigma,e}^{\mathrm{ext}}:=x_\sigma-\varepsilon\,\tau_{\sigma,e}$ for a sufficiently small $\varepsilon>0$.
By mesh conformity, there exists a unique cell containing $x_{\sigma,e}^{\mathrm{ext}}$; we denote this cell by
$K^{\mathrm{ext}}_{\sigma,e}$ (see \Cref{fig:Mesh1}).
If $x_{\sigma,e}^{\mathrm{ext}}$ lies on a mesh edge, it belongs to two adjacent cells. For reconstructions whose edge trace is independent of the adjacent cell choice, the tangential derivative $\partial_{\tau} w$ at the vertex depends only on that common edge trace and is therefore identical for both choices. Hence, any deterministic tie-breaking rule yields the same update.

\smallskip\noindent\textbf{Vertex-point operator.}
For a vertex $\sigma$, we define the semi-discrete vertex operator by
\begin{align}\label{eq:pt_rhs_vertex}
	\big(\mathcal{D}^{\mathrm{pt}}(u_h)\big)_\sigma
	:= -\sum_{e\in\mathcal{E}_\sigma}\omega_{\sigma,e}\Big(
	&J_{n_{\sigma,e}}^{+}(u_\sigma)\,\partial_{n_{\sigma,e}} w_{K^{-}_{\sigma,e},\sigma}
	+J_{n_{\sigma,e}}^{-}(u_\sigma)\,\partial_{n_{\sigma,e}} w_{K^{+}_{\sigma,e},\sigma}
	\nonumber\\
	+&J_{\tau_{\sigma,e}}^{-}(u_\sigma)\,\partial_{\tau_{\sigma,e}} w_{K^{-}_{\sigma,e},\sigma}
	+J_{\tau_{\sigma,e}}^{+}(u_\sigma)\,\partial_{\tau_{\sigma,e}} w_{K^{\mathrm{ext}}_{\sigma,e},\sigma}
	\Big),
\end{align}
where the directional derivatives are defined via \eqref{eq:pt_chain_rule} and averaged using the nonnegative angular weights $\omega_{\sigma,e}$ from \Cref{sec:pt_geometry_sigma}.

\smallskip\noindent\textbf{Mid-edge-point operator.}
If $\sigma$ is the midpoint of an interior edge $e$, the operator reduces to
\begin{align}\label{eq:pt_midpoint_scheme}
	\big(\mathcal{D}^{\mathrm{pt}}(u_h)\big)_\sigma
	:= -\Big(
	&J_{n_{\sigma,e}}^{+}(u_\sigma)\,\partial_{n_{\sigma,e}} w_{K^{-}_{\sigma,e},\sigma}
	+J_{n_{\sigma,e}}^{-}(u_\sigma)\,\partial_{n_{\sigma,e}} w_{K^{+}_{\sigma,e},\sigma}
	+J_{\tau_{\sigma,e}}(u_\sigma)\,\partial_{\tau_{\sigma,e}} w_{K^{-}_{\sigma,e},\sigma}
	\Big),
\end{align}
which follows from $J_\tau=J_\tau^++J_\tau^-$ and the tangential continuity along the edge.

\begin{proposition}[Well-definedness and admissibility of point-stage updates]\label{thm:A}
Consider an explicit point stage of the form
\[
w_\sigma^{n+1}=w_\sigma^n+\Delta t\,\big(\mathcal D^{\mathrm{pt}}(u_h^n)\big)_\sigma,
\qquad
u_\sigma^{n+1}=\mathcal B(w_\sigma^{n+1}),
\qquad \sigma\in\Sigma_h,
\]
where $\mathcal B:\mathbb R^m\to G_{\mathrm{pt}}$ is a globally defined admissibility-preserving inverse map with prescribed range $G_{\mathrm{pt}}$.
Assume either
\begin{itemize}
\item[(E)] \revblue{the exact smooth setting (A$\Psi$1)--(A$\Psi$3), with the prescribed boundary conditions supplying the required exterior traces and boundary directional derivatives, $\mathcal B=\Psi^{-1}$, and $G_{\mathrm{pt}}=\Gint$;} or
\item[(P)] a practical forward map and admissibility-preserving map setting in which the directional derivatives entering \eqref{eq:pt_rhs_vertex}--\eqref{eq:pt_midpoint_scheme} are well-defined at the current point states, the intermediate values $w_\sigma^{n+1}$ are finite, and the chosen admissibility-preserving map $\mathcal B$ has its range contained in $G_{\mathrm{pt}}$.
\end{itemize}
Then, in case \emph{(E)}, the point operator \eqref{eq:pt_rhs_vertex}--\eqref{eq:pt_midpoint_scheme} is well-defined at every skeleton point $\sigma\in\Sigma_h$. Moreover, in both cases,
\[
w_\sigma^{n+1}\in\mathbb R^m
\quad\text{and}\quad
u_\sigma^{n+1}\in G_{\mathrm{pt}}
\qquad \forall \sigma\in\Sigma_h.
\]
\end{proposition}

\begin{proof}
In case \emph{(E)}, \Cref{prop:pt_intrinsic_derivative} establishes that each directional derivative entering \eqref{eq:pt_rhs_vertex}--\eqref{eq:pt_midpoint_scheme} is well-defined because it requires only admissible point states, reconstructed gradients, and \revblue{the exterior trace and derivative data provided by the prescribed boundary conditions}. Hence, the explicit point operator $\big(\mathcal{D}^{\mathrm{pt}}(u_h^n)\big)_\sigma$ is a finite vector in $\mathbb R^m$, which implies that $w_\sigma^{n+1}\in\mathbb R^m$. In case \emph{(P)}, finiteness of $w_\sigma^{n+1}$ is assumed directly. In both cases, the conclusion for the mapped-back state follows directly from the range property of the admissibility-preserving map:
\[
u_\sigma^{n+1}=\mathcal B(w_\sigma^{n+1})\in G_{\mathrm{pt}}
\qquad \forall \sigma\in\Sigma_h,
\]
which completes the proof.
\end{proof}

\begin{corollary}[Practical admissibility-preserving maps for the point stage]\label{cor:pt_practical}
Let $\mathcal T:U_{\mathrm{fwd}}\to\mathbb R^m$ be a given forward map and let $\mathcal B:\mathbb R^m\to G_{\mathrm{pt}}$ be a globally defined admissibility-preserving inverse map with convex range $G_{\mathrm{pt}}\subseteq U_{\mathrm{fwd}}$. If all current point states lie in $U_{\mathrm{fwd}}$ and the resulting transformed stage values $w_\sigma^{n+1}$ are finite, the mapped-back point states satisfy $u_\sigma^{n+1}=\mathcal B(w_\sigma^{n+1})\in G_{\mathrm{pt}}$ for all skeleton points $\sigma\in\Sigma_h$. Explicit practical forward--inverse map pairs for scalar conservation laws and compressible Euler equations are detailed in Appendix~\ref{sec:transform_pairs}.
\end{corollary}

\begin{proposition}[Consistency of the transformed point operator on smooth solutions]\label{prop:pt_operator_consistency}
Let $u$ be a smooth solution of \eqref{eq:cl} in a neighborhood of a skeleton point $x_\sigma$, let $w=\Psi(u)$, and assume that $u(x,t)$ remains in a compact subset of $\Gint$ so that all Jacobians $D\Psi$ and $D\Psi^{-1}$ involved are uniformly bounded. Define the hybrid data associated with the exact solution by
\[
\bar u_K=\frac1{|K|}\int_K u(x,t)\,\mathrm dx,
\qquad
u_\sigma=u(x_\sigma,t),
\]
and let each reconstruction entering \eqref{eq:pt_rhs_vertex}--\eqref{eq:pt_midpoint_scheme} satisfy the smooth approximation bounds \eqref{eq:recon_smooth_bounds_prop} at order $p\ge 1$. \revblue{At boundary points, assume in addition that the boundary treatment supplies exterior directional derivatives that are $O(h^p)$ consistent.} Then the semi-discrete point operator is consistent with the transformed PDE in the sense that
\begin{equation}\label{eq:pt_operator_consistency_est}
\bigl\|(\mathcal D^{\mathrm{pt}}(u_h))_\sigma + \sum_{i=1}^2 J_i\bigl(u(x_\sigma,t)\bigr)\,\partial_{x_i}w(x_\sigma,t)\bigr\|_2 \le C h^p,
\end{equation}
where the constant $C>0$ is independent of $h$. Equivalently, since $w$ satisfies \eqref{eq:w_quasilinear}, it holds that 
\[
\bigl\|(\mathcal D^{\mathrm{pt}}(u_h))_\sigma-\partial_t w(x_\sigma,t)\bigr\|_2\le C h^p.
\]
\end{proposition}

\begin{proof}
We distinguish mid-edge points and vertices. If $\sigma$ is the midpoint of an interior edge $e$, the smooth exact solution has a unique gradient at $x_\sigma$, so the underlying exact directional derivatives agree across adjacent elements. By \Cref{prop:A3862},
\[
\partial_{n_{\sigma,e}} w_{K^{\pm}_{\sigma,e},\sigma}=\partial_{n_{\sigma,e}}w(x_\sigma,t)+O(h^p),
\qquad
\partial_{\tau_{\sigma,e}} w_{K^{-}_{\sigma,e},\sigma}=\partial_{\tau_{\sigma,e}}w(x_\sigma,t)+O(h^p).
\]
Substituting these estimates into \eqref{eq:pt_midpoint_scheme} yields
\[
(\mathcal D^{\mathrm{pt}}(u_h))_\sigma
=-J_{n_{\sigma,e}}\bigl(u(x_\sigma,t)\bigr)\,\partial_{n_{\sigma,e}}w(x_\sigma,t)
-J_{\tau_{\sigma,e}}\bigl(u(x_\sigma,t)\bigr)\,\partial_{\tau_{\sigma,e}}w(x_\sigma,t)
+O(h^p).
\]
Since $\{n_{\sigma,e},\tau_{\sigma,e}\}$ is an orthonormal basis, the directional decomposition satisfies
\[
J_{n_{\sigma,e}}\partial_{n_{\sigma,e}}w+J_{\tau_{\sigma,e}}\partial_{\tau_{\sigma,e}}w
=J_1\,\partial_{x_1}w+J_2\,\partial_{x_2}w,
\]
which proves \eqref{eq:pt_operator_consistency_est} for mid-edge points.

If $\sigma$ is a vertex, we write the contribution of an incident edge $e\in\mathcal E_\sigma$ in \eqref{eq:pt_rhs_vertex} as
\[
\begin{aligned}
\mathcal G_{\sigma,e}^h:=\;J_{n_{\sigma,e}}^{+}(u_\sigma)\,\partial_{n_{\sigma,e}} w_{K^{-}_{\sigma,e},\sigma}
+J_{n_{\sigma,e}}^{-}(u_\sigma)\,\partial_{n_{\sigma,e}} w_{K^{+}_{\sigma,e},\sigma}
+J_{\tau_{\sigma,e}}^{-}(u_\sigma)\,\partial_{\tau_{\sigma,e}} w_{K^{-}_{\sigma,e},\sigma}
+J_{\tau_{\sigma,e}}^{+}(u_\sigma)\,\partial_{\tau_{\sigma,e}} w_{K^{\mathrm{ext}}_{\sigma,e},\sigma}.
\end{aligned}
\]
For a smooth exact solution, all exact directional derivatives at $x_\sigma$ coincide across incident cells; at boundary edges the same holds up to $O(h^p)$ by the assumed boundary consistency. Applying \Cref{prop:A3862} to each directional derivative therefore yields
\[
\mathcal G_{\sigma,e}^h
=J_{n_{\sigma,e}}\bigl(u(x_\sigma,t)\bigr)\,\partial_{n_{\sigma,e}}w(x_\sigma,t)
+J_{\tau_{\sigma,e}}\bigl(u(x_\sigma,t)\bigr)\,\partial_{\tau_{\sigma,e}}w(x_\sigma,t)
+O(h^p).
\]
The same orthonormal-basis identity implies
\[
J_{n_{\sigma,e}}\partial_{n_{\sigma,e}}w+J_{\tau_{\sigma,e}}\partial_{\tau_{\sigma,e}}w
=J_1\,\partial_{x_1}w+J_2\,\partial_{x_2}w,
\]
independently of the incident edge $e$. Averaging with the nonnegative angular weights from \Cref{sec:pt_geometry_sigma} and using $\sum_{e\in\mathcal E_\sigma}\omega_{\sigma,e}=1$, it follows that
\[
(\mathcal D^{\mathrm{pt}}(u_h))_\sigma
=-\sum_{e\in\mathcal E_\sigma}\omega_{\sigma,e}\,\mathcal G_{\sigma,e}^h
= -\sum_{i=1}^2 J_i\bigl(u(x_\sigma,t)\bigr)\,\partial_{x_i}w(x_\sigma,t)+O(h^p),
\]
which establishes \eqref{eq:pt_operator_consistency_est}. The equivalent estimate with $\partial_t w$ follows from \eqref{eq:w_quasilinear}, which completes the proof.
\end{proof}

\subsection{\texorpdfstring{Point-side boundary data and admissible flux inputs}{Point-side boundary data and admissible flux inputs}}
\label{sec:pt_interface}

For each cell $K\in\mathcal{T}_h$, we define the prepared reconstruction operator by
\[
\bigl(\widehat{\mathcal{R}}_h u_h\bigr)|_K := (\Pi_K\circ \mathcal{R}_K)(u_h),
\qquad \Pi_K\in\{\Pi_K^\texttt{c},\Pi_K^\texttt{a}\}.
\]

\begin{proposition}[Admissibility of point updates and flux inputs]\label{prop:closed_interface}
	Consider either of the following two settings:
	\begin{itemize}
		\item[(E)] \emph{Exact smooth setting.} \revblue{Assume (A1)--(A5), (A$\Psi$1)--(A$\Psi$3), and boundary conditions of the type specified in Appendix~\ref{sec:boundary_closure};} and let $u_h\in G_h(G,\Gint)$.
		\item[(P)] \emph{Practical forward map and admissibility-preserving map setting.} \revblue{Assume (A1), (A2), (A3), and (A5), together with boundary conditions of the type specified in Appendix~\ref{sec:boundary_closure}.} Let $\mathcal T:U_{\mathrm{fwd}}\to\mathbb R^m$ be a given forward map defined on a point-admissible forward region $U_{\mathrm{fwd}}$, let $\mathcal B:\mathbb R^m\to G_{\mathrm{pt}}$ be a globally defined admissibility-preserving inverse map with convex range $G_{\mathrm{pt}}\subseteq U_{\mathrm{fwd}}$, and let $u_h\in G_h(G,G_{\mathrm{pt}})$. Suppose in addition that all current point states lie in $U_{\mathrm{fwd}}$ and that the transformed stage values produced by the point-update operator are finite.
	\end{itemize}
	Then the following hold:
	\begin{enumerate}[(i)]
		\item \emph{Point admissibility.}
		After one point stage and the admissibility-preserving mapping, the point values remain in the relevant target set:
		\[
			u_\sigma^+\in
			\begin{cases}
			\Gint, & \text{in case (E)},\\
			G_{\mathrm{pt}}, & \text{in case (P)},
			\end{cases}
			\qquad \forall \sigma\in\Sigma_h.
		\]

		\item \emph{Flux-input admissibility.}
		The prepared reconstruction $\widehat{\mathcal R}_h(u_h)$ is $G$-admissible in the sense of \eqref{eq:admissible_recon}. In particular, every trace state entering the numerical flux $\widehat{F}$ in \eqref{eq:avg_update_FE} belongs to $G$.
	\end{enumerate}
\end{proposition}

\begin{proof}
Assertion~(i) follows directly from \Cref{thm:A} and \Cref{cor:pt_practical}. By the definition of $\Pi_K$, the prepared reconstruction $\widehat{\mathcal R}_h(u_h)$ satisfies \eqref{eq:admissible_recon}, which establishes assertion~(ii) and completes the proof.
\end{proof}

\begin{corollary}[IDP property with a practical forward map and admissibility-preserving map pair]\label{cor:hybrid_stage_practical}
\revblue{Assume (A1), (A2), (A3), and (A5), together with boundary conditions of the type specified in Appendix~\ref{sec:boundary_closure}.} Let $\mathcal T:U_{\mathrm{fwd}}\to\mathbb R^m$ be a given forward map and let $\mathcal B:\mathbb R^m\to G_{\mathrm{pt}}$ be a globally defined admissibility-preserving inverse map with convex range $G_{\mathrm{pt}}\subseteq U_{\mathrm{fwd}}$. Suppose that the current point states satisfy $u_\sigma^n\in U_{\mathrm{fwd}}$ and that the resulting practical transformed stage values $w_\sigma^{n+1}\in\mathbb R^m$ are finite. Let $u_h^n\in G_h(G,G_{\mathrm{pt}})$, let $u_K^n=\mathcal R_K(u_h^n)$ be the associated reconstruction polynomial, and let $\hat u_K^n=\Pi_K(u_K^n;\bar u_K^n)$ satisfy \eqref{eq:admissible_recon}. If the conservative CFL condition \eqref{eq:CFL_abstract} holds, one coupled forward Euler stage with the point update $u_\sigma^{n+1}=\mathcal B(w_\sigma^{n+1})$ satisfies
\[
 u_h^{n+1}\in G_h(G,G_{\mathrm{pt}}).
\]
\end{corollary}

\begin{proof}
By \Cref{prop:closed_interface}~(i), the practical point stage maps the skeleton states into $G_{\mathrm{pt}}$, i.e., $u_\sigma^{n+1}\in G_{\mathrm{pt}}$ for all $\sigma\in\Sigma_h$. Independently, \Cref{thm:B} guarantees that $\bar u_K^{n+1}\in G$ for all $K\in\mathcal T_h$ under the conservative CFL condition \eqref{eq:CFL_abstract}. Combining the two componentwise conclusions proves that $u_h^{n+1}\in G_h(G,G_{\mathrm{pt}})$, which completes the proof.
\end{proof}

\section{Cell-average update: conditional admissibility and \emph{a priori} limiters}
\label{sec:avg_idp}

\subsection{Continuous physical fluxes and the need for an IDP flux mechanism}
\label{sec:avg:necessity}

To isolate the limitations of relying solely on continuous boundary traces, we first consider the conservative update obtained by evaluating the single-state physical flux directly from the continuous boundary traces. The corresponding continuous-physical-flux forward Euler update reads
\begin{equation}\label{eq:avg_cont_FE}
	\bar u_{K}^{\mathrm{cont}}
	=
	\bar u_K^n
	-
	\frac{\Delta t}{|K|}
	\sum_{e \in \mathcal{E}_K} |e|
	\sum_{\nu=1}^{N_q} \omega_{e,\nu}\,
	\big(F(u_{K,e,\nu}^n)\cdot n_{K,e}\big),
\end{equation}
where $u_{K,e,\nu}^n$ denote the edge trace values evaluated at $x_{e,\nu}$, and the edge quadrature points $\{x_{e,\nu}\}_{\nu=1}^{N_q}$ and quadrature weights $\{\omega_{e,\nu}>0\}_{\nu=1}^{N_q}$ are given by \eqref{eq:edge_quad}.

\begin{theorem}[Continuous-trace obstruction to a uniform CFL condition guarantee]\label{thm:necessity_discont}
	Assume that there exists for some cell $K$ a feasible CAD for $\mathbb{W}_K$ on $K$ with at least one internal weight $\beta_{K,s_0}>0$.
	Consider the scalar linear advection equation
	\[
	u_t + \bm a\cdot\nabla u = 0,\qquad \bm a\in\mathbb{R}^2\setminus\{0\},
	\]
	with the invariant domain $G=[0,1]$.
	If an internal control value satisfies $u_{K,s_0}^\ast\notin G$ while
	$\bar u_K\in G$ and all boundary quadrature values belong to $G$, then there exists
	\emph{no} fixed \emph{positive} CFL number that guarantees
	$\bar u_{K}^{\mathrm{cont}}\in G$ for the continuous-flux update
	\eqref{eq:avg_cont_FE}.
\end{theorem}

\begin{proof}
The detailed proof is provided in Appendix~\ref{sec:proof_cont_flux}.
\end{proof}

\begin{remark}\label{rem:necessity} 
	\Cref{thm:necessity_discont} highlights the fundamental limitation of the unmodified continuous physical flux in \eqref{eq:avg_cont_FE}. The failure does not stem from the continuity of the skeleton point values per se, but from attempting to deduce a conservative IDP update solely from admissible boundary traces while the interior CAD contribution remains unprotected. Consequently, the proposed framework first enforces the relevant CAD control states to belong to $G$ before evaluating the conservative update via an IDP numerical flux. Since this pre-flux limiting is performed cellwise, the admissible traces supplied from the two sides of an interface are, in general, distinct.
\end{remark}

\begin{remark}[Relation with convex limiting and flux blending]\label{rem:convex_limiting_trace}
	\Cref{thm:necessity_discont} addresses the direct evaluation of admissible continuous traces via a single-state physical flux, rather than through a limited or blended numerical flux. In this setting, the boundary traces alone fail to guarantee a conservative IDP update under any uniform CFL condition. 
     Within the CAD/GQL framework adopted herein, once admissibility of the internal CAD control states is enforced, the boundary flux contribution must still be evaluated through an IDP numerical-flux mechanism. 
	In the present work, this mechanism is realized explicitly: the cellwise map $\Pi_K$ prepares admissible trace inputs prior to flux evaluation, such that the two states $\hat u_{K,e,\nu}^n$ and $\hat u_{K_e,e,\nu}^n$ entering $\widehat F$ in \eqref{eq:avg_update_FE} are, in general, distinct. Convex limiting or convex blending offers an alternative realization. Even if the high-order trace variable is continuous, an active-flux-type limiter can be schematically expressed as
	\[
		\widehat F^{\rm lim}
		=
		\widehat F^{\rm IDP}
		+
		\theta\bigl(\widehat F^{\rm high}-\widehat F^{\rm IDP}\bigr),
		\qquad 0\le \theta\le 1,
	\]
	where $\widehat F^{\rm high}=F(u^{\rm cont})\cdot n$ denotes the continuous high-order physical flux and $\widehat F^{\rm IDP}$ represents a low-order or Riemann-type IDP numerical flux. When $\theta<1$, the conservative flux ceases to be a purely continuous, single-state physical flux; it incorporates the IDP numerical-flux contribution and thereby acts as a genuine interface numerical flux. Consequently, convex limiting or flux blending does not contradict the continuous-trace obstruction; instead, it introduces the requisite numerical dissipation implicitly rather than through explicitly discontinuous trace states. A similar mechanism is observed in the one-dimensional PAMPA analysis; see \cite[Section~4.2]{abgrall2025novel}.
\end{remark}


\subsection{A priori limiters}
\label{sec:avg:limiter}

For each cell $K\in\mathcal{T}_h$, we define
\[
\bar u_K^n:=\frac{1}{|K|}\int_K u_K^n\,\textrm{d}x,
\qquad
(\Pi_h u_h^n)|_K := \Pi_K u_K^n,
\qquad
\Pi_K\in\{\Pi_K^\texttt{c},\Pi_K^\texttt{a}\}.
\]

The \emph{a priori} limiters employed herein build upon the Zhang--Shu local scaling approach \cite{zhang2010maximum,zhang2011b,Zhang2017PPDGNS}, originally developed in the context of discontinuous Galerkin and finite volume methods.
In the present hybrid framework, however, the local reconstruction space \(\mathbb W_K\) can represent an enriched Active Flux or PAMPA approximation space satisfying
\[
(\mathbb P^k(K))^m \subseteq \mathbb W_K .
\]
This structural enrichment distinguishes the present accuracy analysis from the classical results in \cite{zhang2010maximum,Zhang2017PPDGNS}, which were established for standard \(\mathbb P^k\)-based discontinuous Galerkin schemes for scalar conservation laws.

\subsubsection{Classic limiter}\label{sec:2218}

Let $\{x^*_{K,s}\}_s$, $\{\lambda_{K,e}\}_e$, and $\{\beta_{K,s}\}_s$ be the CAD nodes and weights for $\mathbb{W}_K$ on $K$, and let $\mathcal{X}_K$ denote the union of boundary flux quadrature points and interior control points:
\[
\mathcal{X}_{K} := \mathcal{X}^{\mathrm{edg}}_{K} \cup \mathcal{X}^{\mathrm{int}}_{K}, \qquad
\mathcal{X}^{\mathrm{edg}}_{K} := \{x_{e,\nu}: e\in\mathcal{E}_K,\ 1\le \nu\le N_q\},\qquad
\mathcal{X}^{\mathrm{int}}_{K} := \{x_{K,s}^\ast:1\le s\le S_K\}.
\]
We define the classic limited reconstruction by
\begin{equation}\label{eq:PiK_def1}
	\hat{u}^\texttt{c}_K:=(\Pi_K^\texttt{c} u_K^n):=\bar u_K^n+\theta^\texttt{c}_K\big(u_K^n-\bar u_K^n\big),\qquad x\in K,
\end{equation}
with
\begin{equation}\label{eq:theta_def1}
	\theta_K^\texttt{c}
	:=
	\sup\Big\{
	\theta\in[0,1]:
	\bar u_K^n + \theta\left[u_K^n(x)-\bar u_K^n\right]\in G, \ \forall x\in\mathcal{X}_K
	\Big\},
\end{equation}
where $u_K^n= (\mathcal{R}_K u_h^n)|_K \in \mathbb{W}_K \supseteq (\mathbb P^k(K))^m$ satisfies \eqref{eq:recon_constraints}.

\begin{lemma}\label{lem:PiK_props1}
	Assume that $\bar u_K^n\in \Gint$ and there exists a slack parameter $s_0>0$ such that
	\[
	g_j(\bar{u}^n_K)\ge s_0 \qquad \forall j\in\mathcal{J},
	\] 
	Then the classic limiter in \eqref{eq:PiK_def1}--\eqref{eq:theta_def1} satisfies the following properties:
	\begin{itemize}
		\item(i) (\emph{Conservation}) $\frac1{|K|}\int_K \hat{u}^\texttt{c}_K \,\textrm{d}x=\bar u_K^n$.
		\item(ii) (\emph{$G$-admissibility}) $\hat{u}^\texttt{c}_K$ is $G$-admissible.
       \item(iii) (\emph{Accuracy}) Assume that the exact solution $u(\cdot,t^n)$ is smooth and takes values in $\Gint$, and $\|u^n_K - u(\cdot,t^n)\|_{L^\infty(K)} \to 0$ as $h \to 0$. Then for sufficiently small $h$, 
        \[
            \|\hat{u}^\texttt{c}_K - u^n_K\|_{L^\infty(K)} \le \frac{C}{s_0} \|u^n_K-u(\cdot,t^n)\|_{L^\infty(K)}.
        \]
	\end{itemize}
\end{lemma}

\begin{proof}
	(i) Conservation follows from linearity and $\int_K(u_K^n-\bar u_K^n)\,\textrm{d}x=0$.
	
	(ii) The definition of $\theta_K^\texttt{c}$ in \eqref{eq:theta_def1} ensures that $\hat u_K^\texttt{c}(x_{e,\nu}) \in G$ for all $e\in \mathcal{E}_K$ and $1\le\nu\le N_q$, which implies that condition \eqref{eq:915} holds.
    When $\sum_{e}\lambda_{K,e}=1$, condition \eqref{eq:921} holds vacuously; thus, we focus on the case $\sum_{e}\lambda_{K,e}<1$. Since the CAD is feasible for $\mathbb W_K$ and $\hat u_K^\texttt{c}\in\mathbb W_K$, the convex decomposition identity \eqref{eq:gcd_point} applies to $\hat u_K^\texttt{c}$:
	\[
	\bar u^n_K = \frac1{|K|}\int_K \hat u_K^\texttt{c}(x)\,\textrm{d}x
	=
	\sum_{e\in\mathcal E_K}\lambda_{K,e}\sum_{\nu=1}^{N_q}\omega_{e,\nu}\hat u_K^\texttt{c}(x_{e,\nu})
	+\sum_{s=1}^{S_K}\beta_{K,s}\hat u_K^\texttt{c}(x_{K,s}^\ast).
	\]
	Therefore,
	\[
	\bar u_K^n-\sum_{e\in\mathcal E_K}\lambda_{K,e}\sum_{\nu=1}^{N_q}\omega_{e,\nu}\hat u_K^\texttt{c}(x_{e,\nu})
	=
	\sum_{s=1}^{S_K}\beta_{K,s}\hat u_K^\texttt{c}(x_{K,s}^\ast).
	\]
	Dividing by $1-\sum_{e\in\mathcal E_K}\lambda_{K,e}>0$ and recalling that $\sum_{e\in\mathcal{E}_K}\lambda_{K,e}+\sum_{s=1}^{S_K}\beta_{K,s}=1$, we obtain
	\[
	\hat{u}_K^*  = 
	\frac{\bar u_K^n-\sum_{e\in\mathcal E_K}\lambda_{K,e}\sum_{\nu=1}^{N_q}\omega_{e,\nu}\hat u_K^\texttt{c}(x_{e,\nu})}{1-\sum_{e\in\mathcal E_K}\lambda_{K,e}}
	=
	\sum_{s=1}^{S_K}\frac{\beta_{K,s}\,\hat u_K^\texttt{c}(x_{K,s}^\ast)}{1-\sum_{e\in\mathcal E_K}\lambda_{K,e}} \in G.
	\]
	Thus, $\hat{u}_K^*$ is a convex combination of the control values $\{\hat u_K^\texttt{c}(x_{K,s}^\ast)\} \subset G$, and therefore belongs to $G$, which verifies \eqref{eq:921}.
	
	The proof of~(iii) is provided in Appendix~\ref{sec:PiK_classic_accuracy}.
\end{proof}

\begin{lemma}[Local Lipschitz continuity]\label{lem:PiK_Lipschitz1}
	Assume that the invariant domain $G$ is defined by finitely many concave $C^1$ constraints
	$\{g_j\}_{j\in\mathcal{J}}$ as in \eqref{eq:G_def_g}. For a given cell $K$, and consider two reconstructions
	$u_K,v_K\in\mathbb{W}_K  \supseteq (\mathbb P^k(K))^m$ sharing the same control set $\mathcal{X}_K$.
	Let $a:=\bar u_K$, $a':=\bar v_K$ and $b_x:=u_K(x)$, $b'_x:=v_K(x)$ for $x\in\mathcal{X}_K$.
	Assume that $a,a'\in G$ and that there exists a slack parameter $s_0>0$ such that
	\begin{equation}\label{eq:PiK_slack_assumption}
		g_j(a)\ge s_0,\qquad g_j(a')\ge s_0 \qquad \forall j\in\mathcal{J}.
                 	\end{equation}
	Assume further that the values $a,a'$ and $\{b_x,b'_x\}_{x\in\mathcal{X}_K}$ remain in a bounded convex set
	$\mathcal{U}\subset\mathbb{R}^m$ on which each $g_j$ is Lipschitz continuous with constant $L_{g,j}$.
	Then the limited values at the control points satisfy the Lipschitz estimate
	\begin{equation}\label{eq:PiK_Lipschitz_est}
		\max_{x\in\mathcal{X}_K}\bigl\|(\Pi^\texttt{c}_K u_K)(x)-(\Pi^\texttt{c}_K v_K)(x)\bigr\|_2
		\le
		C_{\Pi}\Bigl(
		\|a-a'\|_2+\max_{x\in\mathcal{X}_K}\|b_x-b'_x\|_2
		\Bigr),
	\end{equation}
	where one may take $C_{\Pi}=1+2M\,\max_{j\in\mathcal{J}}\frac{L_{g,j}}{s_0}$, with
	$$M:=\max\{\|a\|_2,\|a'\|_2,\max_{x\in\mathcal{X}_K}\|b_x\|_2,\max_{x\in\mathcal{X}_K}\|b'_x\|_2\}.$$
	In particular, on sets where the slack condition \eqref{eq:PiK_slack_assumption} holds uniformly,
	the trace map $u_h\mapsto \hat u_{K,e,\nu}=(\Pi^\texttt{c}_K u_K)(x_{e,\nu})$ is locally Lipschitz continuous.
\end{lemma}

The proof of \Cref{lem:PiK_Lipschitz1} is provided in Appendix~\ref{sec:PiK_Lipschitz1}.

\subsubsection{Alternative limiter}\label{sec:2226}

Let $\{x_{e,\nu}\}_{e,\nu}$, $\{x_{K,s}^\ast\}_s$, $\{\lambda_{K,e}\}_e$, and $\{\beta_{K,s}\}_s$ denote the CAD nodes and weights associated with $(\mathbb{P}^{k}(K))^m$ on $K$.

We define the alternative limited reconstruction by
\begin{equation}\label{eq:PiK_def2}
	\hat{u}^\texttt{a}_K:=(\Pi_K^\texttt{a} u_K^n)=\bar u_K^n+\theta^\texttt{a}_K\big(u_K^n-\bar u_K^n\big),\qquad x\in K,
\end{equation}
with
\begin{equation}\label{eq:theta_def2}
\theta_K^\texttt{a}:=
\begin{cases}
\sup\Bigl\{\theta\in[0,1]:\ \bar u_K^n + \theta\bigl(u_K^n(x_{e,\nu})-\bar u_K^n\bigr) \in G\ \forall e\in\mathcal E_K,\ \forall \nu, \\
\hspace{7.2em}\text{and } \bar u_K^n + \theta\bigl(u_K^*-\bar u_K^n\bigr)\in G\Bigr\}, & \text{if } \sum_{e\in\mathcal E_K}\lambda_{K,e}<1,\\[0.4em]
\sup\Bigl\{\theta\in[0,1]:\ \bar u_K^n + \theta\bigl(u_K^n(x_{e,\nu})-\bar u_K^n\bigr)\in G\ \forall e\in\mathcal E_K,\ \forall \nu\Bigr\}, & \text{if } \sum_{e\in\mathcal E_K}\lambda_{K,e}=1.
\end{cases}
\end{equation}
where
\begin{equation}\label{eq:920}
	u_K^*:= \frac{\bar{u}_K^n-\sum_{e\in\mathcal{E}_K}\lambda_{K,e}\bar{u}^n_{K,e}}{1-\sum_{e\in\mathcal{E}_K} \lambda_{K,e}}.
\end{equation}
Here, $\bar{u}_{K,e}^n:=\frac{1}{|e|}\int_e {u}^n_K \, \textrm{d}s$ denotes the average of ${u}^n_K$ over the edge $e\in\mathcal{E}_K$ of $K$.
Moreover, $u_K^n$ belongs to the (possibly enriched) polynomial reconstruction space $\mathbb{W}_K$, where $u_K^n = (\mathcal{R}_K u_h^n)|_K$ satisfies \eqref{eq:recon_constraints}.
By the edge quadrature exactness condition in Assumption~(A3), it follows that $\bar{u}_{K,e}^n = \sum_{\nu=1}^{N_q} \omega_{e,\nu} u_K^n(x_{e,\nu}).$

\begin{lemma}\label{lem:PiK_props2}
    Let $u_K^n \in \mathbb{W}_K \supseteq (\mathbb P^k(K))^m$.
    Assume that $\bar u_K^n\in \Gint$ and that there exists a slack parameter $s_0>0$ such that
	\[
	g_j(\bar{u}^n_K)\ge s_0 \qquad \forall j\in\mathcal{J},
	\] 
	Then the alternative limiter in \eqref{eq:PiK_def2}--\eqref{eq:theta_def2} satisfies the following properties:
	\begin{itemize}
		\item(i) (\emph{Conservation}) $\frac1{|K|}\int_K \hat{u}^\texttt{a}_K \,\textrm{d}x=\bar u_K^n$.
		\item(ii) (\emph{$G$-admissibility}) $\hat{u}^\texttt{a}_K$ is $G$-admissible.
        \item(iii) (\emph{Accuracy}) Assume that the exact solution $u(\cdot,t^n)$ is smooth and takes values within $\Gint$, and $\|u^n_K - u(\cdot,t^n)\|_{L^\infty(K)} \to 0$ as $h \to 0$. Then for sufficiently small $h$, 
        \[
		\|\hat{u}^\texttt{a}_K - u^n_K\|_{L^\infty(K)} \le \frac{C}{s_0} \left(\|u^n_K-u(\cdot,t^n)\|_{L^\infty(K)}+C_q h^{k+1}\|u(\cdot,t^n)\|_{W^{k+1,\infty}(K)}\right),
		\]
        where $k+1$ is the formal spatial accuracy order, and the constants $C$ and $C_q$ are independent of $h$ and $K$ across a shape-regular mesh family, provided that either $\beta_K = 0$ or $\beta_K \ge \beta_0 > 0$. Here,
        $
            \beta_K:=1-\sum_{e\in\mathcal E_K}\lambda_{K,e}\in[0,1].
        $
	\end{itemize}
\end{lemma}

\begin{proof}
     Let
	\[
	E_K:=\|u_K^n-u(\cdot,t^n)\|_{L^\infty(K)}.
	\]
	
	Property~(i) follows from linearity and $\int_K (u_K^n-\bar u_K^n)\,\textrm{d}x=0$.
	
	For Property~(ii), condition \eqref{eq:915} follows directly from the definition of $\theta_K^\texttt{a}$ in \eqref{eq:theta_def2}. 
    When $\sum_{e}\lambda_{K,e}=1$, condition \eqref{eq:921} holds vacuously. 
    When $\sum_{e}\lambda_{K,e}<1$, we have
    \[
        \bar{\hat u}^{\texttt{a}}_{K,e} := \frac{1}{|e|} \int_e \hat u_K^\texttt{a} \, \textrm{d}s = \bar{u}_K^n + \theta_K^\texttt{a} \left(\frac{1}{|e|} \int_e u_K^n \, \textrm{d}s -\bar{u}_K^n\right) = \bar{u}_K^n + \theta_K^\texttt{a} (\bar{u}^n_{K,e}-\bar{u}_K^n).
    \]
    Since the limited reconstruction \(\hat u_K^\texttt{a}\) also belongs to \(\mathbb W_K\), Assumption~(A3) yields
    \[
        \bar{\hat u}^{\texttt a}_{K,e}
        =
        \sum_{\nu=1}^{N_q}\omega_{e,\nu}
        \hat u_K^\texttt a(x_{e,\nu}).
    \]
    Hence, 
    \[
        u_K^{\texttt{a},*}:= \frac{\bar{u}_K^n-\sum_{e\in\mathcal{E}_K}\lambda_{K,e}\bar{\hat u}^{\texttt{a}}_{K,e}}{1-\sum_{e\in\mathcal{E}_K} \lambda_{K,e}} = \bar{u}_K^n + \theta_K^\texttt{a} (u_K^*-\bar{u}_K^n) \in G,
    \]
    and therefore condition \eqref{eq:921} is satisfied.
	
	For Property~(iii), since the exact solution $u(\cdot,t^n)$ is smooth and satisfies $u(\cdot,t^n)\in G$, for sufficiently small $h$ all states
	\[
	\bar u_K^n,\ \{u_K^n(x_{e,\nu})\}_{e,\nu},\ \text{and (if $\beta_K>0$) }u_K^*
	\]
	remain in a bounded convex set $\mathcal U\subset\mathbb R^m$ on which each $g_j$ is of class $C^1$.
	Hence, each $g_j$ is Lipschitz continuous on $\mathcal U$; let $L_{g,j}$ denote the Lipschitz constant of $g_j$ on $\mathcal U$ and set
	\[
	L_g:=\max_{j\in\mathcal J}L_{g,j}.
	\]
    Let
    \begin{equation*}
        \mathcal Y_K:=
        \begin{cases}
            \{u_K^n(x_{e,\nu}):e\in\mathcal E_K,\ 1\le \nu\le N_q\}\cup\{u_K^*\}, & \beta_K>0, \\
            \{u_K^n(x_{e,\nu}):e\in\mathcal E_K,\ 1\le \nu\le N_q\}, & \beta_K=0.
        \end{cases}
    \end{equation*}
    For each $z \in \mathcal{Y}_K$ and $j \in \mathcal{J}$, we introduce the scaling parameter
    $$\theta_{j,z} := \sup\{\theta \in [0, 1] : (1 - \theta)s_0 + \theta g_j(z) \ge 0\}.$$
    A direct computation yields $\theta_{j,z} = 1$ if $g_j(z) \ge 0$, and $\theta_{j,z} = \frac{s_0}{s_0 - g_j(z)}$ otherwise. 
    In either case, the deviation from unity can be bounded as follows:
    $$1 - \theta_{j,z} \le \frac{(-g_j(z))_+}{s_0},$$
    where $(y)_+ := \max(0, y)$. 
    Set $\theta^* := \min_{z \in \mathcal{Y}_K} \min_{j \in \mathcal{J}} \theta_{j,z}$. By the concavity of each $g_j$ and the slack condition $g_j(\bar{u}_K^n) \ge s_0$, it holds that for all $z \in \mathcal{Y}_K$ and $j \in \mathcal{J}$,
    $$g_j(\bar{u}_K^n + \theta^*(z - \bar{u}_K^n)) \ge (1 - \theta^*)s_0 + \theta^* g_j(z) \ge 0.$$
    which implies that $\bar{u}_K^n + \theta^*(z - \bar{u}_K^n) \in G$ for all $z \in \mathcal{Y}_K$. Consequently, the definition of $\theta_K^{\mathtt{a}}$ in \eqref{eq:theta_def2} guarantees that $\theta_K^{\mathtt{a}} \ge \theta^*$. Thus, we obtain
    \begin{equation}\label{eq:lemPiK2_theta_basic}
        1 - \theta_K^{\mathtt{a}} \le 1 - \theta^* = \max_{z \in \mathcal{Y}_K} \max_{j \in \mathcal{J}} (1 - \theta_{j,z}) \le \frac{1}{s_0} \max_{z \in \mathcal{Y}_K} \max_{j \in \mathcal{J}} (-g_j(z))_+.
    \end{equation}
    Thus, it remains to bound the constraint violations at the control values in $\mathcal{Y}_K$.
	
	For each edge quadrature point $x_{e,\nu}$, since the exact solution belongs to $G$, it holds that $g_j(u(x_{e,\nu},t^n))\ge 0$.
	Therefore,
	\[
	\bigl(-g_j(u_K^n(x_{e,\nu}))\bigr)_+
	\le
	\bigl|g_j(u_K^n(x_{e,\nu}))-g_j(u(x_{e,\nu},t^n))\bigr|
	\le L_g\,E_K,
	\]
	and hence
	\begin{equation}\label{eq:lemPiK2_edge_violation}
		\max_{e,\nu}\max_{j\in\mathcal J}\bigl(-g_j(u_K^n(x_{e,\nu}))\bigr)_+ \le L_g\,E_K .
	\end{equation}
	
	When $\beta_K=0$, i.e., $\sum_{e\in\mathcal{E}_K} \lambda_{K,e} = 1$, combining \eqref{eq:lemPiK2_theta_basic} and \eqref{eq:lemPiK2_edge_violation} immediately yields the desired estimate. Hence, it remains to consider the case $\beta_K>0$.
	
	From \cite[Chapter~3]{Ciarlet2002}, there exists a componentwise polynomial approximation $q_K\in (\mathbb{P}^{k}(K))^m$ of the exact solution such that, for a smooth solution 
	$u(\cdot,t^n)$,
	\begin{equation}\label{eq:lemPiK2_q_approx}
		E_K^q:=\|q_K-u(\cdot,t^n)\|_{L^\infty(K)}
		\le C_q h^{k+1}\|u(\cdot,t^n)\|_{W^{k+1,\infty}(K)}.
	\end{equation}
	Define
	\begin{equation}\label{eq:lemPiK2_qstar_def}
		q_K^*:=
		\frac{\bar q_K-\sum_{e\in\mathcal E_K}\lambda_{K,e}\bar q_{K,e}}{\beta_K}.
	\end{equation}
	Because the CAD is feasible for $(\mathbb{P}^{k}(K))^m$, applying the CAD identity \eqref{eq:CAD} to $q_K$ yields
	\[
	\bar q_K=\sum_{e\in\mathcal E_K}\lambda_{K,e}\bar q_{K,e}
	+\sum_{s=1}^{S_K}\beta_{K,s}q_K(x_{K,s}^*),
	\]
	which implies that
	\begin{equation}\label{eq:lemPiK2_qstar_convex_combo}
		q_K^*=\sum_{s=1}^{S_K}\eta_{K,s}\,q_K(x_{K,s}^*),
		\qquad
		\eta_{K,s}:=\frac{\beta_{K,s}}{\beta_K}>0,
		\qquad
		\sum_{s=1}^{S_K}\eta_{K,s}=1.
	\end{equation}
	For all $j\in\mathcal J$ and all interior nodes $x_{K,s}^*$, since $g_j(u(x_{K,s}^*,t^n))\ge 0$ and $g_j$ is Lipschitz continuous on $\mathcal U$,
	\[
	\bigl(-g_j(q_K(x_{K,s}^*))\bigr)_+
	\le
	\bigl|g_j(q_K(x_{K,s}^*))-g_j(u(x_{K,s}^*,t^n))\bigr|
	\le L_g\,E_K^q.
	\]
	Using the concavity of $g_j$ together with \eqref{eq:lemPiK2_qstar_convex_combo}, we have
	\[
	g_j(q_K^*)
	=g_j\Bigl(\sum_{s=1}^{S_K}\eta_{K,s}\,q_K(x_{K,s}^*)\Bigr)
	\ge \sum_{s=1}^{S_K}\eta_{K,s}\,g_j(q_K(x_{K,s}^*)),
	\]
	and therefore
	\begin{equation}\label{eq:lemPiK2_qstar_violation}
		\max_{j\in\mathcal J}\bigl(-g_j(q_K^*)\bigr)_+ \le L_g\,E_K^q.
	\end{equation}
	
	Next, we compare the auxiliary states $u_K^*$ and $q_K^*$. By \eqref{eq:920} and \eqref{eq:lemPiK2_qstar_def},
	\[
	u_K^*-q_K^*
	=
	\frac{
		(\bar u_K^n-\bar q_K)-\sum_{e\in\mathcal E_K}\lambda_{K,e}(\bar u_{K,e}^n-\bar q_{K,e})
	}{\beta_K}.
	\]
    Thus, by the triangle inequality and the definitions of $E_K$ and $E_K^q$, along with the assumption $\beta_K \ge \beta_0 > 0$, we have
	\begin{equation}\label{eq:lemPiK2_ustar_qstar_diff}
		\|u_K^*-q_K^*\|_2 \le \frac{1+\sum_{e\in\mathcal E_K}\lambda_{K,e}}{\beta_K} (E_K+E_K^q)
		\le \frac{2}{\beta_0}\,(E_K+E_K^q).
	\end{equation}
    For sufficiently small $h$, the auxiliary state $q_K^*$ also remains in the bounded convex set $\mathcal{U}$.
	Using the Lipschitz continuity of $g_j$ on $\mathcal U$, it follows that
	\begin{align}
		\max_{j\in\mathcal J}\bigl(-g_j(u_K^*)\bigr)_+
		&\le
		\max_{j\in\mathcal J}\bigl(-g_j(q_K^*)\bigr)_+
		+\max_{j\in\mathcal J}\bigl|g_j(u_K^*)-g_j(q_K^*)\bigr| \notag\\
		&\le
		L_g\,E_K^q + L_g\|u_K^*-q_K^*\|_2 
		\le C_K\,(E_K+E_K^q),
		\label{eq:lemPiK2_ustar_violation}
	\end{align}
	with $C_K$ independent of $h$ and $K$.
	
	Combining \eqref{eq:lemPiK2_theta_basic}, \eqref{eq:lemPiK2_edge_violation}, and \eqref{eq:lemPiK2_ustar_violation}, it follows that
	\begin{equation}\label{eq:lemPiK2_theta_est}
		1-\theta_K^\texttt{a}\le \frac{C}{s_0}\,(E_K+E_K^q),
	\end{equation}
	where $C$ is independent of $h$ (for $h$ sufficiently small).
	
	Finally, from \eqref{eq:PiK_def2},
	\[
	\hat u_K^\texttt{a}-u_K^n=(1-\theta_K^\texttt{a})(\bar u_K^n-u_K^n),
	\]
	hence
	\[
	\|\hat u_K^\texttt{a}-u_K^n\|_{L^\infty(K)}
	\le
	(1-\theta_K^\texttt{a})\|u_K^n-\bar u_K^n\|_{L^\infty(K)}.
	\]
	Since $u(\cdot,t^n)$ is smooth and $u_K^n$ approximates it, $\|u_K^n-\bar u_K^n\|_{L^\infty(K)}$ is uniformly bounded for sufficiently small $h$.
	Therefore, together with \eqref{eq:lemPiK2_theta_est},
	\begin{equation}\label{eq:lemPiK2_pre_final}
		\|\hat u_K^\texttt{a}-u_K^n\|_{L^\infty(K)}
		\le \frac{C}{s_0}\,(E_K+E_K^q).
	\end{equation}
	Substituting \eqref{eq:lemPiK2_q_approx} into \eqref{eq:lemPiK2_pre_final} completes the proof of Property~(iii).
\end{proof}

\begin{corollary}[Asymptotically negligible limiter perturbation on smooth interior states]\label{cor:limiter_small_smooth}
Let $u_K^n \in \mathbb{W}_K \supseteq (\mathbb P^k(K))^m$.
Assume that the hypotheses of \Cref{lem:PiK_props1} or \Cref{lem:PiK_props2} hold, and in addition that the reconstruction satisfies
\[
\|u_K^n-u(\cdot,t^n)\|_{L^\infty(K)}=O(h^{k+1})
\]
on a shape-regular mesh family while the exact solution remains a positive distance away from $\partial G$. Then
\[
\|\Pi_K u_K^n-u_K^n\|_{L^\infty(K)}=O(h^{k+1})
\]
for both the classic and the alternative limiter. In particular, the limiter perturbation is asymptotically negligible on smooth interior solutions.
\end{corollary}

\begin{proof}
For the classic limiter, the estimate follows directly from \Cref{lem:PiK_props1}(iii). For the alternative limiter, \Cref{lem:PiK_props2}(iii) gives
\[
\|\Pi_K^\texttt{a}u_K^n-u_K^n\|_{L^\infty(K)}\le \frac{C}{s_0}\|u_K^n-u(\cdot,t^n)\|_{L^\infty(K)}+C_q h^{k+1}\|u(\cdot,t^n)\|_{W^{k+1,\infty}(K)},
\]
which is again $O(h^{k+1})$ under the stated smooth-regime approximation assumption, which completes the proof.
\end{proof}

\begin{lemma}[Local Lipschitz continuity of the alternative limiter]\label{lem:PiK_Lipschitz2}
	Assume that the invariant domain $G$ is defined by finitely many concave $C^1$ constraints
	$\{g_j\}_{j\in\mathcal J}$ as in \eqref{eq:G_def_g}. 
	For a given cell $K$, consider two reconstructions $u_K,v_K\in\mathbb W_K \supseteq (\mathbb P^k(K))^m$. Let
	\[
	a:=\bar u_K,\qquad a':=\bar v_K,\qquad
	b_{e,\nu}:=u_K(x_{e,\nu}),\qquad b'_{e,\nu}:=v_K(x_{e,\nu}).
	\]
	When $\beta_K>0$, we define the auxiliary states
	\begin{equation}\label{eq:PiKa_ustar_def_u}
		u_K^*:=\frac{a-\sum_{e\in\mathcal E_K}\lambda_{K,e}\bar u_{K,e}}{\beta_K},
		\qquad
		v_K^*:=\frac{a'-\sum_{e\in\mathcal E_K}\lambda_{K,e}\bar v_{K,e}}{\beta_K},
	\end{equation}
	where $\bar u_{K,e},\bar v_{K,e}$ are the edge averages. Assume that $a,a'\in G$ and that there exists a slack parameter $s_0>0$ such that
	\begin{equation}\label{eq:PiKa_slack_assumption}
		g_j(a)\ge s_0,\qquad g_j(a')\ge s_0
		\qquad \forall j\in\mathcal J.
	\end{equation}
	Assume further that all states
	\[
	a,\ a',\ \{b_{e,\nu}\}_{e,\nu},\ \{b'_{e,\nu}\}_{e,\nu},
	\quad\text{and (if $\beta_K>0$) }u_K^*,v_K^*
	\]
	remain in a bounded convex set $\mathcal U\subset\mathbb R^m$ on which each $g_j$ is Lipschitz continuous 
	with constant $L_{g,j}$. Then the alternative limiter \eqref{eq:PiK_def2}--\eqref{eq:theta_def2} is locally Lipschitz continuous on the edge control values:
	\begin{equation}\label{eq:PiKa_Lipschitz_est}
		\max_{e\in\mathcal E_K}\max_{1\le \nu\le N_q}
		\bigl\|(\Pi_K^\texttt{a}u_K)(x_{e,\nu})-(\Pi_K^\texttt{a}v_K)(x_{e,\nu})\bigr\|_2
		\le
		C_{\Pi,\texttt{a}}
		\Bigl(
		\|a-a'\|_2
		+
		\max_{e,\nu}\|b_{e,\nu}-b'_{e,\nu}\|_2
		\Bigr),
	\end{equation}
	where one may take
	\[
	C_{\Pi,\texttt{a}}
	=
	1+2M_{\texttt{a}}
	\Bigl(\max_{j\in\mathcal J}\frac{L_{g,j}}{s_0}\Bigr)\Gamma_K,
	\]
	with
	\[
	\Gamma_K:=
	\begin{cases}
		1, & \beta_K=0,\\
		1+\beta_K^{-1}, & \beta_K>0,
	\end{cases}
	\]
	and
	\[
	M_{\texttt{a}}:=\max\Bigl\{
	\|a\|_2,\|a'\|_2,\max_{e,\nu}\|b_{e,\nu}\|_2,\max_{e,\nu}\|b'_{e,\nu}\|_2,
	\mathbf 1_{\{\beta_K>0\}}\|u_K^*\|_2,\mathbf 1_{\{\beta_K>0\}}\|v_K^*\|_2
	\Bigr\}.
	\]
	In particular, on sets where \eqref{eq:PiKa_slack_assumption} holds uniformly,
	the trace map $u_h\mapsto \{(\Pi_K^\texttt{a}u_K)(x_{e,\nu})\}_{e,\nu}$ is locally Lipschitz continuous. 
	Furthermore, provided the mesh family is shape-regular and satisfies the uniform interior-mass bound $\beta_K \ge \beta_0 > 0$, the factor $\Gamma_K$ is bounded by $1+\beta_0^{-1}$. Consequently, the local Lipschitz constant $C_{\Pi,\texttt{a}}$ is independent of the mesh size $h$ and the cell $K$.
\end{lemma}

The proof of \Cref{lem:PiK_Lipschitz2} is provided in Appendix~\ref{sec:PiK_Lipschitz2}.

\subsection{Geometric IDP property and numerical fluxes}
\label{sec:avg:flux}

\begin{corollary}[Geometric IDP property]\label{cor:idp_flux}
Under Assumption~(A5), we obtain the following geometric inclusion property: for any $u_L,u_R\in G$ and any unit normal vector $n$, let $\alpha$ satisfy
\[
\alpha\ge \max\{\tilde\alpha(u_L,n),\tilde\alpha(u_R,n)\}.
\]
Then the intermediate split states
\begin{equation}\label{eq:weakLF}
	u^{\mathrm{LF},\pm}(u_L,u_R;n)
	:= \frac12\bigl(u_L+u_R\bigr)\mp \frac{1}{2\alpha}\bigl(F(u_R)-F(u_L)\bigr)\cdot n
\end{equation}
belong to the invariant domain $G$. Consequently, the associated first-order Lax--Friedrichs scheme is IDP on $G$.
\end{corollary}

\begin{proof}
	By Assumption~(A2), the invariant domain satisfies $G=G^\ast$ in the sense of \Cref{prop:GQL}. Hence, verifying that $u^{\mathrm{LF}, \pm} \in G$ is equivalent to verifying the linear half-space inequalities in the GQL characterization \eqref{eq:G_star}. Using the definition \eqref{eq:weakLF} together with Assumption~(A5) in \Cref{sec:assumptions}, for any $j \in \mathcal{J}$ and $u_j^\ast \in S_j$, it follows that
	\begin{align*}
		\big( u^{\mathrm{LF}, \pm} (u_L, u_R; n) - u_j^\ast \big) \cdot n_j^\ast 
		&= \frac{1}{2\alpha} \Big[ \alpha (u_L - u_j^\ast) \cdot n_j^\ast \pm \big(F(u_L) \cdot n\big) \cdot n_j^\ast \Big] \\
		& + \frac{1}{2\alpha} \Big[ \alpha (u_R - u_j^\ast) \cdot n_j^\ast \mp \big(F(u_R) \cdot n\big) \cdot n_j^\ast \Big] \\
		&\ge \frac{1}{2\alpha} \Big( \mp \zeta(u_j^\ast) \cdot n \pm \zeta(u_j^\ast) \cdot n \Big) = 0.
	\end{align*}
	Thus, we conclude that $u^{\mathrm{LF}, \pm} \in G^\ast = G$, which completes the proof.
\end{proof}

Let $\widehat F(\cdot,\cdot;n)$ be a conservative and consistent numerical flux satisfying:
\[
\widehat F(u,u;n)=F(u)\cdot n,\qquad
\widehat F(u^-,u^+;n)=-\widehat F(u^+,u^-;-n).
\]
Throughout the subsequent analysis, we employ the local Lax--Friedrichs (Rusanov) numerical flux family
\begin{equation}\label{eq:LF_flux}
	\widehat F_{\alpha}(u^-,u^+;n)
	:=\frac12\big[F(u^-)+F(u^+)\big]\cdot n - \frac{\alpha}{2}(u^+-u^-),
\end{equation}
with the viscosity parameter $\alpha\ge \alpha(u^-,u^+,n)$ specified in \Cref{cor:idp_flux}.
More generally, any numerical flux that is one-dimensionally invariant-domain preserving (e.g., Godunov or HLL-type fluxes) can be directly accommodated by replacing the constant $1$ in the CFL condition with the corresponding one-dimensional IDP CFL constant $c_0\in(0,1]$  (see \Cref{rem:idp_flux_general}).

In the subsequent analysis, for each interface quadrature node $(e,\nu)$, we choose the local numerical viscosity parameter $\alpha_{K,e,\nu}$ satisfying
\[
\alpha_{K,e,\nu}\ \ge\
\alpha(\hat u_{K,e,\nu}^n,\hat u_{K_e,e,\nu}^n,n_{K,e})
=
\max\Big\{
\tilde\alpha(\hat u_{K,e,\nu}^n,n_{K,e}),\ 
\tilde\alpha(\hat u_{K_e,e,\nu}^n,n_{K,e})
\Big\}.
\]
This choice guarantees both the geometric IDP property of the numerical flux and the validity of the algebraic half-space condition \eqref{eq:weakLF2} (cf.~Assumption~(A5) in \Cref{sec:assumptions}) for the prepared trace states on both sides of the interface.

\begin{remark}[Beyond Lax--Friedrichs numerical fluxes]\label{rem:idp_flux_general}
	When $\widehat F(\cdot,\cdot;n)$ is an arbitrary 1D IDP numerical flux with maximal allowable CFL number $c_0\in(0,1]$, the resulting conservative CFL condition in \Cref{thm:B} modifies to
	$\max_{e,\nu} \frac{\alpha_{K,e,\nu} \Delta t\,|e|}{|K| \lambda_{K,e}}=\max_{e,\nu}\frac{\alpha_{K,e,\nu}\Delta t}{h_{K,e}}\le c_0$.
    The underlying analytical framework remains unchanged, with only the upper bound in the CFL constraint adjusted by the factor $c_0$.
\end{remark}

\subsection{Proof of the conservative cell-average theorem}
\label{sec:avg:proofB}

We now establish the conditional IDP property for the conservative cell-average update formulated in \Cref{thm:B}.

Consider an arbitrary forward Euler step (or an individual stage within an SSP Runge--Kutta integration). For each cell $K$, let
\[
\hat u_K := (\widehat{\mathcal{R}}_h u_h)|_K = (\Pi_K\circ \mathcal{R}_K)(u_h),
\]
and let the interior and exterior prepared trace states evaluated at the edge quadrature nodes be denoted by
\[
\hat u_{K,e,\nu} := \hat u_K(x_{e,\nu}),\qquad
\hat u_{K_e,e,\nu} := \hat u_{K_e}(x_{e,\nu}),
\]
where $K_e$ denotes the adjacent cell sharing the edge $e$. Here, $\Pi_K$ represents either the classic limiter $\Pi_K^\texttt{c}$ or the alternative limiter $\Pi_K^\texttt{a}$. By the hypotheses of \Cref{thm:B}, the prepared reconstruction $\hat{u}_K$ is $G$-admissible in the sense of \eqref{eq:admissible_recon}, as established in \Cref{lem:PiK_props1} and \Cref{lem:PiK_props2} for the classic and alternative limiters, respectively.

\begin{proof} [Proof of \Cref{thm:B}] We first focus on the general case where $\sum_{e\in\mathcal{E}_K} \lambda_{K,e} < 1$.
	
	\noindent\emph{Step 1: Substitution into the forward Euler update.}
	By the CAD in Assumption (A3), we have
	\begin{equation}\label{eq:gcd_hat}
\begin{aligned}
\bar u_K^n
=&\sum_{e\in\mathcal{E}_K} \lambda_{K,e}\sum_{\nu=1}^{N_q}\omega_{e,\nu}\,\hat u_{K,e,\nu}^n
+\left( \bar u_K^n - \sum_{e\in\mathcal{E}_K} \lambda_{K,e}\sum_{\nu=1}^{N_q}\omega_{e,\nu}\,\hat u_{K,e,\nu}^n\right)\\
=&\sum_{e\in\mathcal{E}_K} \lambda_{K,e}\sum_{\nu=1}^{N_q}\omega_{e,\nu}\,\hat u_{K,e,\nu}^n
+\left( 1 - \sum_{e\in\mathcal{E}_K} \lambda_{K,e} \right)
\underbrace{\left( \dfrac{\bar u_K^n - \sum_{e\in\mathcal{E}_K} \lambda_{K,e}\sum_{\nu=1}^{N_q}\omega_{e,\nu}\,\hat u_{K,e,\nu}^n}{1 - \sum_{e\in\mathcal{E}_K} \lambda_{K,e}}\right)}_{=:\hat u_K^*}.
\end{aligned}
\end{equation}
	Since $\hat u_K^n$ is $G$-admissible, it holds that $\hat u_{K,e,\nu}^n \in G$ and $\hat u_{K_e,e,\nu}^n \in G$; hence, the numerical flux is well-defined on $G\times G$. 
    Furthermore, \Cref{lem:PiK_props1}(ii) or \Cref{lem:PiK_props2}(ii) guarantees that $\hat u^*_K \in G$.
	We introduce the edge ratios $c_{K,e}=\lambda_{K,e}/|e|$ as in \eqref{eq:edge_ratio}.
    Substituting the identity \eqref{eq:gcd_hat} into the conservative forward Euler update \eqref{eq:avg_update_FE}, we obtain
	\begin{equation}\label{eq:avg_after_subst}
		\begin{aligned}
			\bar u_K^{n+1}
			= 
			\sum_{e\in\mathcal{E}_K} |e|\sum_{\nu=1}^{N_q}\omega_{e,\nu} U_{e,\nu}
			+
			\left( 1 - \sum_{e\in\mathcal{E}_K} \lambda_{K,e} \right) \hat u^*_K,
		\end{aligned}
	\end{equation}
    with 
    $U_{e,\nu} := c_{K,e}\hat u_{K,e,\nu}^n
			-\frac{\Delta t}{|K|}
			\widehat F_{\alpha_{K,e,\nu}}(\hat u_{K,e,\nu}^n,\hat u_{K_e,e,\nu}^n,n_{K,e}).$
	
	\noindent\emph{Step 2: Lax--Friedrichs flux splitting.}
	Using \eqref{eq:LF_flux} and introducing the flux-split sub-states
    \begin{align}
		U_{K,e,\nu}
		&:=
		\left(c_{K,e}-\frac{\alpha_{K,e,\nu}\Delta t}{2|K|}\right)\hat u_{K,e,\nu}^n
		-\frac{\Delta t}{2|K|}\big(F(\hat u_{K,e,\nu}^n)\cdot n_{K,e}\big),
		\label{eq:U_in}\\
		U_{K_e,e,\nu}
		&:=
		\frac{\alpha_{K,e,\nu}\Delta t}{2|K|}\hat u_{K_e,e,\nu}^n
		-\frac{\Delta t}{2|K|}\big(F(\hat u_{K_e,e,\nu}^n)\cdot n_{K,e}\big),
		\label{eq:U_out}
	\end{align}
    we obtain for each $(e,\nu)$ the additive decomposition 
    $
        U_{e,\nu}=U_{K,e,\nu} + U_{K_e,e,\nu},
    $
    which recasts \eqref{eq:avg_after_subst} into
	\begin{equation}\label{eq:avg_convex_form}
		\bar u_K^{n+1}
		=
		\sum_{e\in\mathcal{E}_K} |e|\sum_{\nu=1}^{N_q}\omega_{e,\nu}
		\big(U_{K,e,\nu}+U_{K_e,e,\nu}\big)
		+
		\left( 1 - \sum_{e\in\mathcal{E}_K} \lambda_{K,e} \right) \hat u^*_K.
	\end{equation}
    
	\noindent\emph{Step 3: Verification of the GQL inequalities.}
	Fix an arbitrary constraint index $j\in \mathcal{J}$ and any supporting state $u_j^\ast\in S_j$.
	Our goal is to establish the GQL inequality $(\bar u_K^{n+1}-u_j^\ast)\cdot n_j^\ast \ge 0$, which will be verified via the complete boundary summation in Step~4.
	
	First, the conservative CFL condition \eqref{eq:CFL_abstract} implies that
	\begin{equation}\label{eq:c1_lower}
		c_{1}:=c_{K,e}-\frac{\alpha_{K,e,\nu}\Delta t}{2|K|}
		\ \ge\ \frac{\alpha_{K,e,\nu}\Delta t}{2|K|} =: c_2.
	\end{equation}
	For the interior contribution $U_{K,e,\nu}$ in \eqref{eq:U_in}, we compute
	\begin{align*}
		\big(U_{K,e,\nu}-c_{1}u_j^\ast\big)\cdot n_j^\ast
		&=
		c_{1}\,(\hat u_{K,e,\nu}^n-u_j^\ast)\cdot n_j^\ast
		-\frac{\Delta t}{2|K|}\big(F(\hat u_{K,e,\nu}^n)\cdot n_{K,e}\big)\cdot n_j^\ast\\
		&\overset{\eqref{eq:c1_lower}}{\ge}
		c_{2}\,(\hat u_{K,e,\nu}^n-u_j^\ast)\cdot n_j^\ast
		-\frac{\Delta t}{2|K|}\big(F(\hat u_{K,e,\nu}^n)\cdot n_{K,e}\big)\cdot n_j^\ast
	\end{align*}
	Since $\hat u_{K,e,\nu}^n\in G$ and $\alpha_{K,e,\nu}\ge \tilde\alpha(\hat u_{K,e,\nu}^n,n_{K,e})$,
	Assumption~(A5) ensures that
	\begin{equation}\label{eq:ineq_Uin}
		\big(U_{K,e,\nu}-c_{1}u_j^\ast\big)\cdot n_j^\ast
		\ \ge\
		\frac{\Delta t}{2|K|}\,\zeta(u_j^\ast)\cdot n_{K,e}.
	\end{equation}
	Similarly, for the exterior contribution $U_{K_e,e,\nu}$ in \eqref{eq:U_out}, 
	Assumption~(A5) applied to $\hat u_{K_e,e,\nu}^n\in G$ yields
	\begin{equation}\label{eq:ineq_Uout}
		\big(U_{K_e,e,\nu}-c_{2}u_j^\ast\big)\cdot n_j^\ast
		\ \ge\
		\frac{\Delta t}{2|K|}\,\zeta(u_j^\ast)\cdot n_{K,e}.
	\end{equation}
	Finally, \Cref{lem:PiK_props1}(ii) and \Cref{lem:PiK_props2}(ii) yield $\hat u_K^* \in G$, which guarantees that
	\begin{equation}\label{eq:ineq_internal}
		(\hat u_K^*-u_j^\ast)\cdot n_j^\ast \ \ge\ 0.
	\end{equation}
	
	\noindent\emph{Step 4: Summation over edges and polygon closure identity.}
	From \eqref{eq:avg_convex_form} and the relation that
	\[
	(c_{1}+c_{2})|e| = c_{K,e}|e|=\lambda_{K,e},
	\]
	we can write
	\begin{align*}
		(\bar u_K^{n+1}-u_j^\ast)\cdot n_j^\ast
		&=
		\sum_{e\in\mathcal{E}_K}|e|\sum_{\nu=1}^{N_q}\omega_{e,\nu}
		\Big[
		\big(U_{K,e,\nu}-c_1u_j^\ast\big)\cdot n_j^\ast
		+
		\big(U_{K_e,e,\nu}-c_2u_j^\ast\big)\cdot n_j^\ast
		\Big]\\
		&\quad 
		+ 
		\left( 1 - \sum_{e\in\mathcal{E}_K} \lambda_{K,e} \right) \,(\hat u_{K}^{\ast}-u_j^\ast)\cdot n_j^\ast.
	\end{align*}
	Substituting \eqref{eq:ineq_Uin}--\eqref{eq:ineq_internal} into the above identity yields
	\[
	(\bar u_K^{n+1}-u_j^\ast)\cdot n_j^\ast
	\ \ge\
	\frac{\Delta t}{2|K|}
	\sum_{e\in\mathcal{E}_K}|e|\sum_{\nu=1}^{N_q}\omega_{e,\nu}
	\Big( \zeta(u_j^\ast)\cdot n_{K,e} + \zeta(u_j^\ast)\cdot n_{K,e}\Big).
	\]
	Since $\sum_{\nu}\omega_{e,\nu}=1$, this simplifies to
	\[
	(\bar u_K^{n+1}-u_j^\ast)\cdot n_j^\ast
	\ \ge\
	\frac{\Delta t}{|K|}\,\zeta(u_j^\ast)\cdot
	\Big(\sum_{e\in\mathcal{E}_K}|e|\,n_{K,e}\Big).
	\]
	For any (convex) polygonal cell, the boundary-normal closure identity holds:
	\[
	\sum_{e\in\mathcal{E}_K}|e|\,n_{K,e} = 0,
	\]
	which follows directly from the divergence theorem applied to a constant vector field.
	Therefore, $(\bar u_K^{n+1}-u_j^\ast)\cdot n_j^\ast \ge 0$ for every
	$j\in \mathcal{J}$ and $u_j^\ast\in S_j$, which implies that $\bar u_K^{n+1}\in G^\ast=G$.
	Finally, when $\sum_{e\in\mathcal{E}_K}\lambda_{K,e}=1$, the result follows analogously by omitting the internal auxiliary term involving $\hat u_K^*$, which completes the proof.
\end{proof}

\begin{proof}[Proof of \Cref{thm:hybrid_stage_IDP}]
By \Cref{prop:closed_interface}(i), the point stage computed with the exact inverse mapping $\Psi^{-1}$ satisfies $u_\sigma^{n+1}\in \Gint$ for all $\sigma\in\Sigma_h$. Under \Cref{thm:B}, the conservative CFL condition \eqref{eq:CFL_abstract} guarantees that $\bar u_K^{n+1}\in G$ for all $K\in\mathcal T_h$ because the prepared reconstruction is $G$-admissible in the sense of \eqref{eq:admissible_recon}. Combining the two componentwise conclusions yields $u_h^{n+1}\in G_h(G,\Gint)$, which completes the proof.
\end{proof}

\begin{proof}[Proof of \Cref{cor:ssp_hybrid_stage}]
An SSP Runge--Kutta method can be expressed stage by stage as a convex combination of forward Euler steps. By \Cref{thm:hybrid_stage_IDP} or \Cref{cor:hybrid_stage_practical}, each such forward Euler step preserves the hybrid target set $G_h(G, G_{\mathrm{pt}})$.
For the cell averages, the convex combination is performed on the conservative variables; since $G$ is convex, the resulting cell average values remain in $G$. For the point values, the convex combination occurs in the transformed stage $w$-variables according to \eqref{eq:point_update_w}. Since the final point values are recovered via the admissibility-preserving map $\mathcal{B}$ whose codomain is $G_{\mathrm{pt}}$, the point-value admissibility is guaranteed by the mapping property itself, regardless of the convexity of $G_{\mathrm{pt}}$.
Consequently, each intermediate SSP stage satisfies $u_h^{n+1}\in G_h(G,G_{\mathrm{pt}})$.
For the classical third-order SSP Runge--Kutta method, the SSP coefficient equals unity, ensuring that the forward Euler conservative CFL condition \eqref{eq:CFL_abstract} remains sufficient, which completes the proof.
\end{proof}

\section{Geometric realizations and CFL conditions}\label{sec:geom}

This section constructs explicit third-order geometric realizations of the CAD in Assumption~(A3) on triangular elements, rectangular Cartesian cells, convex quadrilateral elements, and general convex polygons, together with their associated CFL bounds and limiter parameters.

\subsection{IDP CFL conditions formulated via CAD weights}\label{subsec:cfl-gcd}

The cell-average invariance analysis established in \Cref{sec:avg:proofB} relies on a CAD of the form \eqref{eq:gcd_point}, expressed as 
\begin{equation}\label{eq:CD-general-geom}
	\bar u_K^n \;=\;
	\sum_{e\in \mathcal{E}_K}\;\sum_{\nu=1}^{N_q}\gamma_{K,e,\nu}\,\hat u_{K,e,\nu}^n
	\;+\;
	\sum_{s=1}^{S_K}\beta_{K,s}\,\hat u_{K,s}^{n,\ast},	
\end{equation}
with 
\[
\gamma_{K,e,\nu}:=\lambda_{K,e} \omega_{e,\nu}\ge 0,\;\beta_{K,s}\ge 0,\qquad
\sum_{e,\nu}\gamma_{K,e,\nu}+\sum_s\beta_{K,s}
:=\sum_{e\in\mathcal{E}_K} \lambda_{K,e}\sum_{\nu=1}^{N_q}\omega_{e,\nu}+\sum_s\beta_{K,s}=1.
\]

Based on the discrete representation \eqref{eq:CD-general-geom}, \Cref{thm:B} ensures that the updated cell average satisfies $\bar u_K^{n+1} \in G$, provided that the time step satisfies
\begin{equation}\label{eq:CFL-gamma}
	\frac{\Delta t}{|K|}
	\;\le\;
    \min_{1\le \nu\le N_q} \frac{\lambda_{K,e}}{|e|\alpha_{K,e,\nu}}
    \;=\;
	\min_{1\le \nu\le N_q}\frac{\gamma_{K,e,\nu}}{|e|\,\omega_{e,\nu} \alpha_{K,e,\nu}},
	\qquad\forall e\in \mathcal{E}_K,
\end{equation}
where $\lambda_{K,e}$ in \eqref{eq:CFL_abstract} corresponds to $\gamma_{K,e,\nu} / \omega_{e,\nu}$.
Equivalently, by introducing the \emph{local IDP CFL indicator}
\begin{equation}\label{eq:CFL-indicator}
	\mathsf{CFL}_K^{\mathrm{IDP}}
    \;=\; \min_{e\in \mathcal{E}_K}
	\frac{\lambda_{K,e}}{|e|\, \max_{1\le \nu\le N_q}\alpha_{K,e,\nu}}
	\;:=\;
	\min_{e\in \mathcal{E}_K}\;\min_{1\le\nu\le N_q}
	\frac{\gamma_{K,e,\nu}}{|e|\,\omega_{e,\nu}\,\alpha_{K,e,\nu}} ,
\end{equation}
the time-step restriction \eqref{eq:CFL-gamma} can be recast as
\begin{equation}\label{eq:CFL-dt}
	\Delta t \;\le\; |K|\,\mathsf{CFL}_K^{\mathrm{IDP}}.
\end{equation}

\subsection{Triangular element realization}\label{subsec:tri-ocad}

Let $K$ be a nondegenerate triangular element with edges $\{e_i\}_{i=1}^3$, edge lengths
$l_K^{(i)}:=|e_i|$, and outward unit normals $n_{K,i}$.
On each edge $e_i$, we employ the three-point Gauss--Lobatto (GL) quadrature rule for the edge average:
\[
\omega^{\mathrm{GL}}_1=\omega^{\mathrm{GL}}_3=\frac16,\qquad
\omega^{\mathrm{GL}}_2=\frac23,\qquad
\sum_{\nu=1}^3\omega_\nu^{\mathrm{GL}}=1,
\]
with nodes $\{x_{K,i,\nu}\}_{\nu=1}^3$ corresponding to the two vertices and the midpoint of $e_i$.

\subsubsection{Optimal CAD (OCAD) boundary weights}\label{sec:OCADTri}

We adopt the separable edgewise decomposition
\begin{equation}\label{eq:tri-separable}
	\gamma_{K,i,\nu} \;=\; \lambda_{K,i}\,\omega_\nu^{\mathrm{GL}},\qquad i=1,2,3,\;\nu=1,2,3,
\end{equation}
where $\lambda_{K,i}>0$ are the edge weights, $\sum_{i=1}^3 \lambda_{K,i}<1$, and the remaining
interior weight $1-\sum_i \lambda_{K,i}$ is assigned to a set of internal quadrature points
$\{x^\star_{K,s}\}_{s=1}^{S_K}$ (e.g.,  $S_K=2$) with positive weights $\{\beta_{K,s}\}$.

On $\mathbb{P}^2(K)$ with GL boundary nodes, we determine the edge weights $\lambda_{K,i}$ according to
\begin{equation}\label{eq:tri-ocad-w}
	\lambda_{K,i}
	\;=\;
	\frac{2\,|e_i|}{9\,\bar l_K + 3\,\hat l_K},
	\qquad i=1,2,3,
\end{equation}
where
\begin{equation}\label{eq:tri-lbar-lhat}
	\begin{aligned}
	\bar l_K &:= \frac{|e_1|+|e_2|+|e_3|}{3},
	\\
	\hat l_K &:= \sqrt{|e_1|^2+|e_2|^2+|e_3|^2-\frac{2}{3}\Big(|e_1||e_2|+|e_2||e_3|+|e_3||e_1|\Big)}.
	\end{aligned}
\end{equation}
We select the interior nodes $\{x^\star_{K,s}\}$ and weights $\{\beta_{K,s}\}$ to ensure exact integration on the polynomial space $(\mathbb{P}^2(K))^m$ following \cite{ding2025TriOCAD}.

\begin{proposition}[Triangular realization of Assumption~(A3)]\label{prop:tri-A3}
For every triangular element $K$, the decomposition defined by \eqref{eq:tri-separable} together with the weights \eqref{eq:tri-ocad-w} and the interior completion from \cite{ding2025TriOCAD} is a feasible cell average decomposition on $\mathbb{P}^2(K)$ in the sense of Assumption~(A3). In particular, the coefficients are strictly positive, sum to unity, and the pointwise convex decomposition \eqref{eq:gcd_point} holds on the prescribed Gauss--Lobatto edge nodes.
\end{proposition}

\begin{proof}
The assertion follows directly from the OCAD construction of \cite{ding2025TriOCAD}, which provides a feasible CAD on $\mathbb P^2(K)$ with the stated boundary weights and positive interior completion, thereby verifying Assumption~(A3).
\end{proof}

\subsubsection{Associated CFL condition for triangular elements}\label{sec:tri-cfl-sub}
According to \eqref{eq:tri-separable}, the CFL indicator \eqref{eq:CFL-indicator} reduces to
\begin{equation}\label{eq:tri-cfl}
	\mathsf{CFL}^{\mathrm{IDP}}_K
	\;=\;
	\min_{1\le i\le 3}\frac{\lambda_{K,i}}{|e_i|\,\alpha_{K,e_i}},
	\qquad
	\text{since}\quad
	\frac{\gamma_{K,i,\nu}}{|e_i|\omega^{\mathrm{GL}}_\nu}
	=\frac{\lambda_{K,i}}{|e_i|}\;\;\text{for all }\nu,
\end{equation}
{where $\alpha_{K,e_i}=\max_{1\leq \nu \leq N_q} \alpha_{K,e_i, \nu}$ denotes the maximum wave speed on the edge $e_i$.}
Hence, \Cref{thm:B} guarantees that $\bar u_K^{n+1}\in G$ under the time-step condition
\begin{equation}\label{eq:tri-dt}
	\Delta t
	\;\le\;
	|K|\,
	\min_{1\le i\le 3}\frac{\lambda_{K,i}}{|e_i|\,\alpha_{K,e_i}}.
\end{equation}
In particular, when adopting the optimal weights \eqref{eq:tri-ocad-w} , the resulting CFL condition reduces to
\[
\Delta t
\;\le\; \frac{2 |K|}{9\,\bar l_K + 3\,\hat l_K}
\,
\frac{1}{\max_{1\le i\le 3} \alpha_{K,e_i}} .
\]

\begin{corollary}[Instantiation of \Cref{thm:B} on triangular meshes]\label{cor:tri-realization}
Assume (A1), (A2), and (A5). For the third-order triangular realization described above, Assumption~(A3) is established by \Cref{prop:tri-A3}. Consequently, the conservative update on triangular meshes satisfies the cell-average IDP property of \Cref{thm:B} under the triangular CFL condition \eqref{eq:tri-dt}.
\end{corollary}

\subsection{Rectangular (Cartesian) element realization}\label{sec:OCADCar}
For structured Cartesian meshes, the OCAD construction provides an \emph{optimal} IDP CFL condition, as established in \cite{CuiDingWu2024}. Let $|e_1|=|e_3|=\dy$ and $|e_2|=|e_4|=\dx$. Here, $\bar u_K^{(i)}$ denotes the average over the edge $e_i$, $\hat u_K^{*,i}$ represents the midpoint value on $e_i$, and $\hat u_K^{\mathrm{ctr}}$ the cell-center value; see \Cref{fig:mesh2}. The OCAD on the tensor-product polynomial space $\mathbb{Q}^2(K)$ is formulated as
\begin{equation}\label{Car:OCAD}
	\begin{aligned}
		\bar{u}_{K}^n&=
		\frac{1+\Theta}{12} \big( \bar{u}_K^{(1)} + \bar{u}_K^{(3)} \big) + \frac{1-\Theta}{12} \big( \bar{u}_K^{(2)} + \bar{u}_K^{(4)} \big)  \\
		&+ \frac23 \omega_{2}^{\rm GL} \hat{u}_{K}^{\mathrm{ctr}} +
		\frac{1-\Theta}{3} 
		\big(\omega_{1}^{\rm GL} (\hat{u}_{K}^{*,1} + \hat{u}_{K}^{*,3})\big)
		+
		\frac{1+\Theta}{3} \big(\omega_{1}^{\rm GL} (\hat{ u}_{K}^{*,2} + \hat{u}_{K}^{*,4})\big)
	\end{aligned}
\end{equation}
where
\begin{equation}\label{Theta}
	\Theta=\frac{\alpha_{x}/\dx - \alpha_{y}/\dy}{\alpha_{x}/\dx + \alpha_{y}/\dy}
	\in (-1,1)
\end{equation}
Here, $\alpha_{x} = \max\{\alpha_{K,e_1}, \alpha_{K,e_3}\}$ and $\alpha_{y} = \max\{\alpha_{K,e_2}, \alpha_{K,e_4}\}$ denote the directional wave-speed bounds in the $x$- and $y$-directions, respectively. We also denote $\alpha_{K,e_i}=\max_{1\leq \nu \leq N_q} \alpha_{K,e_i, \nu}$ as the maximum wave speed along the edge $e_i$.
\begin{corollary}[Rectangular OCAD realization of the cell-average IDP property]\label{thm:Car0}
	Assume that Assumptions (A1), (A2), and (A5) in \Cref{sec:assumptions} hold, and let $\mathcal{T}_h$ be a rectangular Cartesian mesh with spatial grid spacings $\dx$ and $\dy$. The rectangular OCAD \eqref{Car:OCAD} realizes Assumption~(A3) on the tensor-product polynomial space $\mathbb{Q}^2(K)$.
	Let $\hat{u}_K^n = \Pi_K(u_K^n; \bar{u}_K^n)$ be a $G$-admissible reconstruction in the sense of \eqref{eq:admissible_recon}. If the time step satisfies the explicit CFL condition
	\begin{equation}\label{eq:1668}
		\dt \Big(\frac{\alpha_x}{\dx} + \frac{\alpha_y}{\dy}\Big) \leq \frac16,
	\end{equation}
	then the cell averages updated by the conservative forward-Euler step \eqref{eq:avg_update_FE} satisfy
	\[
	\bar{u}_K^{n+1} \in G \qquad \forall K \in \mathcal{T}_h.
	\]
\end{corollary}
\begin{proof}
	According to \Cref{thm:B}, the updated cell average satisfies $\bar{u}_{K}^{n+1} \in G$, provided that 
	\begin{equation*}
		\frac{\dt}{|K|} \leq 
		\min \Big\{ \frac{1+\Theta}{12 \dy \alpha_x}, \frac{1-\Theta}{12 \dx \alpha_y}\Big\}
		=
		\frac{1/6}{ \dx \dy \big(\alpha_x/\dx + \alpha_y/\dy\big)}.
	\end{equation*}
	Substituting $|K| = \dx \dy$, this bound reduces precisely to \eqref{eq:1668}, which concludes the proof.
\end{proof}

\subsection{Convex quadrilateral element realization}\label{subsec:quad-tensor}

Let $K$ be a shape-regular convex quadrilateral element with vertices $\{x_K^{i}\}_{i=1}^4$ ordered counterclockwise, and let 
\[
x_K^{\mathrm{ctr}}:=\frac{x_K^{1}+x_K^{2}+x_K^{3}+x_K^{4}}{4}, \qquad 
x_K^{*,i} := \frac{x_K^{i-1}+x_K^{i}}{2}, \quad i = 1,2,3,4,
\]
be the geometric center of element $K$ and the corresponding edge midpoints, respectively, with the cyclic convention $x_K^{0}:=x_K^{4}$.
The cell average over $K$ admits the representation (see \Cref{fig:mesh2})
\[
	\bar{u}_K = \sum_{i=1}^4 \frac{\Delta_i}{9S_\Delta} u(x_K^{i}) + \sum_{i=1}^{4} \frac{2(\Delta_i + \Delta_{i+1})}{9S_\Delta} u(x_K^{*,i}) + \frac{4}{9} u(x_K^{\mathrm{ctr}}),
\]
where $S_\Delta:=\Delta_1+\Delta_2+\Delta_3+\Delta_4$ and
\[
\Delta_1:=(x^{(2)}_K-x^{(1)}_K) \times (x^{(4)}_K-x^{(1)}_K),\qquad
\Delta_2:=(x^{(2)}_K-x^{(1)}_K) \times (x^{(3)}_K-x^{(2)}_K),
\]
\[
\Delta_3:=(x^{(3)}_K-x^{(4)}_K) \times (x^{(3)}_K-x^{(2)}_K),\qquad
\Delta_4:=(x^{(3)}_K-x^{(4)}_K) \times (x^{(4)}_K-x^{(1)}_K).
\]
We define the balancing parameter
\[
\Theta_K
:=
\frac12\left(
\max\!\left\{0,\frac{2(\Delta_2-\Delta_3)}{3S_\Delta}\right\}
+
\min\!\left\{\frac{2\Delta_1}{3S_\Delta},\frac{2\Delta_2}{3S_\Delta},\frac{\Delta_1+\Delta_2}{3S_\Delta}\right\}
\right).
\]
Consequently, we obtain the edgewise cell average decomposition
\[
\bar u_K
=
\sum_{i=1}^4 \lambda_{K,i}\,\bar u_K^{(i)}
+
\sum_{i=1}^4 \beta_{K,i}\,u(x_K^{*,i})
+
\frac49\,u(x_K^{\mathrm{ctr}}),
\]
where
\[
\lambda_{K,1}=\Theta_K,\quad
\lambda_{K,2}=\frac{2\Delta_2}{3S_\Delta}-\Theta_K,\quad
\lambda_{K,3}=\frac{2(\Delta_3-\Delta_2)}{3S_\Delta}+\Theta_K,\quad
\lambda_{K,4}=\frac{2\Delta_1}{3S_\Delta}-\Theta_K,
\]
\[
\beta_{K,1}=\beta_{K,3}
=
\frac{2(\Delta_1+\Delta_2)}{9S_\Delta}-\frac23\,\Theta_K,
\qquad
\beta_{K,2}=\beta_{K,4}
=
\frac{2(\Delta_3-\Delta_2)}{9S_\Delta}+\frac23\,\Theta_K.
\]
All coefficients above are strictly positive. Moreover, when $K$ reduces to a parallelogram, the weights simplify to
\[
\lambda_{K,1}=\lambda_{K,2}=\lambda_{K,3}=\lambda_{K,4}=\frac1{12},
\qquad
\beta_{K,1}=\beta_{K,2}=\beta_{K,3}=\beta_{K,4}=\frac1{18},
\]
and consequently
\[
\bar u_K
=
\frac1{12}\sum_{i=1}^4 \bar u_K^{(i)}
+
\frac1{18}\sum_{i=1}^4 u(x_K^{*,i})
+
\frac49\,u(x_K^{\mathrm{ctr}}).
\]

\begin{figure}[!htb]
	\centering
	\begin{tikzpicture}[scale=1.4]
		
		\coordinate (A) at (0,0);     
		\coordinate (B) at (3,0.2);   
		\coordinate (C) at (2.7,2.2); 
		\coordinate (D) at (0.2,2);   
		
		\draw[thick] (A)--(B)--(C)--(D)--cycle;
		
		\coordinate (u31) at (A);
		\coordinate (u33) at (B);
		\coordinate (u13) at (C);
		\coordinate (u11) at (D);
		
		\coordinate (u21) at ($(A)!0.5!(D)$); 
		\coordinate (u23) at ($(B)!0.5!(C)$); 
		\coordinate (u12) at ($(D)!0.5!(C)$); 
		\coordinate (u32) at ($(A)!0.5!(B)$); 
		
		\coordinate (u22) at ($(u21)!0.5!(u23)$); 
		
		\draw[dashed] (u21) -- (u23);  
		\draw[dashed] (u12) -- (u32);  
		
		\node[above left]  at (u11) {$\hat{u}(x_{K}^4)$};
		\node[above right] at (u13) {$\hat{u}(x_{K}^3)$};
		\node[below left]        at (u31) {$\hat{u}(x_{K}^1)$};
		\node[below right] at (u33) {$\hat{u}(x_{K}^2)$};
		
		\node[above]       at (u12) {$\hat{u}(x_{K}^{*,4})$};
		\node[right]       at (u23) {$\hat{u}(x_{K}^{*,3})$};
		\node[left]       at (u21) {$\hat{u}(x_{K}^{*,1})$};
		\node[below]       at (u32) {$\hat{u}(x_{K}^{*,2})$};
		
		\node[above right]       at (u22) {$\hat{u}(x_{K}^{\mathrm{ctr}})$};
		
		
		\draw[decorate,decoration={brace,raise=3pt}]
		($(u31)+(-1cm,0)$) -- ($(u11)+(-0.9cm,0)$)
		node[midway,xshift=-0.6cm] {$\bar u_K^{(1)}$};
		
		\draw[decorate,decoration={brace,raise=3pt}]
		($(u11)+(0,0.5cm)$) -- ($(u13)+(0,0.5cm)$)
		node[midway, yshift=0.6cm] {$\bar u_K^{(4)}$};
		
		\draw[decorate,decoration={brace,raise=3pt}]
		($(u33)+(0,-0.5cm)$) -- ($(u31)+(0,-0.5cm)$)
		node[midway, yshift=-0.6cm] {$\bar u_K^{(2)}$};
		
		\draw[decorate,decoration={brace,raise=3pt}]
		($(u13)+(0.9cm,0)$) -- ($(u33)+(1cm,0)$)
		node[midway, xshift=0.6cm] {$\bar u_K^{(3)}$};
		
	\end{tikzpicture}
	\caption{Quadrilateral element $K$ with centroid, edge-midpoint, and edge-average degrees of freedom.}
	\label{fig:mesh2}
\end{figure}
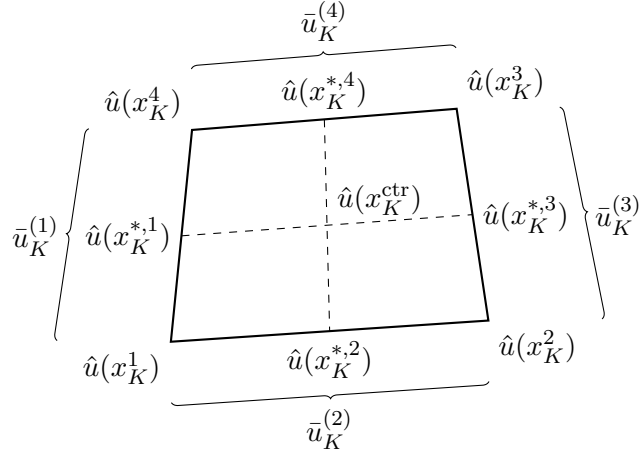

\begin{proposition}[Quadrilateral realization of Assumption~(A3)]\label{prop:quad-A3}
The decomposition formulated above defines a feasible cell average decomposition on shape-regular convex quadrilateral elements for the third-order polynomial space.
\end{proposition}

\begin{proof}
For a shape-regular convex quadrilateral element, each diagonal-area quantity $\Delta_i$ defined above is strictly positive. Moreover, the diagonal-area identities $\Delta_1+\Delta_3=\Delta_2+\Delta_4=2|K|$ yield $S_\Delta=4|K| >0.$
We define the lower and upper bounds
\[
L_K := 
\max\!\left\{0,\frac{2(\Delta_2-\Delta_3)}{3S_\Delta}\right\}, \qquad
U_K:=\min\!\left\{\frac{2\Delta_1}{3S_\Delta},\frac{2\Delta_2}{3S_\Delta},\frac{\Delta_1+\Delta_2}{3S_\Delta}\right\}=\frac{2}{3S_\Delta} \min\left\{\Delta_1,\Delta_2\right\}.
\]
We first verify that $0 \le L_K < U_K$.
The condition $\Delta_2 \le \Delta_3$ trivially yields $L_K = 0 < U_K$. For $\Delta_2 > \Delta_3$, the diagonal-area identities imply $\Delta_2 - \Delta_3 = \Delta_1 - \Delta_4 < \Delta_1$. Consequently, $\max\{0, \Delta_2 - \Delta_3\} < \min\{\Delta_1, \Delta_2\}$, which ensures that $L_K < U_K$. 
By the definition of $\Theta_K=\frac12(L_K+U_K)$, it follows that $0 \le L_K < \Theta_K < U_K$. Thus, the bounds built into the formulas for $\lambda_{K,i}$ and $\beta_{K,i}$ ensure $\lambda_{K,i} > 0$ and $\beta_{K,i} > 0$ for $i=1,\dots,4$.
Furthermore, using $S_\Delta=4|K|$, direct algebraic evaluation yields
\[
\sum_{i=1}^4 \lambda_{K,i}=\frac13,\qquad \sum_{i=1}^4 \beta_{K,i}=\frac29,\qquad \frac13+\frac29+\frac49=1.
\]
Hence, all CAD weights are strictly positive and sum to unity. Exactness follows by expanding the edge averages through Simpson's rule on each quadratic edge trace and comparing coefficients with the displayed point-value decomposition; the algebraic verification is detailed in Appendix~\ref{sec:quad_cad_appendix}. The vector-valued statement holds componentwise, which completes the proof.
\end{proof}

\begin{corollary}[Quadrilateral realization of the cell-average IDP property]\label{thm:1662}
	Assume that Assumptions (A1), (A2), and (A5) in \Cref{sec:assumptions} hold, and let $\mathcal{T}_h$ be a convex quadrilateral mesh. By \Cref{prop:quad-A3}, Assumption~(A3) is satisfied via the decomposition above.
	Let $\hat{u}_K^n=\Pi_K(u_K^n;\bar{u}_K^n)$ be a $G$-admissible reconstruction in the sense of \eqref{eq:admissible_recon}. 
	For the conservative forward-Euler update \eqref{eq:avg_update_FE}, 
	if the time step satisfies the explicit CFL condition
	\begin{equation}\label{eq:1667}
		\alpha_{K,e_i}\,\frac{\Delta t\,|e_i|}{|K|}\ \le\ \lambda_{K,i}
		\qquad i=1,2,3,4, \, \forall K\in\mathcal{T}_h,
	\end{equation}
	where $\alpha_{K,e_i} := \max_{1\leq \nu \leq N_q} \alpha_{K,e_i, \nu}$, then the updated cell averages satisfy
	\[
	\bar{u}_K^{n+1}\in G\qquad \forall K\in\mathcal{T}_h.
	\]
\end{corollary}

\subsection{General convex polygonal element realization with explicit CAD weights}\label{subsec:poly}

For a general convex polygonal element $K$ with vertices $v_1,\dots,v_{N_K}$ ordered counterclockwise and boundary edges $e_i=[v_i,v_{i+1}]$ (with the cyclic convention $v_{N_K+1}:=v_1$), we construct explicit boundary and interior CAD weights as follows.

Let $x_K^\star$ denote the centroid of $K$, and introduce the fan triangulation
\[
K_i := \mathrm{conv}\{x_K^\star,v_i,v_{i+1}\},\qquad i=1,\dots,N_K,
\qquad
K=\bigcup_{i=1}^{N_K}K_i,\quad \text{with disjoint interiors}.
\]
As illustrated in \Cref{fig:1761}. Let $n_i$ denote the outward unit normal to $e_i$, and let $h_i$ denote the perpendicular distance from $x_K^\star$ to the supporting line of $e_i$. With $\xi:=x-x_K^\star$, it holds that
\[
\xi\cdot n_i=h_i\quad \text{on } e_i,
\qquad
|K_i|=\frac12 |e_i|h_i.
\]

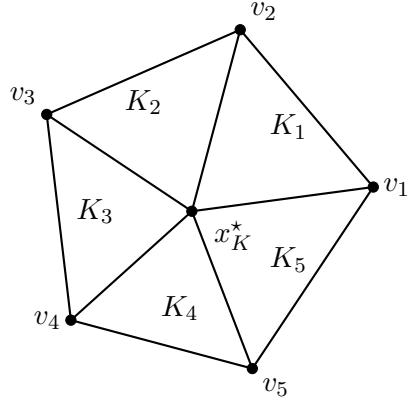
\begin{figure}[!htb]
	\centering
	\begin{center}
		
		\begin{tikzpicture}[scale=1.6]
			\draw [ thick] (0,0) -- (1.5,0.2);
			\draw [ thick] (1.5,0.2) -- (0.4,1.5) -- (0,0);
			\draw [ thick] (0,0) -- (-1.2,0.8) -- (0.4,1.5);
			\draw [ thick] (0,0) -- (-1.0,-0.9) -- (-1.2,0.8);
			\draw [ thick] (0,0) -- (0.5,-1.3) -- (-1.0,-0.9);
			\draw [ thick] (1.5,0.2) -- (0.5,-1.3);
			
			\filldraw (0,0) circle (1.2pt);
			\node[below right] at (0.1,0) {$x^\star_K$};
			
			\filldraw (1.5,0.2)   circle (1.2pt);
			\filldraw (0.4,1.5)   circle (1.2pt);
			\filldraw (-1.2,0.8)  circle (1.2pt);
			\filldraw (-1.0,-0.9) circle (1.2pt);
			\filldraw (0.5,-1.3)  circle (1.2pt);
			
			\node[right]      at (1.5,0.2)   {$v_1$};
			\node[above right]at (0.4,1.5)   {$v_2$};
			\node[above left] at (-1.2,0.8)  {$v_3$};
			\node[left]       at (-1.0,-0.9) {$v_4$};
			\node[below right]at (0.5,-1.3)  {$v_5$};
			
			\node at (0.8,0.7)   {$K_1$}; 
			\node at (-0.4,0.9)  {$K_2$}; 
			\node at (-0.8,0.0)  {$K_3$}; 
			\node at (-0.1,-0.8) {$K_4$}; 
			\node at (0.8,-0.4)  {$K_5$}; 
			
		\end{tikzpicture}
		
	\end{center}
	\caption{Fan triangulation of a pentagonal ($N_K=5$) cell $K$.}
	\label{fig:1761}
\end{figure}

For any $\phi\in\mathbb{P}^2(K)$, we expand it about the centroid $x_K^\star$:
\[
\phi(x_K^\star+\xi)=\phi_0+\phi_1(\xi)+\phi_2(\xi),
\]
where $\phi_j$ is a homogeneous polynomial of degree $j$ in $\xi$, and $\phi_0=\phi(x_K^\star)$. For a homogeneous polynomial $q_j$ of degree $j$, Euler's identity yields $\nabla_\xi\cdot\bigl(\xi q_j(\xi)\bigr)=(j+2)q_j(\xi).$ Hence, applying the divergence theorem yields,
\[
\int_{\partial K}(\xi\cdot n)q_j\,ds
=
(j+2)\int_K q_j\,dx .
\]
Since $x_K^\star$ is the area centroid,
\[
\int_K \xi\,dx=0,
\qquad
\text{which implies}
\qquad
\int_K \phi_1\,dx=0 .
\]
Consequently,
\[
\sum_{i=1}^{N_K} h_i\int_{e_i}\phi\,ds
=
\int_{\partial K}(\xi\cdot n)\phi\,ds
=
2|K|\phi_0+4\int_K\phi_2\,dx .
\]
Meanwhile,
\[
\int_K\phi\,dx
=
|K|\phi_0+\int_K\phi_2\,dx .
\]
Eliminating $\int_K\phi_2\,dx$ leads to the exact identity
\begin{equation}
\label{eq:edge-focused-tri}
\bar\phi_K
=
\frac{1}{4|K|}
\sum_{i=1}^{N_K} h_i\int_{e_i}\phi\,ds
+
\frac12 \phi(x_K^\star)
=
\sum_{i=1}^{N_K}
\frac{|K_i|}{2|K|}\,\bar\phi_{e_i}
+
\frac12 \phi(x_K^\star).
\end{equation}

Applying the three-point GL quadrature along each edge $e_i$ yields
\[
\bar \phi_{e_i}
=
\sum_{\nu=1}^3\omega^{\mathrm{GL}}_\nu\,\phi(x_{K,i,\nu}),
\qquad
x_{K,i,1}=v_i,\; x_{K,i,2}=\tfrac12(v_i+v_{i+1}),\; x_{K,i,3}=v_{i+1}.
\]
Thus, for any $\phi\in \mathbb{P}^2(K)$, we arrive at the discrete decomposition
\begin{equation}\label{eq:poly-CD-final}
	\bar \phi_K
	=
	\sum_{i=1}^{N_K}\sum_{\nu=1}^3
	\underbrace{\left(\frac{|K_i|}{|K|}\cdot\frac{1}{2}\,\omega^{\mathrm{GL}}_\nu\right)}_{=:~\gamma_{K,i,\nu}}
	\,\phi(x_{K,i,\nu})
	\;+\;
	\underbrace{\frac12}_{=:~\beta_{K}}
	\,\phi(x_K^\star).
\end{equation}
Equivalently, this construction satisfies Assumption~(A3) with a single interior quadrature node:
\[
S_K=1,
\qquad
x_{K,1}^\star=x_K^\star,
\qquad
\beta_{K,1} = \beta_K =\frac12,
\qquad
\lambda_{K,e_i}:=\sum_{\nu=1}^3\gamma_{K,i,\nu}
=
\frac{|K_i|}{2|K|}.
\]

\begin{proposition}[Convex-polygon realization of Assumption~(A3)]\label{prop:poly-A3}
	For every convex polygonal element $K$, the construction \eqref{eq:poly-CD-final} defines a feasible cell average decomposition for the third-order polynomial space. The weights $\gamma_{K,i,\nu}$ and $\beta_{K}$ are strictly positive, sum to unity, and realize the cell average through boundary quadrature values and a single explicitly constructed interior node $x_K^\star$.
\end{proposition}

\begin{proof}
	The strict positivity of the weights follows directly from the areas $|K_i|>0$ of the fan triangles and $\beta_{K} = 1/2 > 0$. The partition-of-unity property follows from $\sum_{\nu=1}^3 \omega^{\mathrm{GL}}_\nu = 1$ and $\sum_{i=1}^{N_K} |K_i| = |K|$, which yields
	$$ \sum_{i=1}^{N_K}\sum_{\nu=1}^3 \gamma_{K,i,\nu} + \beta_{K} = \sum_{i=1}^{N_K} \frac{|K_i|}{2|K|} + \frac{1}{2} = \frac{1}{2} + \frac{1}{2} = 1. $$
    The exactness is guaranteed by the identity \eqref{eq:edge-focused-tri} and the exactness of the three-point GL rule for quadratic polynomials on each straight edge, completing the proof.
\end{proof}

\begin{proposition}[CFL condition for general convex polygonal elements]\label{prop:poly-cfl}
	Let $\gamma_{K,i,\nu}$ be given by \eqref{eq:poly-CD-final} and let $\alpha_{K,e_i}=\max_{1\leq \nu \leq 3} \alpha_{K,e_i, \nu}$. Then the CFL condition
	\eqref{eq:CFL-gamma} reduces to
	\begin{equation}\label{eq:poly-cfl}
		\frac{\Delta t}{|K|}\,\alpha_{K,e_i}
		\;\le\;
		\frac{|K_i|}{2|K|\,|e_i|},
		\qquad i=1,\dots,N_K,
	\end{equation}
	for each polygonal boundary edge $e_i$, which implies the explicit time-step restriction
	\begin{equation}\label{eq:poly-dt}
		\Delta t
		\;\le\;
		\min_{1\le i\le N_K}\frac{|K_i|}{2\,|e_i|\,\alpha_{K,e_i}}.
	\end{equation}
\end{proposition}

\begin{remark}[Alternative positive quadratic quadrature on polygons]
    If a polygon has many edges, the fan construction may involve a large number of sub-triangles. This motivates the study of alternative positive-weight quadrature rules that are exact for quadratic polynomials on general polygonal cells. Such rules were constructed in \cite{abgrall2025some} using a Virtual Element Method (VEM) interpretation. They require only one additional interior point, chosen as the maximizer of a positive bubble function obtained from an elliptic problem. Since this point depends on the geometry of the polygon, no simple analytic expression is available in general; nevertheless, it can be accurately approximated by standard finite element techniques.
\end{remark}

\section{Numerical experiments}\label{sec:numerics}

\subsection{Numerical setup}

The experiments serve two objectives. \revblue{Smooth scalar tests assess whether the \emph{a priori} scaling remains inactive, or, when it activates, whether the designed high-order accuracy is retained.} Discontinuous scalar and Euler tests validate the IDP property in regimes where loss of scalar bounds or \revblue{loss of strict positivity of density or pressure} would be immediately apparent. We report final-time admissibility, all-stage minima, and limiter activation fractions to distinguish invariant domain preservation from the extent of active numerical stabilization.

In our simulations, the triangular meshes are generated using EasyMesh~\cite{Niceno2002easymesh}, while the Cartesian meshes are uniform. We denote by $N$ the number of cells and by $h_{\mathcal T}$ the characteristic mesh size, taken as the uniform subdivision length along cell edges. Point values are advanced with $\mathcal D^{\mathrm{pt}}$ in \eqref{eq:pt_rhs_vertex}--\eqref{eq:pt_midpoint_scheme}. For the scalar tests, we use \eqref{u:scalar} and \eqref{w:scalar} in the discontinuous cases, and \eqref{u:scalar2}--\eqref{w:scalar2} in the smooth cases on $(U_{\min}-\delta,U_{\max}+\delta)$ with $\delta=0.01$. The Euler tests use \eqref{w:euler1}--\eqref{w:euler2}. We employ $\Pi_K^\texttt{a}$ on triangular meshes and $\Pi_K^\texttt{c}$ on rectangular meshes; implementation details for $\Pi_K$ are deferred to Appendix~\ref{sec:IDP_imple}. Furthermore, we deactivate the oscillation-eliminating (OE) procedure $\mathcal M_K$ from Appendix~\ref{sec:OE} exclusively in \Cref{Ex:Smooth,Ex:BurgSmooth}.  \Cref{tab:numerics-purpose} summarizes the numerical configurations and diagnostics. 
Enabling the $\mathcal M_K$ operator is referred to as ``OE'' in \Cref{tab:numerics-purpose}.

The numerical interface flux is chosen as the local Lax--Friedrichs flux \eqref{eq:LF_flux} with numerical dissipation coefficient
\begin{equation}\label{7.alpha}
	\alpha^{\mathrm{LF}}_{K,e,\nu}:=
	\max\Big\{\varrho\big(A_{n_{K,e}}(\hat u^n_{K,e,\nu})\big),\ 
	\varrho\big(A_{n_{K,e}}(\hat u^n_{K_e,e,\nu})\big)\Big\},
\end{equation}
where $\alpha^{\mathrm{LF}}_{K,e}:=\max_{\nu}\alpha^{\mathrm{LF}}_{K,e,\nu}$.

We adopt the computable proxy CFL condition
\begin{equation}\label{eq:CFL-used}
	\Delta t=\min_{K\in\mathcal{T}_h}\Delta t_K^{\mathrm{IDP}},
	\qquad
	\Delta t_K^{\mathrm{IDP}}
	:=c_0\,|K|\,
	\min_{e\in\mathcal{E}_K}
	\frac{\lambda_{K,e}}{|e|\,\tilde{\alpha}^{\,n}_{\max}},
\end{equation}
where $\tilde{\alpha}^{\,n}_{\max}$ denotes the proxy wave speed, $\{\lambda_{K,e}\}$ are the boundary weights from Section~\ref{sec:geom}, and $c_0=1$ for the local Lax--Friedrichs flux. These cell-average proxies are analytically compatible with trace-based wave speeds on smooth solutions, as justified in Appendix~\ref{sec:proxy_bridge}. 
Nevertheless, the proxy CFL employed in computations is not the rigorous trace-based sufficient condition of \Cref{thm:B} in nonsmooth regimes; admissibility is validated \emph{a posteriori} by the reported diagnostics.
\begin{itemize}
	\item On triangular meshes,
	\[
	\Delta t_K^{\mathrm{IDP}} = \frac{2 |K|}{\tilde{\alpha}^{\,n}_{\max} (9 \bar l_K + 3 \hat l_K)}, 
	\qquad
	\tilde{\alpha}^{\,n}_{\max} = \max_{K, e} \varrho\big(A_{n_{K,e}}(\bar u^n_{K})\big);
	\]
	\item On Cartesian meshes,
	\[
	\Delta t = \frac{C_{\mathrm{CFL}}}{\tilde{\alpha}^{\,n}_{1,\max}/\dx + \tilde{\alpha}^{\,n}_{2,\max}/\dy},
	\qquad
	\tilde{\alpha}^{\,n}_{\ell,\max} = \max_{K} \alpha_\ell(\bar{u}_K^n), \ \ell = 1,2,
	\]
	with $C_{\mathrm{CFL}}=0.15<1/6$.
\end{itemize}
Temporal discretization is performed using the classical three-stage third-order SSP Runge--Kutta method.

For scalar problems with the invariant domain $G=[U_{\min},U_{\max}]$, the discrete invariant domain preservation is quantified by
\[
\delta_{\max}(\bar u_h)=U_{\max}-\max_{K \in \mathcal{T}_h}\bar u_K, \qquad
\delta_{\min}(\bar u_h)=\min_{K \in \mathcal{T}_h}\bar u_K-U_{\min}.
\]

\begin{table}[!htb]
\centering
\caption{{Numerical tests, mesh families, and primary diagnostics.}}
\label{tab:numerics-purpose}
\scriptsize
\setlength{\tabcolsep}{3pt}
\begin{tabularx}{\textwidth}{>{\raggedright\arraybackslash}p{0.15\textwidth}>{\raggedright\arraybackslash}p{0.11\textwidth}>{\raggedright\arraybackslash}p{0.12\textwidth}>{\centering\arraybackslash}p{0.06\textwidth}>{\raggedright\arraybackslash}p{0.21\textwidth}>{\raggedright\arraybackslash}X}
\toprule
Test & PDE and regime & Mesh family & OE & Focus & Main diagnostics \\
\midrule
Smooth linear convection & scalar smooth & triangular, Cartesian & off & smooth accuracy & $L^1$ errors, observed orders, and average limiter fraction $\Lambda$. \\
Discontinuous linear convection & scalar discontinuous & triangular, Cartesian & on & scalar admissibility diagnostics & Cell-average admissibility diagnostics with and without the \emph{a priori} limiter, and profiles. \\
Smooth Burgers & scalar smooth nonlinear & triangular, Cartesian & off & smooth nonlinear accuracy & $L^1$ errors, observed orders, and average limiter fraction $\Lambda$. \\
Oblique Burgers Riemann & scalar discontinuous nonlinear & triangular, Cartesian & on & nonsmooth scalar diagnostics & Cell-average admissibility diagnostics and profiles. \\
Sedov blast & Euler near vacuum & Cartesian & on & near-vacuum Euler diagnostics & final-time minima; {per-stage minima}; limiter diagnostics; contours; diagonal cuts. \\
Mach 80 and Mach 2000 jets & Euler high-Mach jets & Cartesian & on & high-Mach jet diagnostics & final-time minima; {per-stage minima}; limiter diagnostics; density and pressure contours. \\
{Shock diffraction; shock reflection and diffraction} & Euler irregular-domain benchmarks & triangular & on & irregular-domain Euler diagnostics & final-time minima; {per-stage minima}; limiter diagnostics; density contours. \\
\bottomrule
\end{tabularx}
\end{table}

\subsection{Accuracy-preservation and scalar diagnostics}\label{subsec:conv}
\subsubsection{Smooth linear convection: accuracy-preservation test} \label{Ex:Smooth}
We first examine the two-dimensional linear advection equation
\begin{equation}\label{eq:linear}
	u_t+ u_x + u_y=0,\qquad (x,y) \in \Omega = [0,1]^2,
\end{equation}
subject to periodic boundary conditions.
The initial condition is prescribed by
$
u(x, y, 0) = \sin^4(2\pi(x+y)),
$
which gives the exact solution
$
u(x, y, t) = \sin^4(2\pi(x+y-2t)).
$
This exact solution inherently satisfies the maximum principle, i.e., $u(x, y, t) \in G=[0, 1]$ for all $t \ge 0$.
The computation is carried out up to the final time $t=1$.
The convergence study is performed on uniform Cartesian meshes and six unstructured triangular meshes successively refined as in \cite[Example~5.1]{ding2025TriOCAD}.

Table~\ref{tab:Ex-LinSmooth} lists the $L^1$ errors for the cell averages $\bar u_h$ and the point values $u_\Sigma$, together with the average limiter activation fraction 
$
\Lambda := \frac{N_{\tt IDP}}{3NN_t}\times 100\%,
$ 
where $N_{\tt IDP}$ denotes the total number of limiter activations, $N$ is the number of cells, and $N_t$ denotes the number of time steps. 
These results confirm that the limiter $\Pi_h$ preserves the formal design order of accuracy; the reported activation fractions remain very small in the smooth regime, in agreement with the asymptotic estimates in \Cref{cor:limiter_small_smooth}.

\begin{table}[!htb] 
	\centering
	\belowrulesep=0pt
	\aboverulesep=0pt
	\caption{Smooth linear convection: errors and limiter fraction.}
	\label{tab:Ex-LinSmooth}
	\setlength{\tabcolsep}{3mm}{
		\begin{tabular}{c|clclcc}
			\toprule[1.5pt]
			\multirow{2}{*}{ Mesh } &
			\multirow{2}{*}{$ N $} &
			\multicolumn{2}{c}{cell averages} &
			\multicolumn{2}{c}{point values}  &
			\multirow{2}{*}{$ \Lambda (\%)$} \\
			\cmidrule(r){3-4} \cmidrule(r){5-6}
			&
			& $ L^{1} $ error & order & $ L^{1} $ error & order & \\

			\midrule[1.5pt]
			\multirow{6}{*}{ Triangular}
			& 176    & 4.09e-2& -    & 5.38e-2& -    & 1.9808 \\
			& 704    & 2.16e-2& 0.92 & 2.59e-2& 1.06 & 2.6697 \\
			& 2816   & 8.07e-3& 1.42 & 8.77e-3& 1.56 & 7.6703 \\
			& 11264  & 2.04e-3& 1.99 & 2.17e-3& 2.01 & 6.2568 \\
			& 45056  & 3.34e-4& 2.61 & 3.48e-4& 2.64 & 3.7128 \\
			& 180224 & 3.92e-5& 3.09 & 4.01e-5& 3.12 & 1.9330 \\
			
			\midrule[1.5pt]
			\multirow{6}{*}{ Cartesian  }
			&$20^2$     & 1.82e-02& -    & 2.65e-02& -    & 1.2235 \\
			&$40^2$     & 1.34e-02& 0.44 & 1.51e-02& 0.81 & 6.6479 \\
			&$80^2$     & 3.27e-03& 2.04 & 3.53e-03& 2.10 & 8.5372 \\
			&$160^2$    & 6.13e-04& 2.42 & 6.47e-04& 2.45 & 6.4732 \\
			&$320^2$    & 9.07e-05& 2.76 & 9.27e-05& 2.80 & 3.2843 \\
			&$640^2$    & 1.19e-05& 2.93 & 1.20e-05& 2.95 & 1.6537 \\

			\bottomrule[1.5pt]
		\end{tabular}
	}
\end{table}

\subsubsection{Discontinuous linear convection: IDP diagnostics} \label{Ex:LinDis}
We next consider a discontinuous initial condition on $\Omega=[-1,1]^2$
\begin{equation*}
	u(x,y,0) =
	\begin{cases}
		1, & \sqrt{x^2+y^2} \leq r_\star(\theta), \\
		0, & \text{otherwise},
	\end{cases}
	\quad
	\theta =
	\begin{cases}
		\arccos{\frac{x}{\sqrt{x^2+y^2}}}, & y \geq 0, \\
		2\pi - \arccos{\frac{x}{\sqrt{x^2+y^2}}}, & y < 0,
	\end{cases}
\end{equation*}
where $r_\star(\theta):=\frac{1}{8}\bigl(3+3^{\sin(5\theta)}\bigr)$.
The profile is advected according to \eqref{eq:linear} with periodic boundary conditions. The invariant domain is $G=[0,1]$, and the simulation is carried out to the final time $t=1.8$. Computations are performed on both a uniform Cartesian mesh with $200 \times 200$ cells and an unstructured triangular mesh with $N=92{,}474$ cells and characteristic mesh size $h_{\mathcal T}=0.01$.

\Cref{tab:Ex-LinDis} summarizes the cell-average admissibility diagnostics $\delta_{\max}(\bar u_h)$ and $\delta_{\min}(\bar u_h)$ with and without the \emph{a priori} limiter $\Pi_h$. With $\Pi_h$ active, the diagnostics remain strictly nonnegative up to machine roundoff; without $\Pi_h$, admissibility violations are observed on both meshes.

\begin{table}[!htbp] 
	\centering
	\belowrulesep=0pt
	\aboverulesep=0pt
	\caption{Discontinuous linear convection: cell-average diagnostics.}
	\label{tab:Ex-LinDis}
	\setlength{\tabcolsep}{2mm}{
		\begin{tabular}{c|cccc}
			\toprule[1.5pt]
			\multirow{2}{*}{Mesh} & \multicolumn{2}{c}{With limiter $\Pi_h$} & \multicolumn{2}{c}{Without limiter $\Pi_h$}  \\
			\cmidrule(r){2-3} \cmidrule(r){4-5}
			& $\delta_{\max}(\bar u_h)$ & $\delta_{\min}(\bar u_h)$ & $\delta_{\max}(\bar u_h)$ & $\delta_{\min}(\bar u_h)$ \\
			\midrule[1.5pt]
			Cartesian & 1.1102e-14 & 7.2393e-36 & -8.8193e-8 & -1.2370e-7 \\
			Triangular & 8.8819e-16 & 6.2748e-34 & -2.1716e-7 & -3.6127e-7 \\
			\bottomrule[1.5pt]
		\end{tabular}
	}
\end{table}

\subsubsection{Burgers' equation: smooth nonlinear accuracy test} \label{Ex:BurgSmooth}
We solve the 2D smooth nonlinear Burgers' equation
$
u_t+\frac{1}{2}(u^2)_x+\frac{1}{2}(u^2)_y=0$ on $(x,y)\in[0,1]^2,
$
with periodic boundary conditions and the smooth initial condition $u(x,y,0)=2+0.5\sin(2\pi(x+y))$. The invariant domain is $G=[1.5,2.5]$. Because the exact solution remains smooth until $t=(2\pi)^{-1}\approx0.159$, the simulation is performed up to the final time $t=0.1$. The unstructured triangular meshes are identical to those in \Cref{Ex:Smooth}. 

\Cref{tab:Ex-BurSmooth} lists the $L^1$ errors, observed orders, and the average limiter activation fraction $\Lambda$.  This smooth nonlinear test confirms that the a priori limiter remains essentially inactive away from the boundary of the invariant domain.

\begin{table}[!htbp] 
	\centering
	\belowrulesep=0pt
	\aboverulesep=0pt
	\caption{Smooth Burgers: errors and limiter fraction.}
	\label{tab:Ex-BurSmooth}
	\setlength{\tabcolsep}{3mm}{
		\begin{tabular}{c|cccccc}
			\toprule[1.5pt]
			\multirow{2}{*}{Mesh} &
			\multirow{2}{*}{$N$} &
			\multicolumn{2}{c}{cell averages} &
			\multicolumn{2}{c}{point values} & \multirow{2}{*}{$\Lambda (\%)$} \\
			\cmidrule(r){3-4} \cmidrule(r){5-6}
			& & $L^1$ error & order & $L^1$ error & order & \\
			\midrule[1.5pt]
			\multirow{6}{*}{ Triangular}
			& 176    & 9.89e-3& -    & 1.32e-2& -    & 0.0303 \\
			& 704    & 3.60e-3& 1.46 & 4.08e-3& 1.69 & 0.0169 \\
			& 2816   & 8.31e-4& 2.12 & 9.43e-4& 2.11 & 0.0083 \\
			& 11264  & 1.35e-4& 2.62 & 1.44e-4& 2.71 & 0.0050 \\
			& 45056  & 1.98e-5& 2.77 & 2.06e-5& 2.81 & 0.0047 \\
			& 180224 & 2.62e-6& 2.92 & 2.68e-6& 2.94 & 0.0036 \\
			
			\midrule[1.5pt]
			\multirow{6}{*}{ Cartesian  }
			& $20^2$  & 2.38e-3& -    & 3.50e-3& -    & 0.0299 \\
			& $40^2$  & 4.78e-4& 2.32 & 5.73e-4& 2.66 & 0.0137 \\
			& $80^2$  & 5.69e-5& 3.07 & 6.30e-5& 3.21 & 0.0131 \\
			& $160^2$ & 6.31e-6& 3.17 & 6.69e-6& 3.25 & 0.0196 \\
			& $320^2$ & 7.55e-7& 3.06 & 8.17e-7& 3.04 & 0.0199 \\
			& $640^2$ & 9.28e-8& 3.02 & 1.17e-7& 2.81 & 0.0167 \\
			\bottomrule[1.5pt]
		\end{tabular}
	}
\end{table}

\subsubsection{Burgers' equation: oblique Riemann problem} \label{Ex:BurgRM}
We consider  the oblique Riemann problem proposed in \cite{ding2025TriOCAD} for the inviscid Burgers' equation with the piecewise constant initial data
\begin{equation*}
	u(x,y,0) = 
	\begin{cases}
		-0.2, & x < 0.5, \, y \ge 0.5, \\
		-1,   & x \ge 0.5, \, y \ge 0.5, \\
		0.5,  & x < 0.5, \, y < 0.5, \\
		0.8,  & x \ge 0.5, \, y < 0.5.
	\end{cases}
\end{equation*} 
The invariant domain is $G=[-1,0.8]$. For $t \ge 0$, outflow boundary conditions are imposed along the segments $\{x = 0, \, y \geq 0.5 + 0.15t\}$ and $\{x = 1, \, y \leq 0.5 - 0.1t\}$, while inflow conditions are prescribed on the remainder of $\partial\Omega$. The simulation is advanced to the final time $t = 0.5$. The computational domain is discretized using a $256 \times 256$ uniform Cartesian mesh and an unstructured triangular mesh consisting of $N = 151{,}742$ cells with characteristic mesh size $h_{\mathcal{T}} = 1/256$.

\Cref{fig:Ex_BurgRM} displays the numerical solutions together with 1D profiles along the diagonal cut $y = 1 - x$. \Cref{tab:Ex-Burg} lists the cell-average admissibility diagnostics $\delta_{\max}(\bar u_h)$ and $\delta_{\min}(\bar u_h)$ with and without the \emph{a priori} limiter $\Pi_h$. With $\Pi_h$ active, the diagnostics remain strictly nonnegative; without $\Pi_h$, admissibility violations occur on both meshes.

\begin{figure}[!htb]
	\centering
	\begin{subfigure}{0.42\textwidth}
		\centering
		\includegraphics[width=\textwidth]{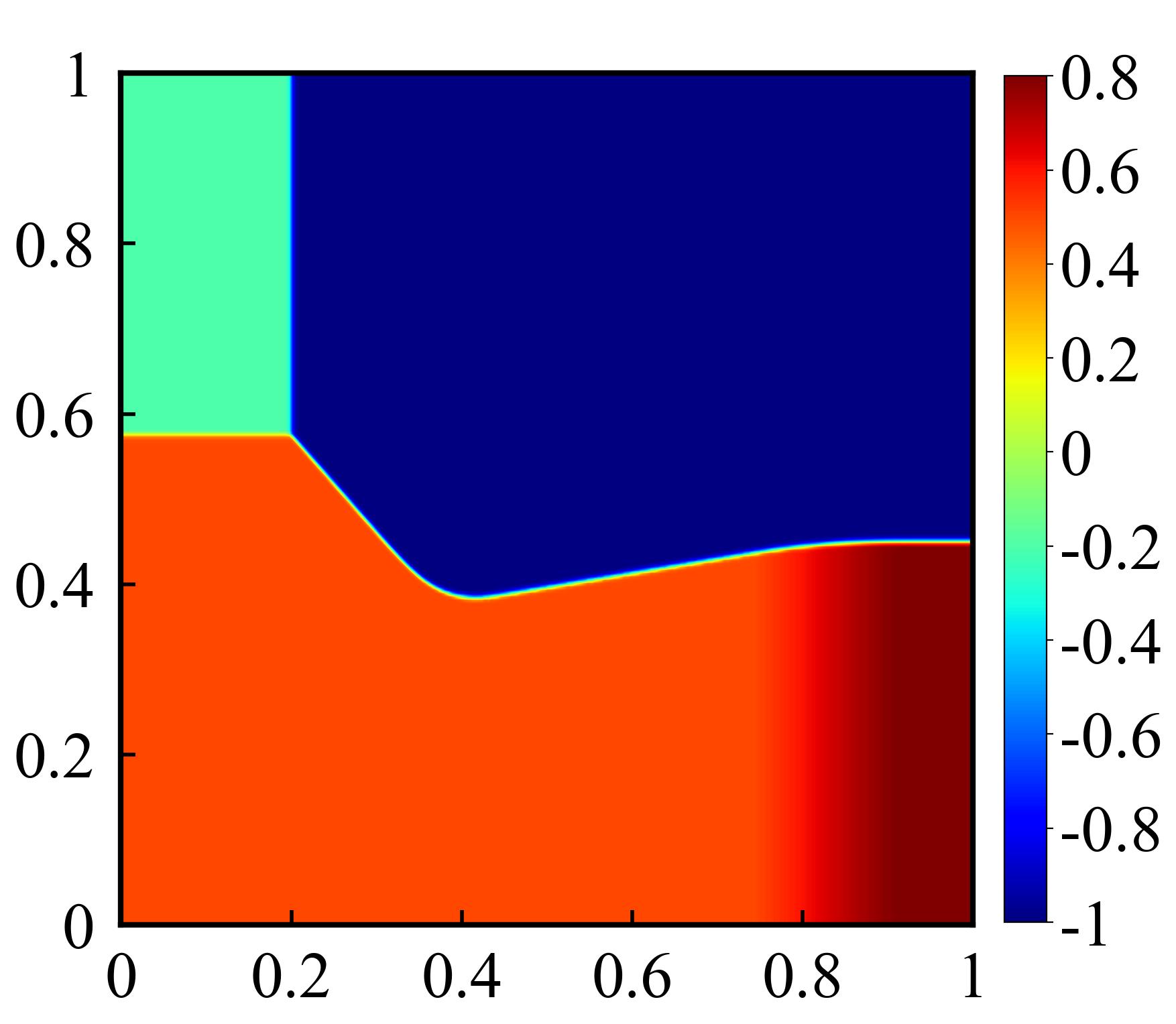}
	\end{subfigure}
	\qquad
	\begin{subfigure}{0.42\textwidth}
		\centering
		\includegraphics[width=\textwidth]{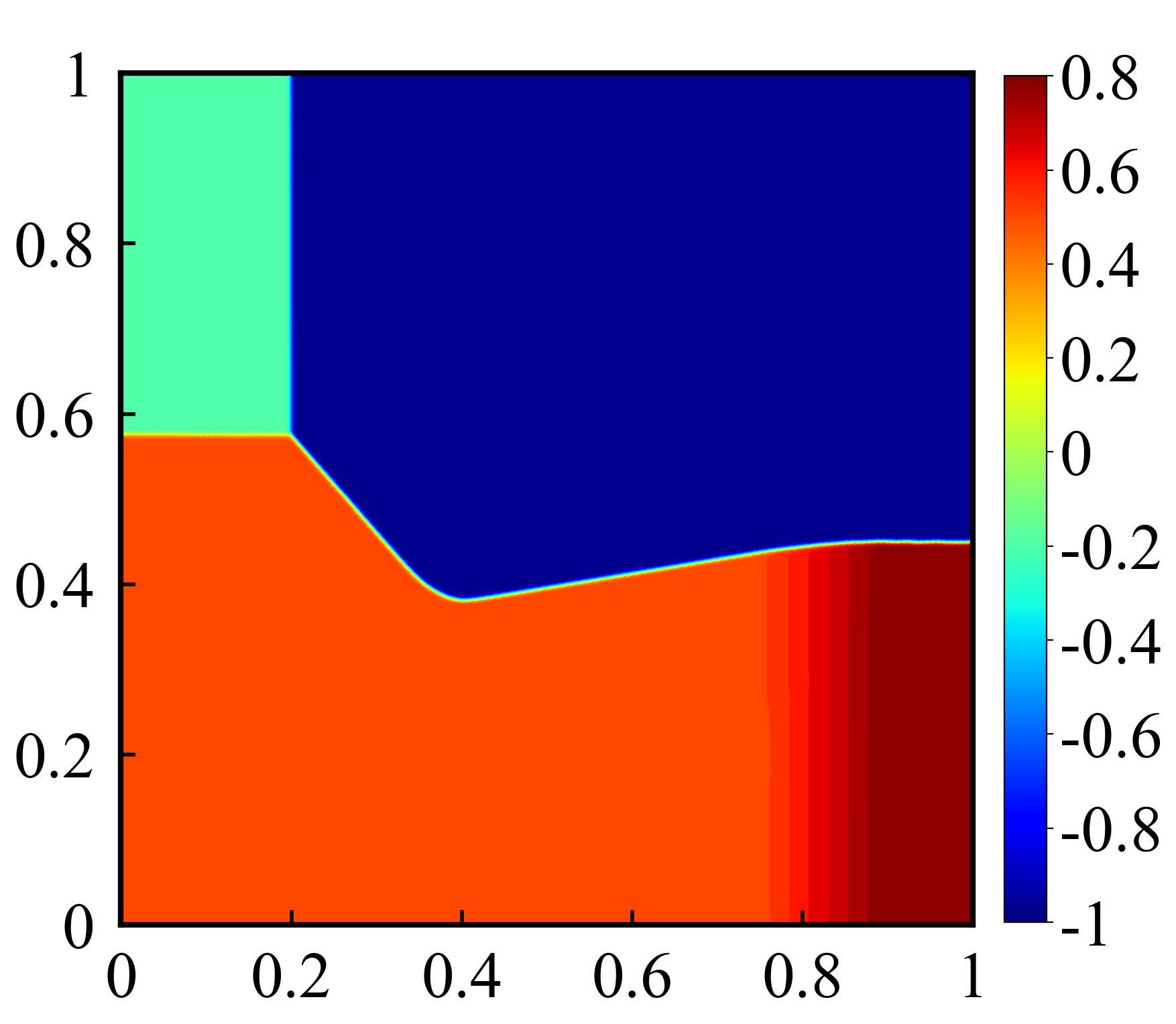}
	\end{subfigure}

	\begin{subfigure}{0.42\textwidth}
		\centering
		\includegraphics[width=\textwidth]{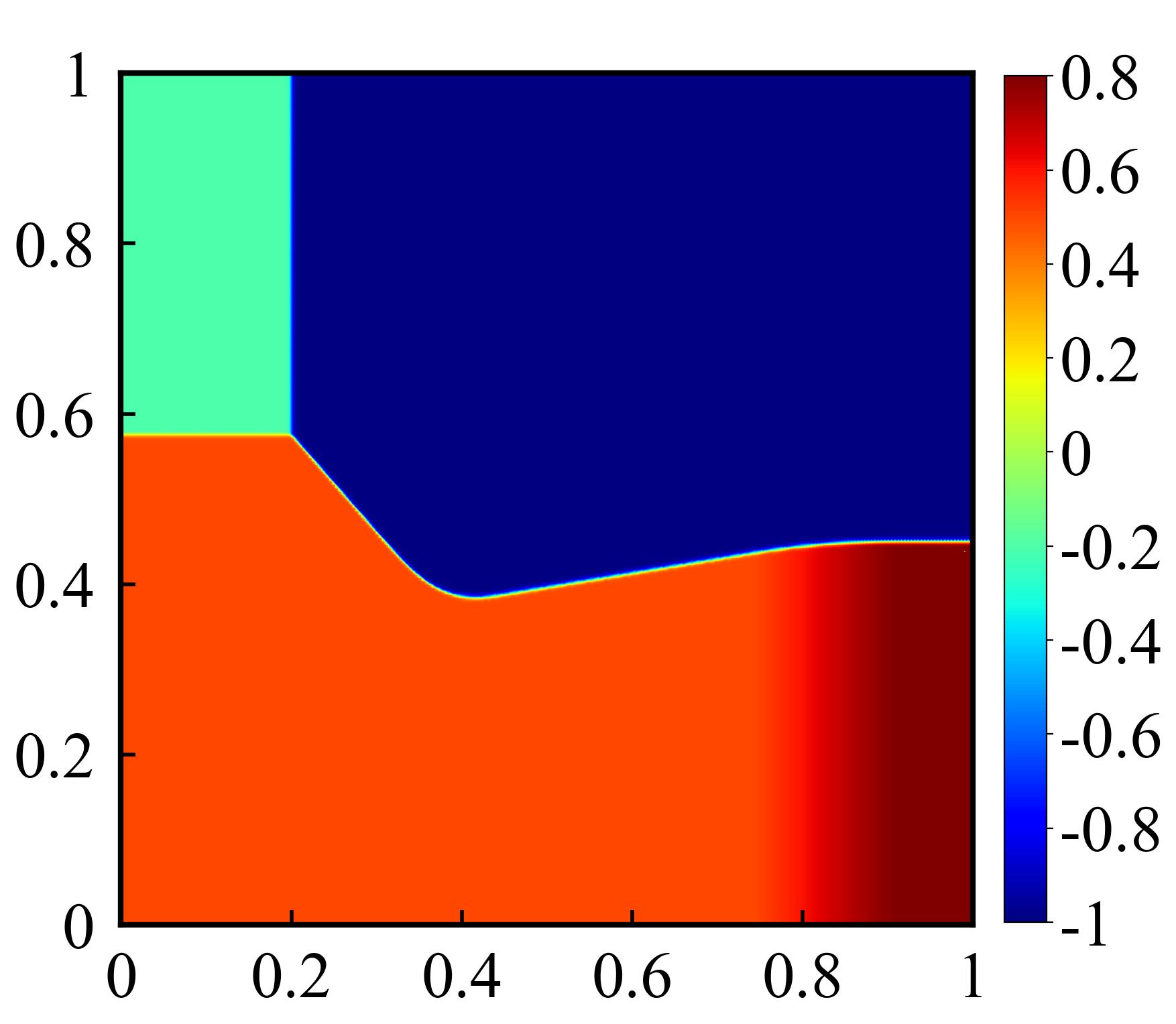}
	\end{subfigure}
	\qquad
	\begin{subfigure}{0.42\textwidth}
		\centering
		\includegraphics[width=\textwidth]{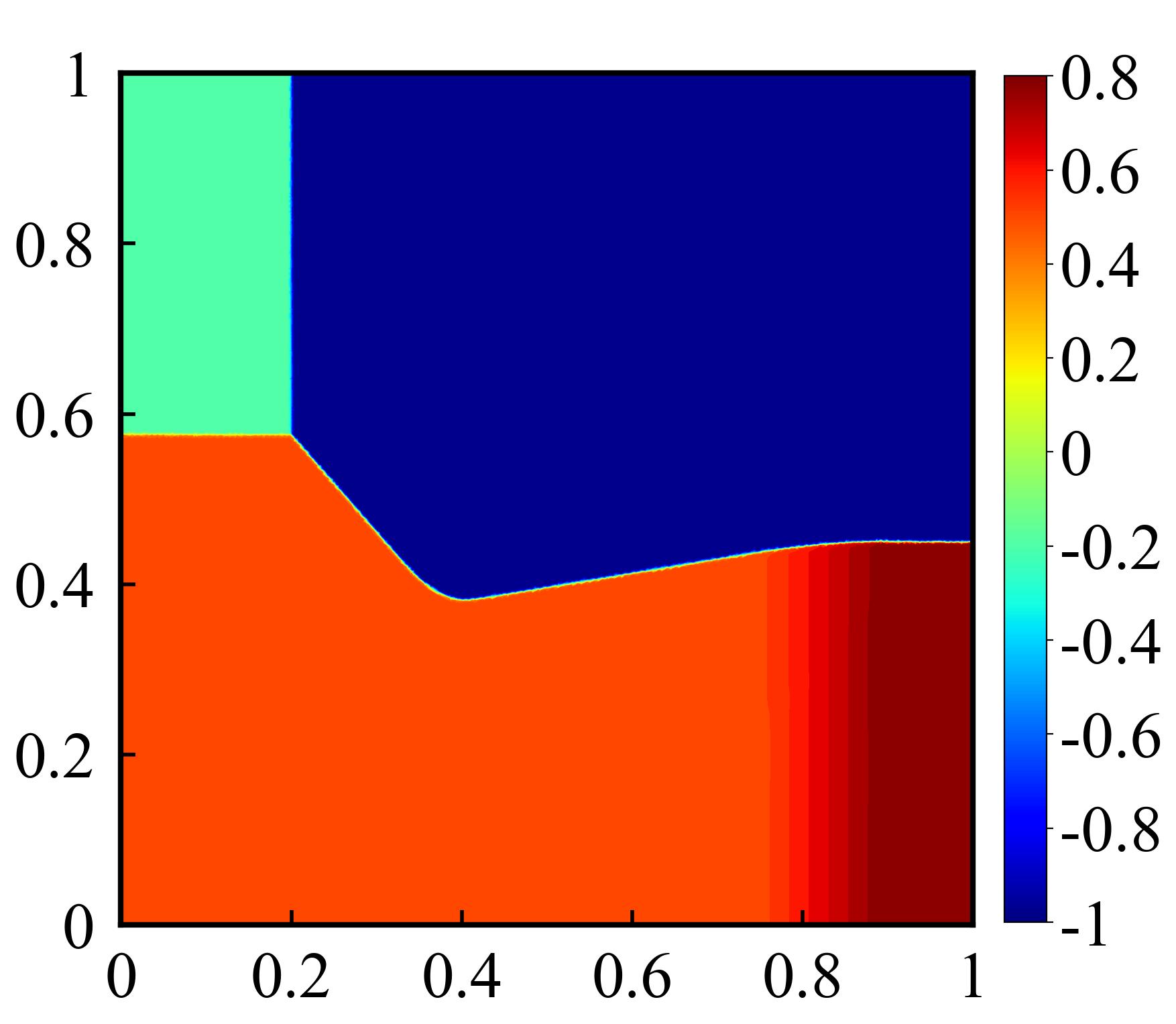}
	\end{subfigure}

	\begin{subfigure}{0.42\textwidth}
		\centering
		\includegraphics[width=\textwidth]{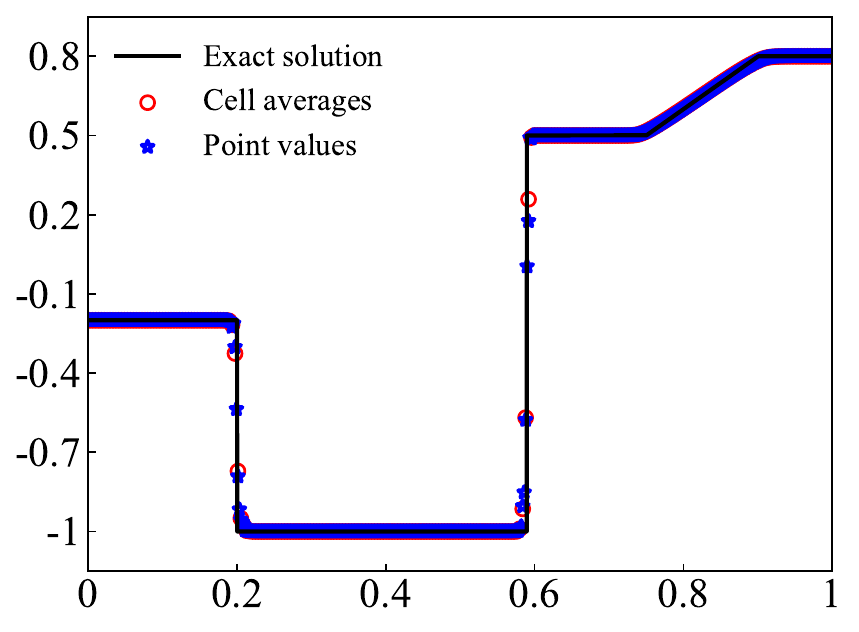}
	\end{subfigure}
	\qquad
	\begin{subfigure}{0.42\textwidth}
		\centering
		\includegraphics[width=\textwidth]{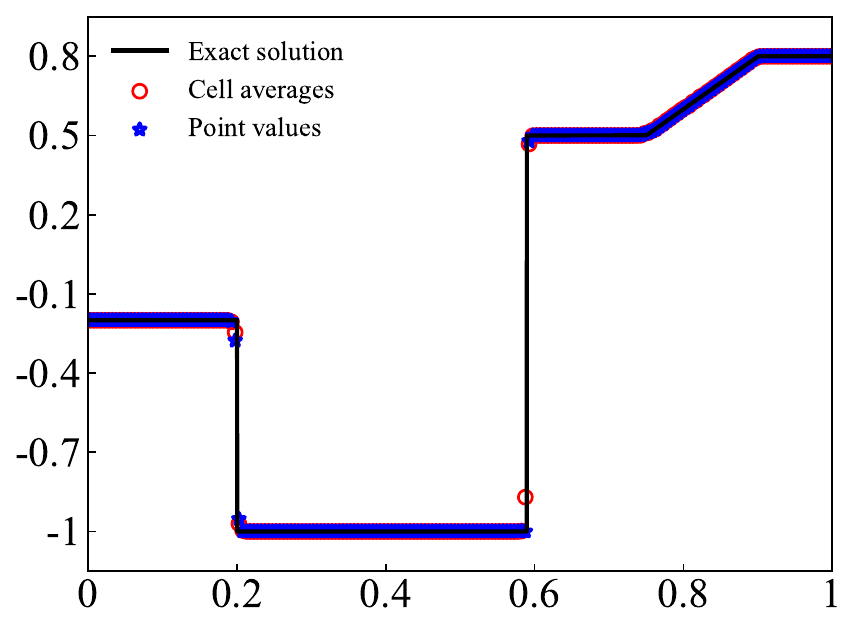}
	\end{subfigure}
	
	\caption{Burgers' oblique Riemann problem:  Cartesian (left) and triangular (right) numerical solutions. Top: cell averages; middle: point values; bottom: diagonal cut along $y = 1 - x$.}
	\label{fig:Ex_BurgRM}
\end{figure}

\begin{table}[!htb] 
	\centering
	\belowrulesep=0pt
	\aboverulesep=0pt
	\caption{Burgers' oblique Riemann problem: cell-average diagnostics.}
	\label{tab:Ex-Burg}
	\setlength{\tabcolsep}{2mm}{
		\begin{tabular}{c|cccc}
			\toprule[1.5pt]
			\multirow{2}{*}{Mesh} & \multicolumn{2}{c}{With limiter $\Pi_h$} & \multicolumn{2}{c}{Without limiter $\Pi_h$}  \\
			\cmidrule(r){2-3} \cmidrule(r){4-5}
			& $\delta_{\max}(\bar u_h)$ & $\delta_{\min}(\bar u_h)$ & $\delta_{\max}(\bar u_h)$ & $\delta_{\min}(\bar u_h)$ \\
			\midrule[1.5pt]
			Cartesian & 1.9915e-11 & 0 & -3.6098e-7 & -5.9203e-7 \\
			Triangular & 0 & 0 & -4.3381e-8 & -2.5536e-7 \\
			\bottomrule[1.5pt]
		\end{tabular}
	}
\end{table}

\subsection{Compressible Euler benchmarks and admissibility diagnostics}\label{subsec:euler}

The two-dimensional compressible Euler benchmarks are performed using the triangular and Cartesian geometric discretizations developed in Section~\ref{sec:geom}. Unless specified otherwise, the adiabatic index is set to $\gamma = 1.4$. 
We refer to \Cref{tab:euler_overview} for the computational configurations and boundary conditions.
Table~\ref{tab:Ex-Euler-Min} reports the final-time density and pressure minima, while \Cref{tab:Ex-Euler-StageMin,tab:Ex-Euler-StageLimiter} list the all-stage minima and the per-stage IDP-limiter activation diagnostics.

\begin{table}[!htb]
\centering
\caption{Euler benchmark configurations.}
\label{tab:euler_overview}
\small
\setlength{\tabcolsep}{3.5pt}
\begin{tabularx}{\textwidth}{>{\raggedright\arraybackslash}p{0.19\textwidth}>{\raggedright\arraybackslash}p{0.16\textwidth}>{\centering\arraybackslash}p{0.14\textwidth}>{\centering\arraybackslash}p{0.11\textwidth}>{\raggedright\arraybackslash}X}
\toprule
Example & Geometry and mesh & Final time & $\gamma$ & Boundary treatment \\
\midrule
Sedov blast & Cartesian, $161\times 161$ & $1$ & $1.4$ & outflow (right and top), reflective (left and bottom) \\
Mach 80 jet & Cartesian, $550\times 132$ on $\Omega^+$ & $0.07$ & $5/3$ & inflow jet on left inlet, reflective symmetry at $y=0$, outflow elsewhere \\
Mach 2000 jet & Cartesian, $475\times 114$ on $\Omega^+$ & $0.001$ & $5/3$ & same as Mach 80 jet \\
Shock diffraction & Triangular, $N=123{,}253$ & $0.9$ & $1.4$ & inflow, outflow, exact-data top boundary, reflective on remaining walls \\
Shock reflection and diffraction & Triangular, $N=303{,}116$ & $0.245$ & $1.4$ & inflow, outflow, exact-data segments, reflective on remaining walls \\
\bottomrule
\end{tabularx}
\end{table}

To rigorously monitor discrete physical admissibility throughout the temporal evolution, we evaluate the all-stage minima
\[
\begin{aligned}
\rho_{\min,\mathrm{avg}}^{\mathrm{stg}}&:=\min_{n,s,K}\rho(\bar u_K^{(n,s)}),
&\qquad
 p_{\min,\mathrm{avg}}^{\mathrm{stg}}&:=\min_{n,s,K}p(\bar u_K^{(n,s)}),\\
\rho_{\min,\mathrm{pt}}^{\mathrm{stg}}&:=\min_{n,s,\sigma}\rho(u_\sigma^{(n,s)}),
&\qquad
 p_{\min,\mathrm{pt}}^{\mathrm{stg}}&:=\min_{n,s,\sigma}p(u_\sigma^{(n,s)}),
\end{aligned}
\]
where $s$ ranges over the intermediate SSP stages. Furthermore, the per-stage IDP-limiter activity is quantified by
\[
\begin{aligned}
\eta_{\max}^{\mathrm{act}}&:=\max_{n,s}\frac{\#\{K: \theta_K^{(n,s)}<1\}}{\#\mathcal T_h}, \qquad
\theta_{\min}^{\mathrm{stg}}&:=\min_{n,s,K}\theta_K^{(n,s)},
\end{aligned}
\]
where $\theta_K^{(n,s)}$ denotes the scaling parameter of the IDP limiter $\Pi_K$ after any optional OE preprocessing has been executed.

\begin{table}[!htb] 
	\centering
	\belowrulesep=0pt
	\aboverulesep=0pt
\caption{Euler benchmarks: final-time minima.}
	\label{tab:Ex-Euler-Min}
	\setlength{\tabcolsep}{1mm}{
		\begin{tabular}{c|cccc}
			\toprule[1.5pt]
			\multirow{2}{*}{ Examples } &
			\multicolumn{2}{c}{cell averages} &
			\multicolumn{2}{c}{point values} \\
			\cmidrule(r){2-3} \cmidrule(r){4-5}
			& $\rho_{\rm min}$ & $p_{\rm min}$ & $\rho_{\rm min}$ & $p_{\rm min}$ \\

			\midrule[1.5pt]

			{Sedov blast} & 1.7919e-3 & 4e-13 &1.7914e-3 & 3.6883e-13 \\
			{Mach 80 jet}  & 6.1694e-2 &4.0618e-1 & 6.0763e-2& 6.1912e-2 \\
			{Mach 2000 jet}  & 6.7829e-2 & 3.4775e-1 & 6.6421e-2 & 1.6657e-6 \\
			{Shock diffraction}  & 1.3326e-1 & 9.4016e-1 & 1.2638e-1 & 8.4504e-1 \\
			{Shock reflection and diffraction}  & 6.9e-2 & 5.4808e-1 & 6.8487e-2 & 5.4220e-1 \\

			\bottomrule[1.5pt]
		\end{tabular}
	}
\end{table}

\begin{table}[!htb] 
	\centering
	\belowrulesep=0pt
	\aboverulesep=0pt
\caption{Euler benchmarks: all-stage minima.}
	\label{tab:Ex-Euler-StageMin}
	\setlength{\tabcolsep}{1mm}{
		\begin{tabular}{c|cccc}
			\toprule[1.5pt]
			\multirow{2}{*}{Examples} & \multicolumn{2}{c}{cell averages} & \multicolumn{2}{c}{point values} \\
			\cmidrule(r){2-3} \cmidrule(r){4-5}
			& $\rho_{\min,\mathrm{avg}}^{\mathrm{stg}}$ & $p_{\min,\mathrm{avg}}^{\mathrm{stg}}$ & $\rho_{\min,\mathrm{pt}}^{\mathrm{stg}}$ & $p_{\min,\mathrm{pt}}^{\mathrm{stg}}$ \\
			\midrule[1.5pt]
			Sedov blast & 1.7919e-3 & 4e-13 & 1.7914e-3 & 3.6719e-13 \\
			Mach 80 jet & 6.1693e-2 & 4.0614e-1 & 6.0763e-2 & 2.1051e-2 \\
			Mach 2000 jet & 6.7829e-2 & 3.4772e-1 & 6.6421e-2 & 1.6031e-6 \\
			Shock diffraction &1.3326e-1 & 9.4016e-1 & 1.2638e-1 & 7.6237e-1 \\
			Shock reflection and diffraction & 6.9e-2 & 5.4807e-1 & 6.8485e-2 & 5.4219e-1 \\
			\bottomrule[1.5pt]
		\end{tabular}
	}
\end{table}

\begin{table}[!htb]
	\centering
	\belowrulesep=0pt
	\aboverulesep=0pt
\caption{Euler benchmarks: per-stage limiter diagnostics.}
	\label{tab:Ex-Euler-StageLimiter}
	\setlength{\tabcolsep}{1.2mm}{
		\begin{tabular}{c|cc}
			\toprule[1.5pt]
			Examples & $\eta_{\max}^{\mathrm{act}}$ & $\theta_{\min}^{\mathrm{stg}}$ \\
			\midrule[1.5pt]
			Sedov blast & 1.9052e-4 & 4.6336e-16 \\
			Mach 80 jet & 2.0279e-4 & 7.1886e-3 \\
			Mach 2000 jet & 4.5182e-4 & 1.0038e-5 \\
			Shock diffraction & 6.1662e-4 & 2.6739e-2 \\
			Shock reflection and diffraction & 3.8599e-4 & 2.4933e-2 \\
			\bottomrule[1.5pt]
		\end{tabular}
	}
\end{table}

\subsubsection{Sedov blast}\label{Ex:Sedov}
We first test the Sedov blast benchmark following the setup in \cite{zhang2010positivity,ZXSPP2012}. The computational domain is $[0,1.1]^2$, discretized using a uniform Cartesian mesh of $161\times161$ cells. The initial state is
\begin{equation*}\label{Ex:SedovIni}
	\rho=1, ~~ v_1=v_2=0, ~~ 
	E=\begin{cases}
		\frac{0.244816}{|I_{0,0}|}, & \text{if} ~ (x,y)\in I_{0,0} ~\text{(the bottom-left cell)}, \\
		10^{-12}, & \text{otherwise}.
	\end{cases}
\end{equation*}
where $I_{0,0}$ denotes the bottom-left computational cell, and the analytical reference solution is provided in \cite{Sedov2018}. Outflow boundary conditions are imposed on the right and top boundaries, while reflective wall conditions are prescribed along the bottom and left boundaries.
\Cref{fig:Sedov} presents the density contours and the 1D profile along the diagonal $y=x$ at the final time $t=1$, while \Cref{tab:Ex-Euler-Min,tab:Ex-Euler-StageMin,tab:Ex-Euler-StageLimiter} summarize the final-time minima, the all-stage minima, and the associated IDP-limiter diagnostics. 
The $10^{-13}$-level minimum pressure values reflect the floating-point safeguard described in Appendix~\ref{sec:IDP_imple}.

\begin{figure}[!htb]
	\centering
	\begin{subfigure}{0.32\textwidth}
		\centering
		\includegraphics[width=\textwidth]{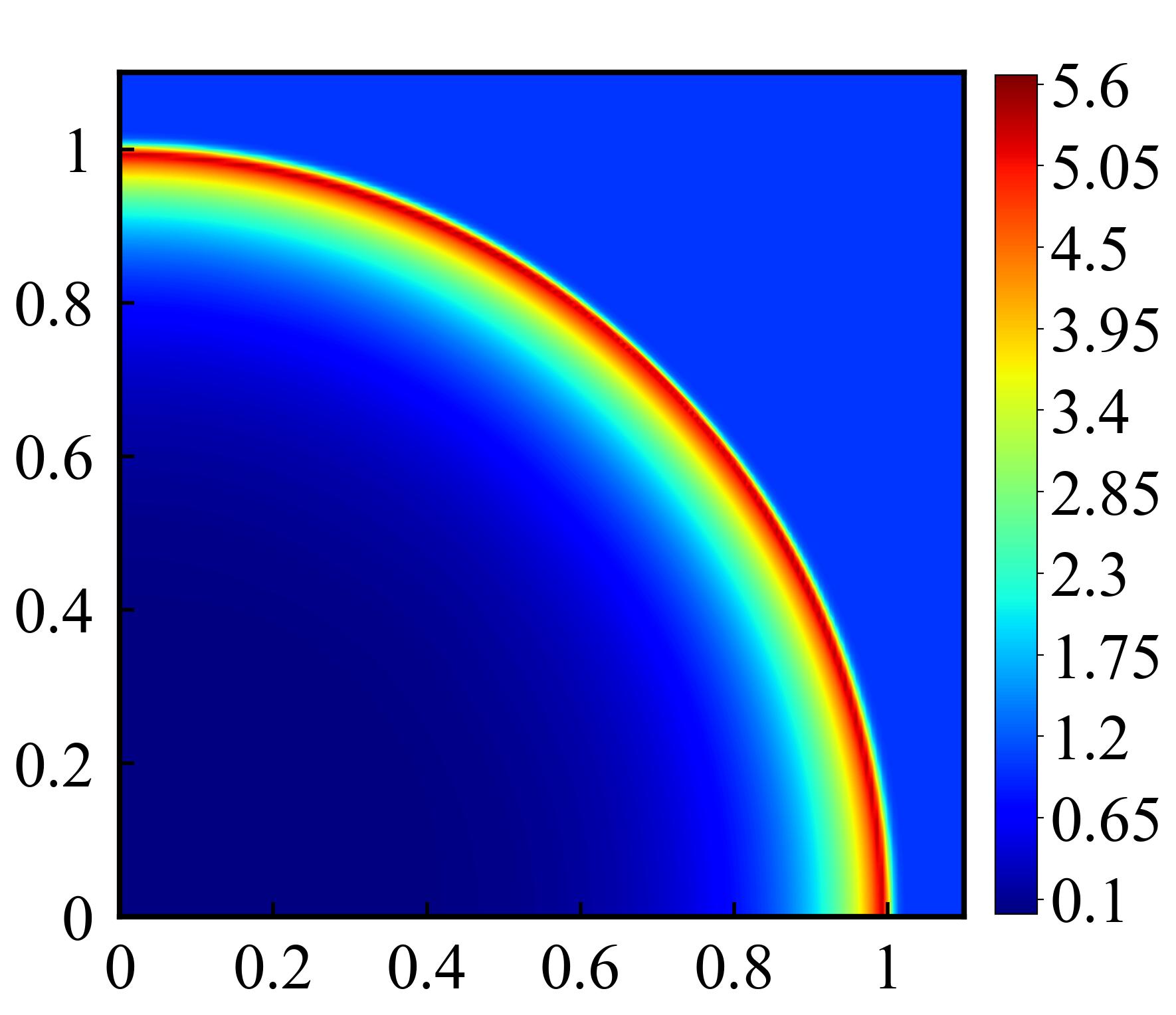}
		\caption{Cell averages.}
		\label{fig:Sedov1}
	\end{subfigure}
	\hfill
	\begin{subfigure}{0.32\textwidth}
		\centering
		\includegraphics[width=\textwidth]{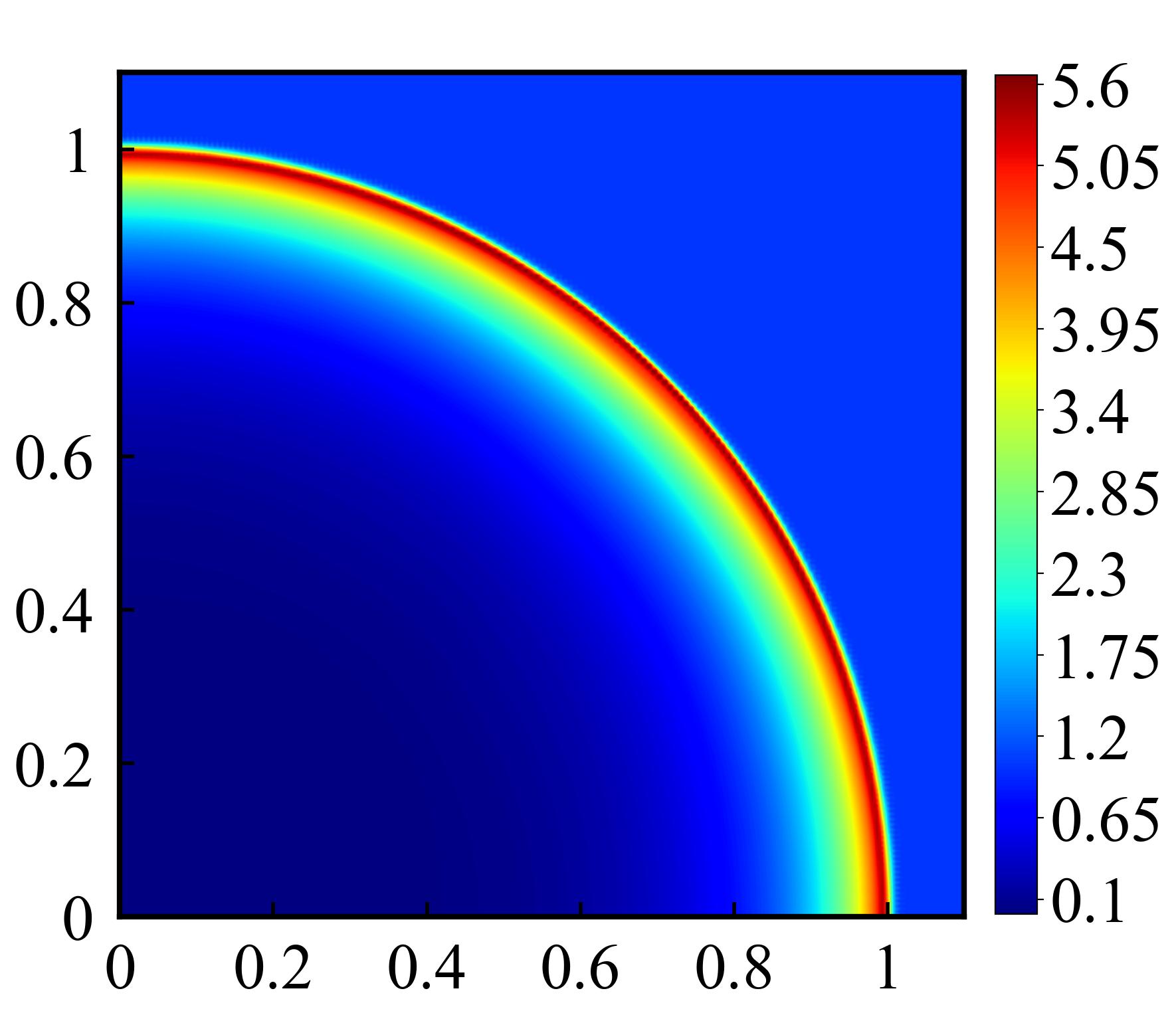}
		\caption{Point values.}
		\label{fig:Sedov2}
	\end{subfigure}
	\hfill
	\begin{subfigure}{0.32\textwidth}
		\includegraphics[width=\textwidth]{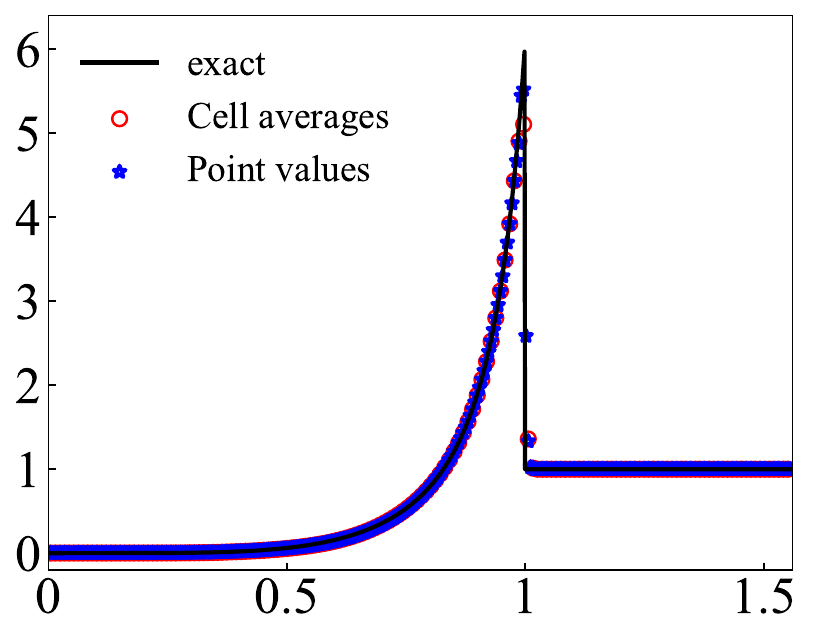}
		\caption{Cut along $y=x$.}
		\label{fig:Sedov3}
	\end{subfigure}
	
	\caption{Sedov blast: density contours and 1D diagonal cut along $y=x$.}
	\label{fig:Sedov}
\end{figure}

\subsubsection{High-Mach jet problems} \label{Ex:Jet}
We next investigate the Mach 80 and Mach 2000 astrophysical jet benchmarks from \cite{CDWOCAD2023,peng2025oedg}. For both configurations, the adiabatic index is $\gamma=5/3$, and the ambient state in $\Omega$ is $(\rho,v_1,v_2,p)=(0.5,0,0,0.4127)$. A high-speed jet is injected through the left boundary segment $y\in[-0.05,0.05]$ along the positive $x$-direction.

\emph{Mach 80 jet.}
We first consider the physical domain $\Omega=[0,2]\times[-0.48,0.48]$. Exploiting symmetry, the simulation is carried out on the upper half-domain $\Omega^+=[0,2]\times[0,0.48]$, discretized using a $550\times132$ uniform Cartesian mesh. On the inlet segment $\Gamma_{\mathrm{in}}=\{(x,y):x=0,\ 0\le y\le0.05\}$, the inflow state is $(\rho,v_1,v_2,p)=(5,30,0,0.4127)$. A reflective symmetry condition is prescribed along $y=0$, while outflow conditions are applied on the remaining boundaries. The simulation is advanced to the final time $t=0.07$.

\emph{Mach 2000 jet.}
We next consider the domain $\Omega=[0,1]\times[-0.24,0.24]$. The simulation is similarly performed on the half-domain $\Omega^+=[0,1]\times[0,0.24]$, discretized using a $475\times114$ uniform Cartesian mesh. The inflow state on $\Gamma_{\mathrm{in}}$ is $(\rho,v_1,v_2,p)=(5,800,0,0.4127)$. The boundary conditions are identical to those in the Mach 80 benchmark, and the simulation is run until $t=0.001$.

\Cref{fig:Ex-Jet1,fig:Ex-Jet2} display the logarithmic density and pressure distributions for the Mach 80 and Mach 2000 jets, while \Cref{tab:Ex-Euler-Min,tab:Ex-Euler-StageMin,tab:Ex-Euler-StageLimiter} report the corresponding discrete admissibility diagnostics.

\begin{figure}[!htb]
	\centering	
	\begin{subfigure}{0.48\textwidth}
		\includegraphics[width=\textwidth]{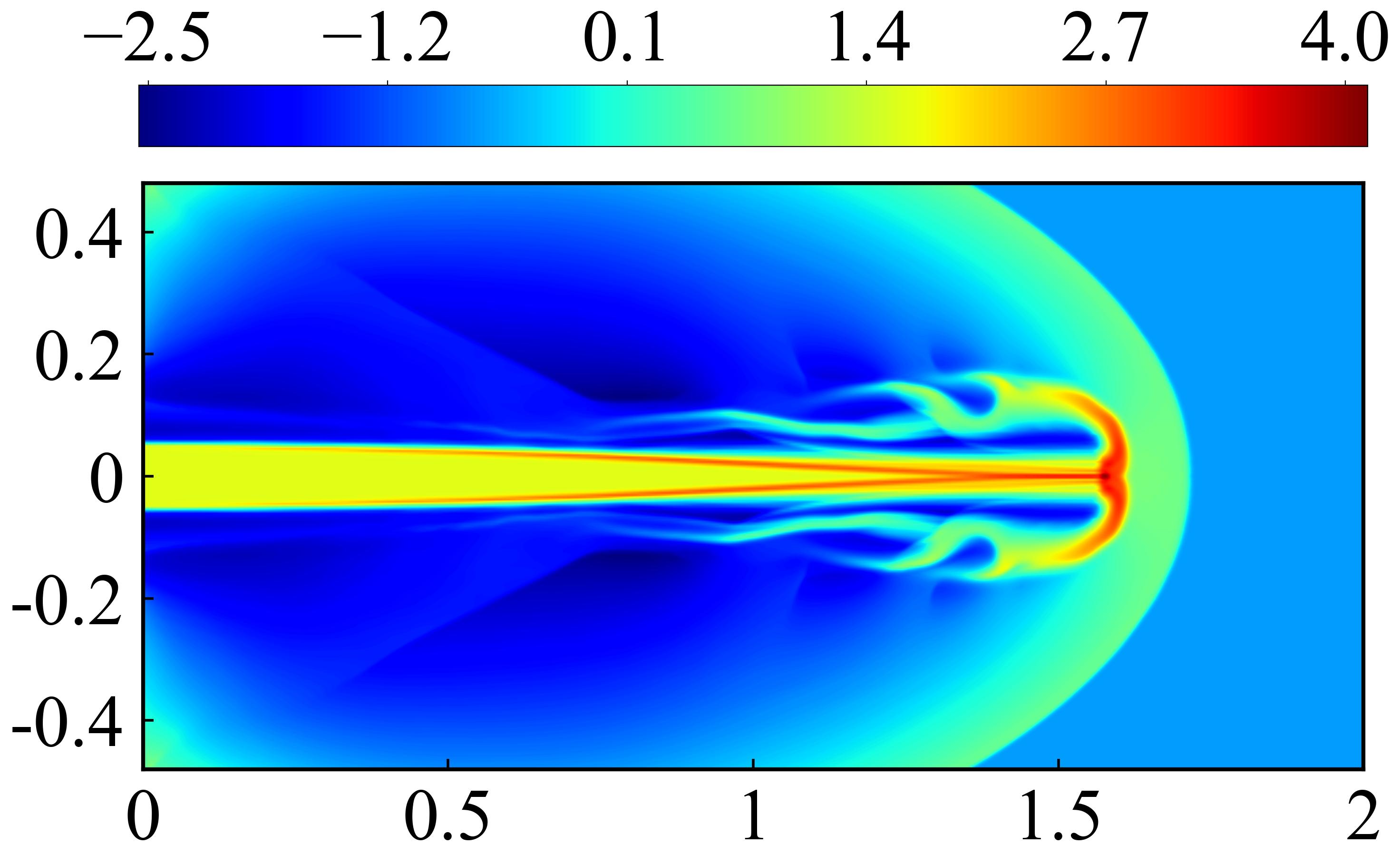}
	\end{subfigure}
	\quad
	\begin{subfigure}{0.48\textwidth}
		\includegraphics[width=\textwidth]{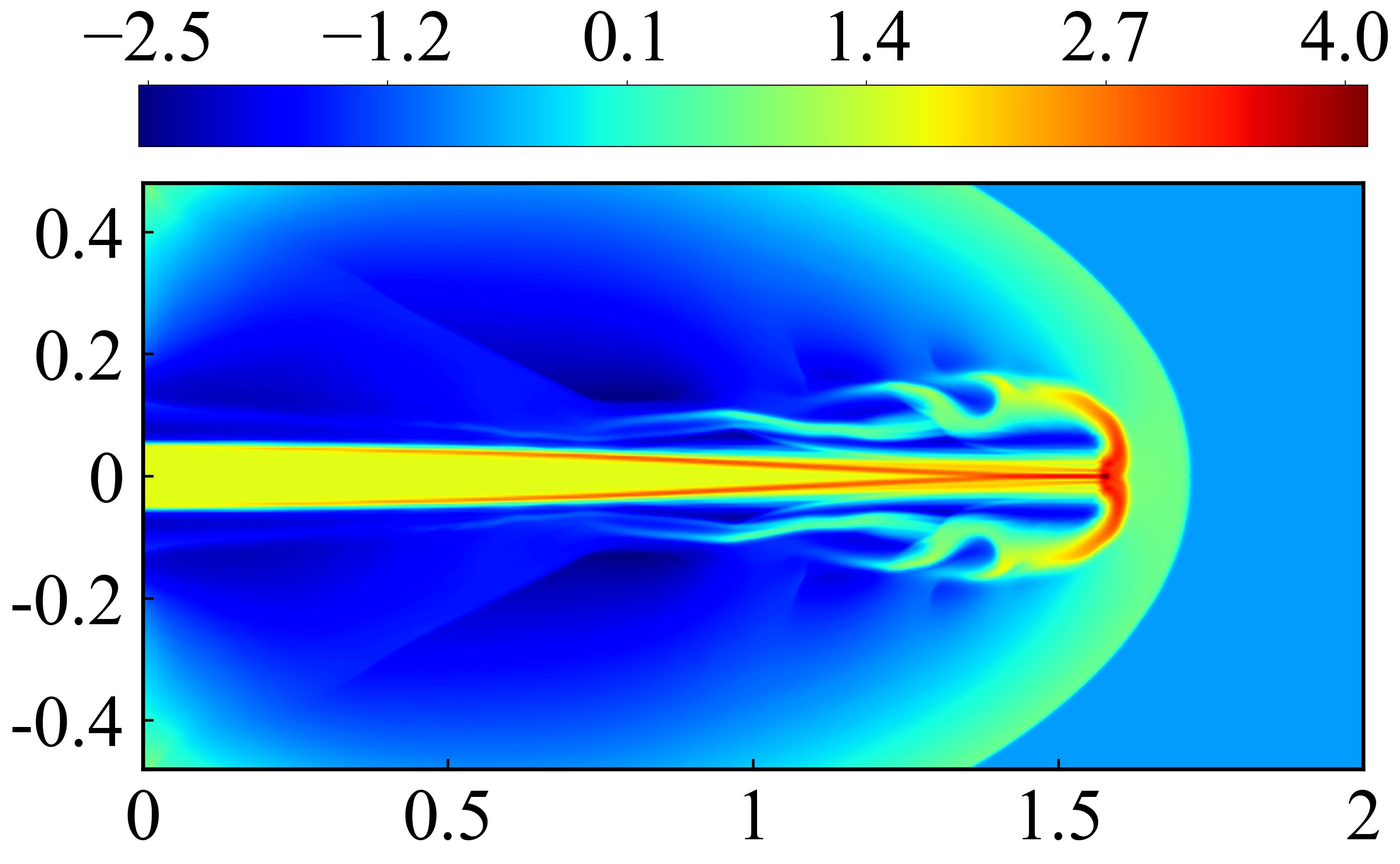}
	\end{subfigure}
	
	\begin{subfigure}{0.48\textwidth}
		\includegraphics[width=\textwidth]{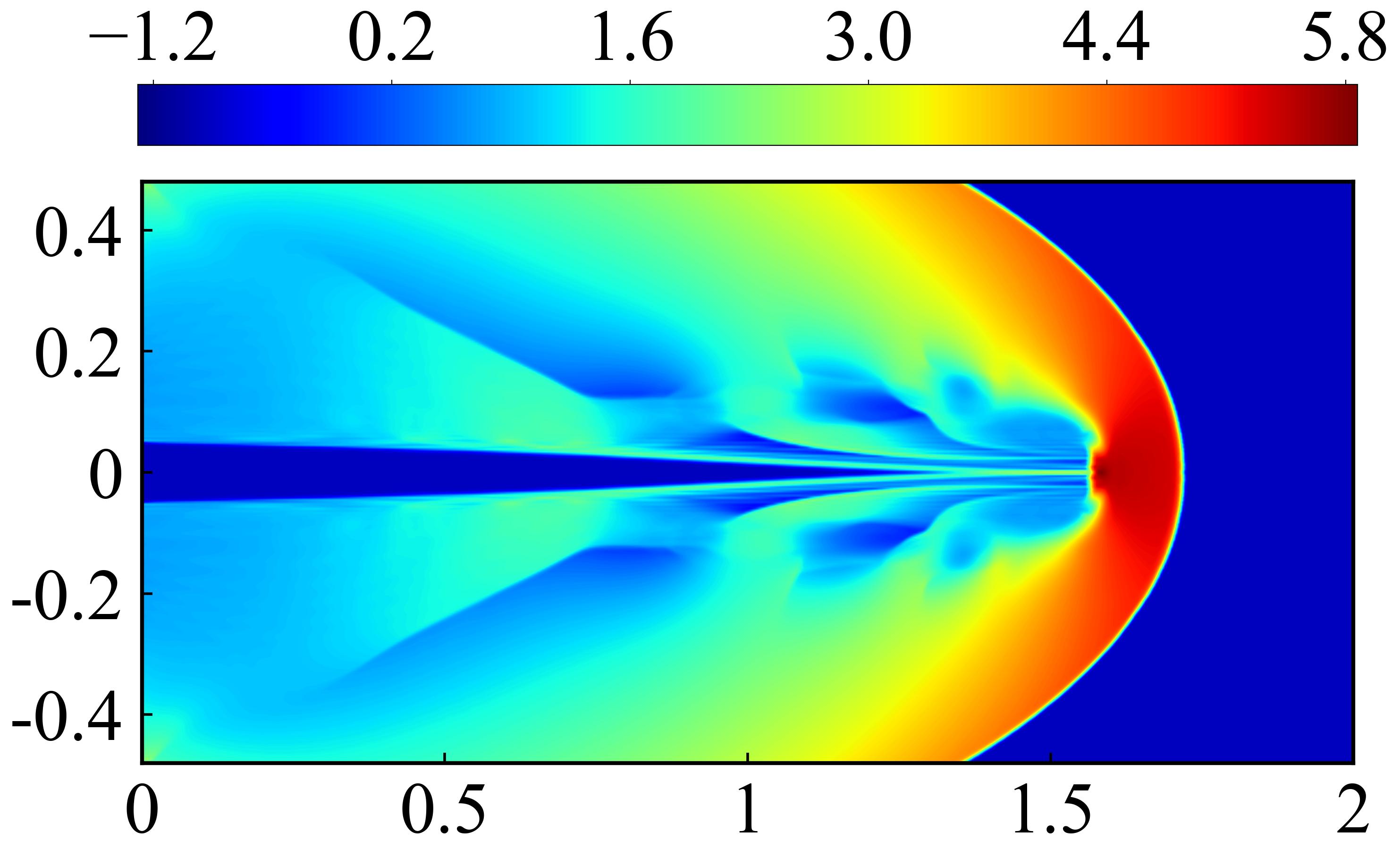}
	\end{subfigure}
	\quad
	\begin{subfigure}{0.48\textwidth}
		\includegraphics[width=\textwidth]{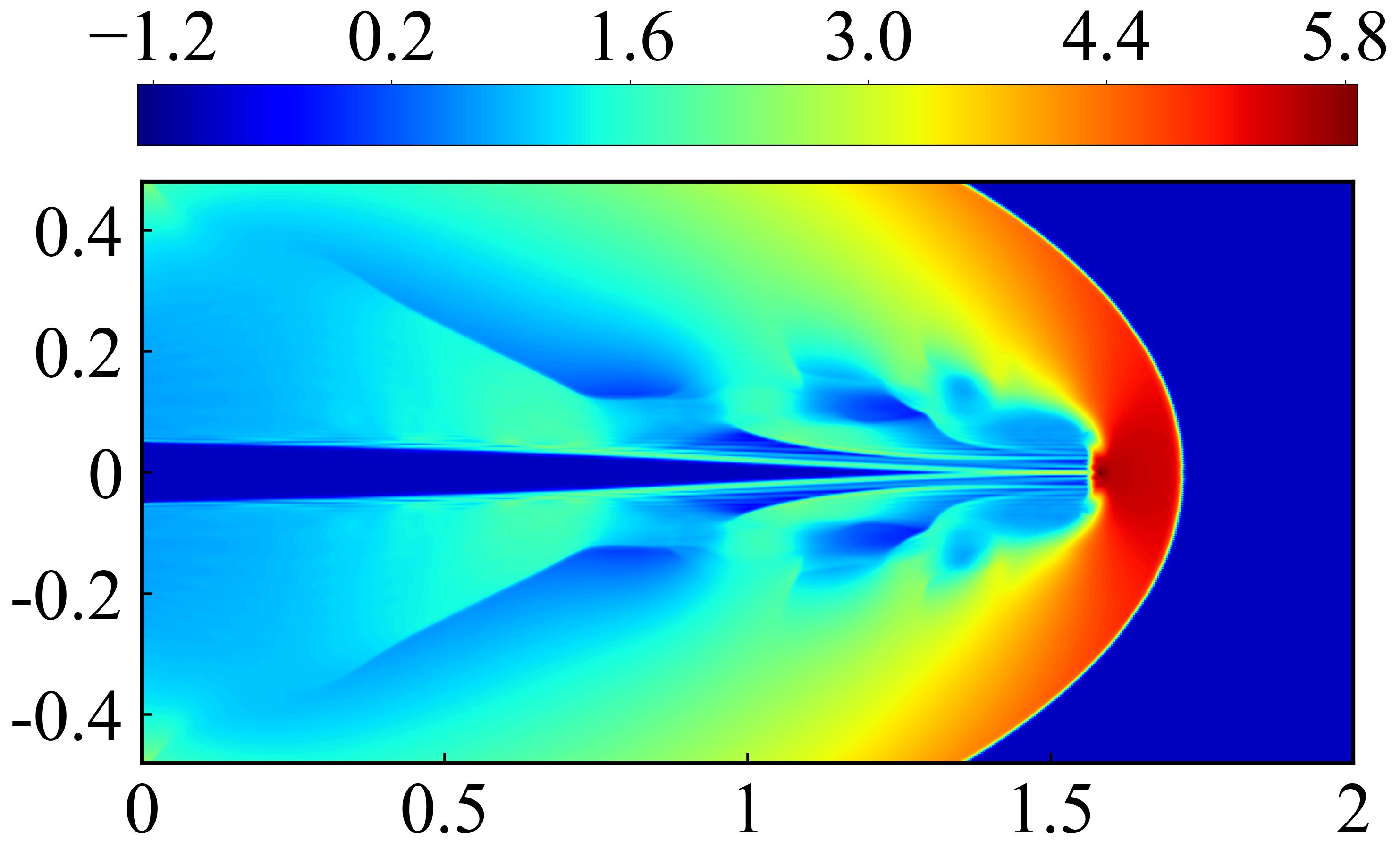}
	\end{subfigure}
	
	\caption{Mach 80 jet: logarithmic density (top) and logarithmic pressure (bottom). Left: cell averages; right: point values.}
	\label{fig:Ex-Jet1}
\end{figure} 

\begin{figure}[!htb]
	\centering	
	\begin{subfigure}{0.48\textwidth}
		\includegraphics[width=\textwidth]{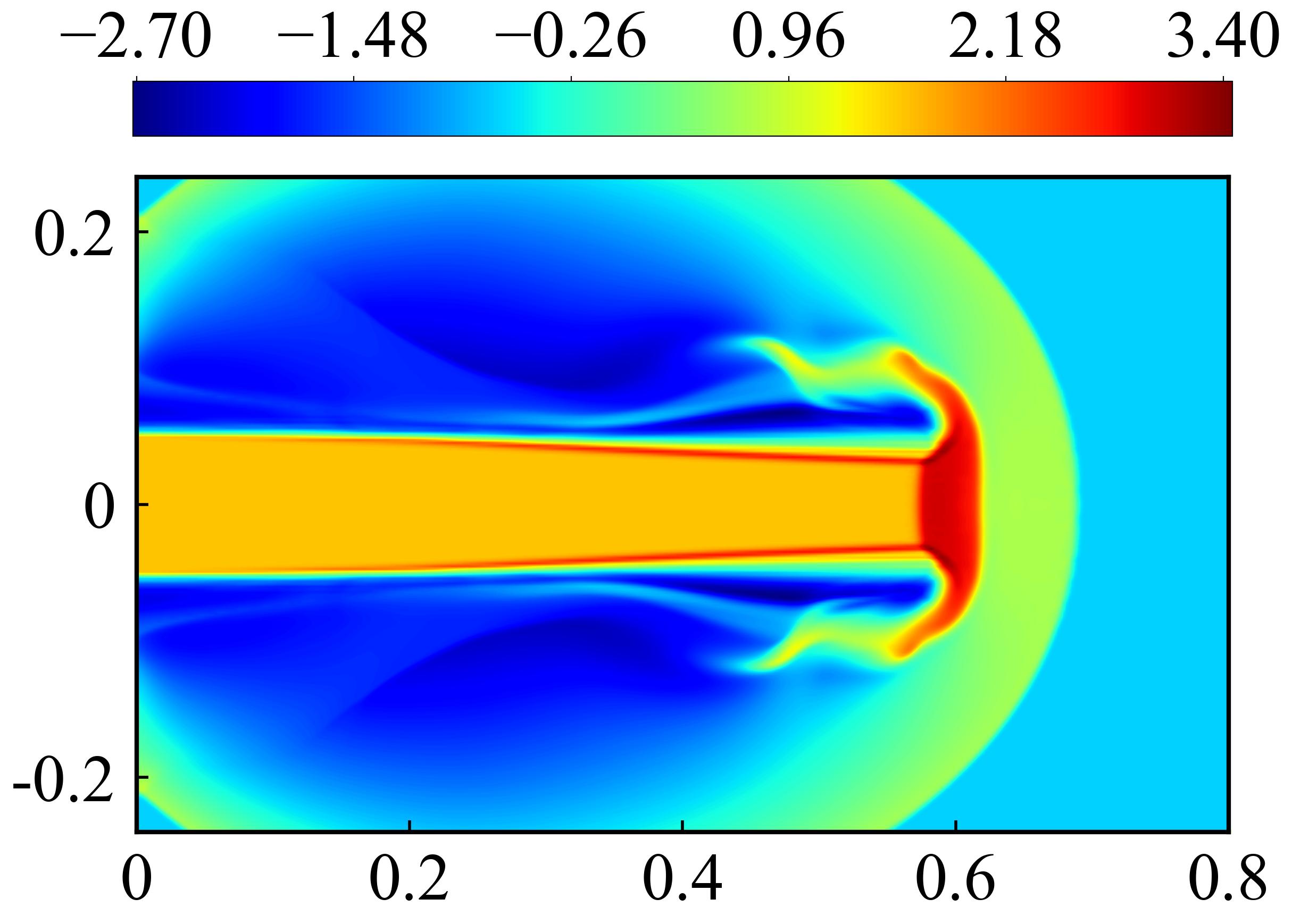}
	\end{subfigure}
	\quad
	\begin{subfigure}{0.48\textwidth}
		\includegraphics[width=\textwidth]{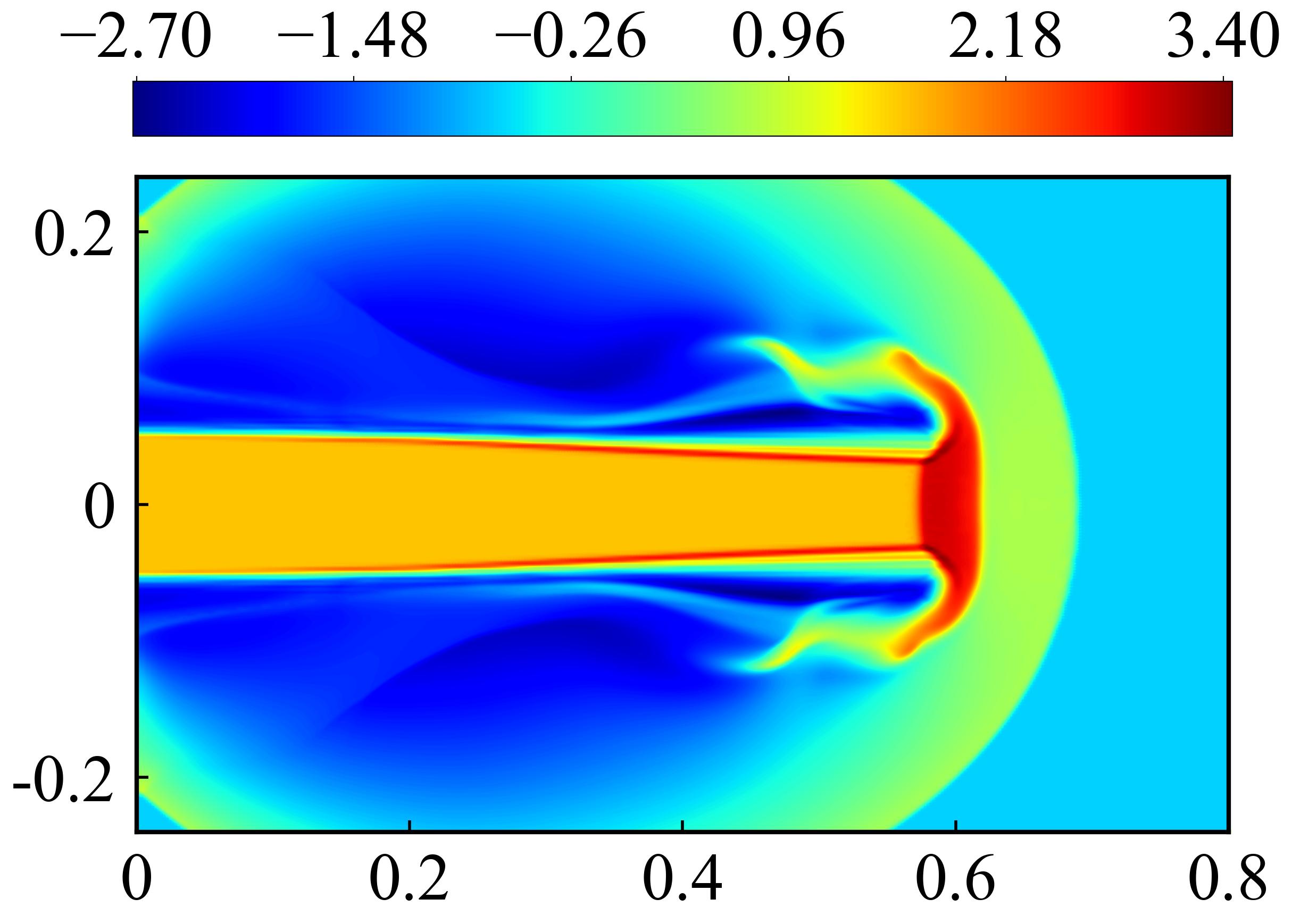}
	\end{subfigure}
	
	\begin{subfigure}{0.48\textwidth}
		\includegraphics[width=\textwidth]{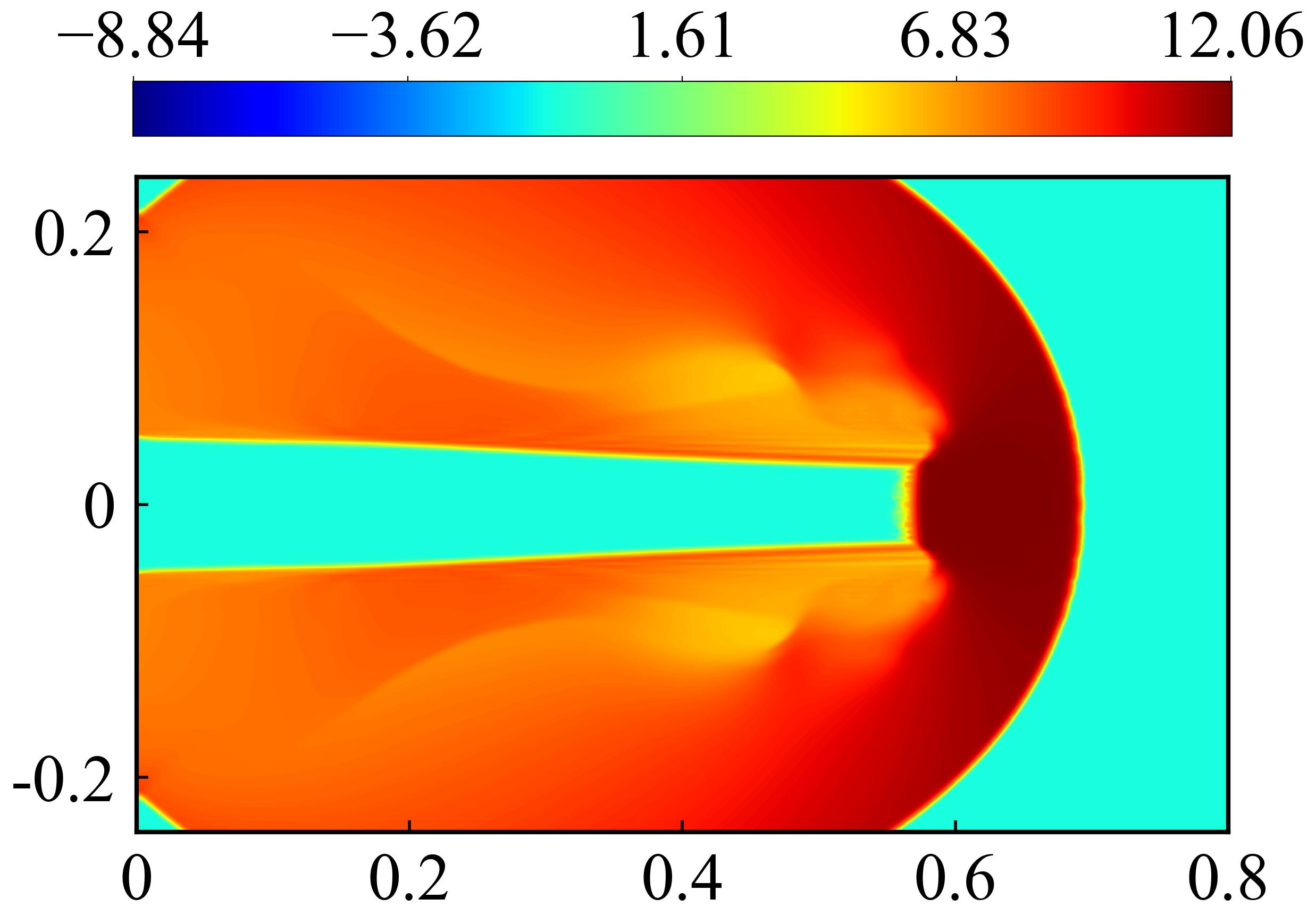}
	\end{subfigure}
	\quad
	\begin{subfigure}{0.48\textwidth}
		\includegraphics[width=\textwidth]{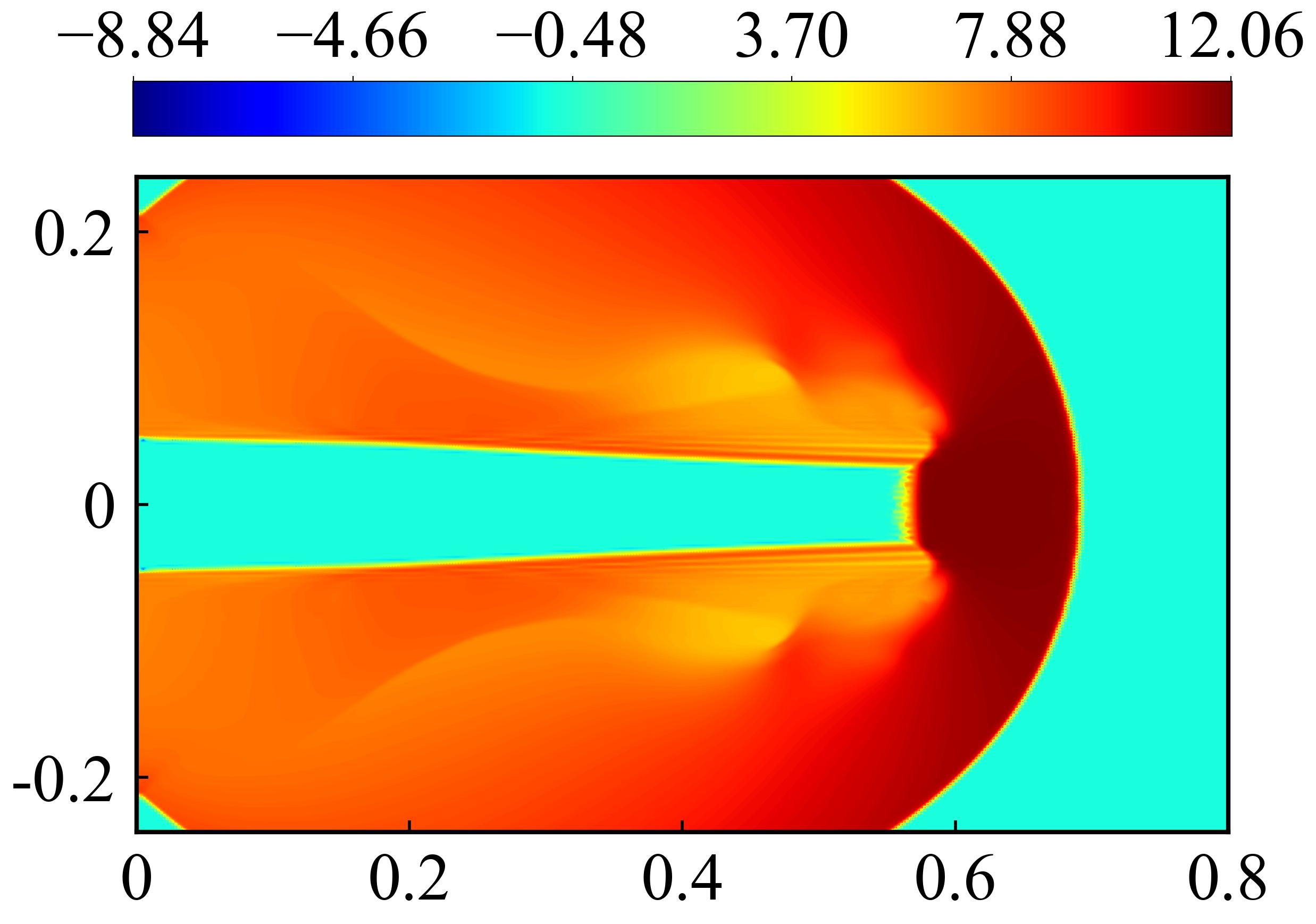}
	\end{subfigure}
	
	\caption{Mach 2000 jet: logarithmic density (top) and logarithmic pressure (bottom). Left: cell averages; right: point values.}
	\label{fig:Ex-Jet2}
\end{figure}

\subsubsection{Shock diffraction at a convex corner (irregular domain)} \label{Ex:SDW}
We consider the Mach 10 shock-diffraction benchmark on an irregular polygonal domain with vertices $(0,0)$, $(13,0)$, $(13,11)$, $(0,11)$, $(0,6)$, and $(2\sqrt{3},6)$. The domain is discretized using an unstructured triangular mesh consisting of $N=123{,}253$ cells and characteristic mesh size $h_{\mathcal T}=0.05$. A coarsened sample mesh ($h=0.5$) is illustrated in \Cref{fig:Ex_SDWMesh}. Initially, a planar right-moving shock located at $x=3.4$ with $y\in[6,11]$ separates the unperturbed states
\begin{equation*}
	(\rho, v_1, v_2, p) = 
	\begin{cases}
		(8, 8.25, 0, 116.5), & x\leq 3.4, 6 \leq y \leq 11, \\
		(1.4, 0, 0, 1), & \text{otherwise.}
	\end{cases}
\end{equation*}
The left boundary segment $\{x=0,\ 6\le y\le 11\}$ is prescribed with inflow data; the right boundary $\{x=13,\ 0\le y\le 11\}$ and the bottom boundary $\{0\le x\le 13,\ y=0\}$ are equipped with outflow conditions; the top boundary $\{0\le x\le 13,\ y=11\}$ imposes the exact-data boundary treatment; and the remaining solid walls enforce the reflective slip-wall condition. 
\Cref{fig:Ex_SDW1,fig:Ex_SDW2} display the density contours at the final time $t=0.9$, while \Cref{tab:Ex-Euler-Min,tab:Ex-Euler-StageMin,tab:Ex-Euler-StageLimiter} report the final-time minima, the all-stage minima, and the associated IDP-limiter activation diagnostics.

\begin{figure}[!htb]
	\centering	
    \begin{subfigure}{0.32\textwidth}
		\includegraphics[width=\textwidth]{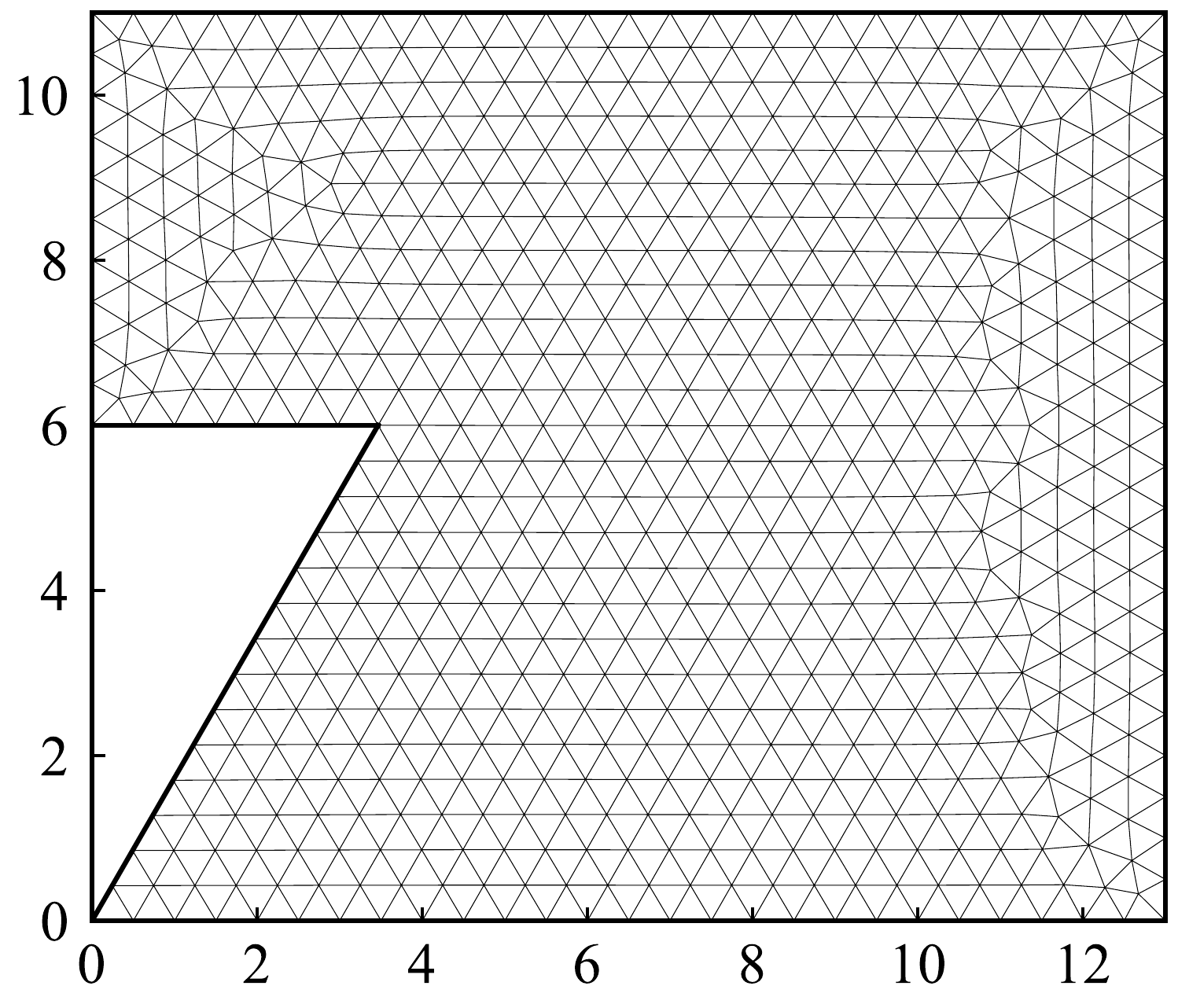}
		\caption{Coarsened sample mesh.}
		\label{fig:Ex_SDWMesh}
	\end{subfigure}
	\hfill
	\begin{subfigure}{0.32\textwidth}
		\includegraphics[width=\textwidth]{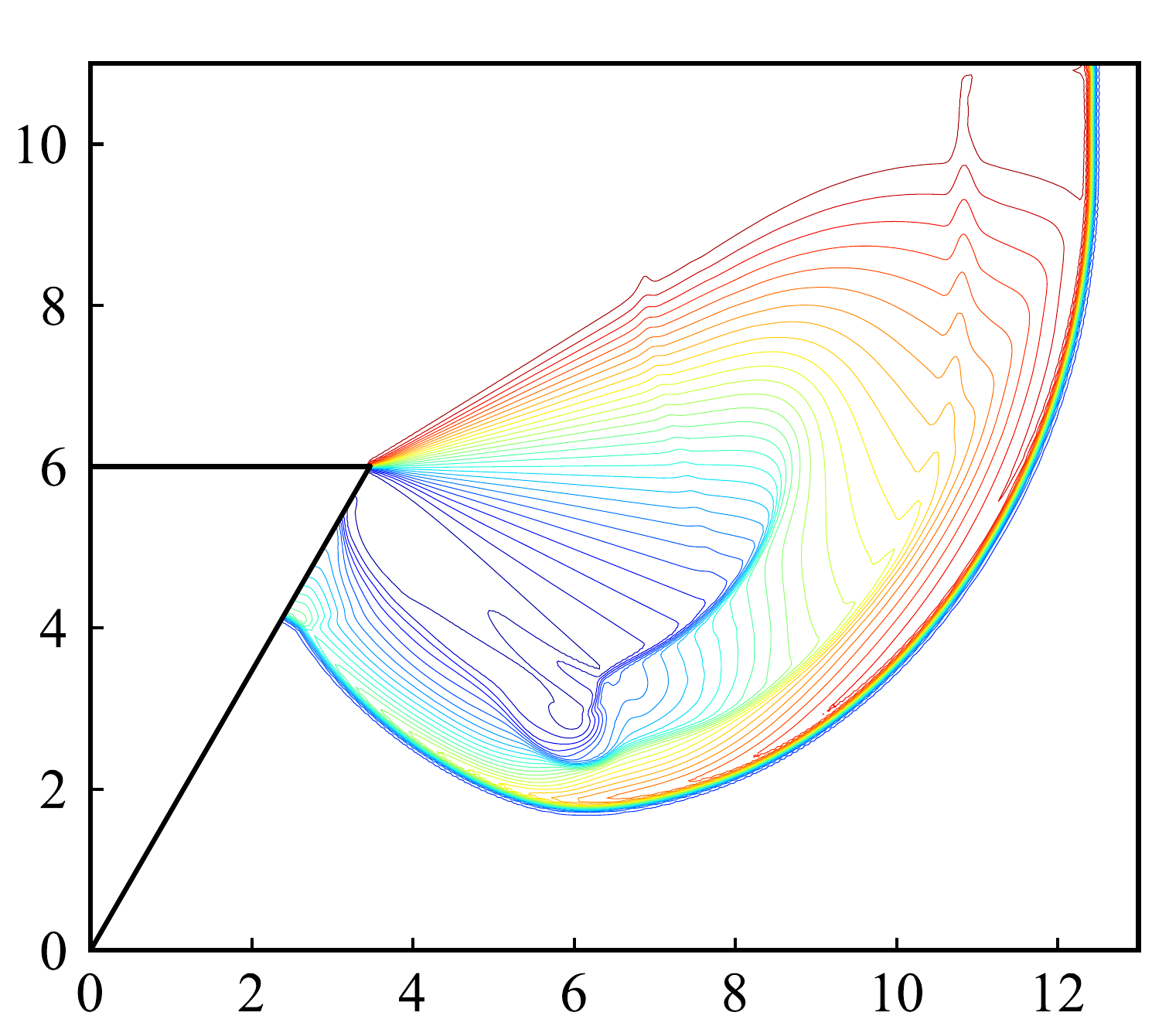}
		\caption{Cell averages.}
		\label{fig:Ex_SDW1}
	\end{subfigure}
	\hfill
	\begin{subfigure}{0.32\textwidth}
		\includegraphics[width=\textwidth]{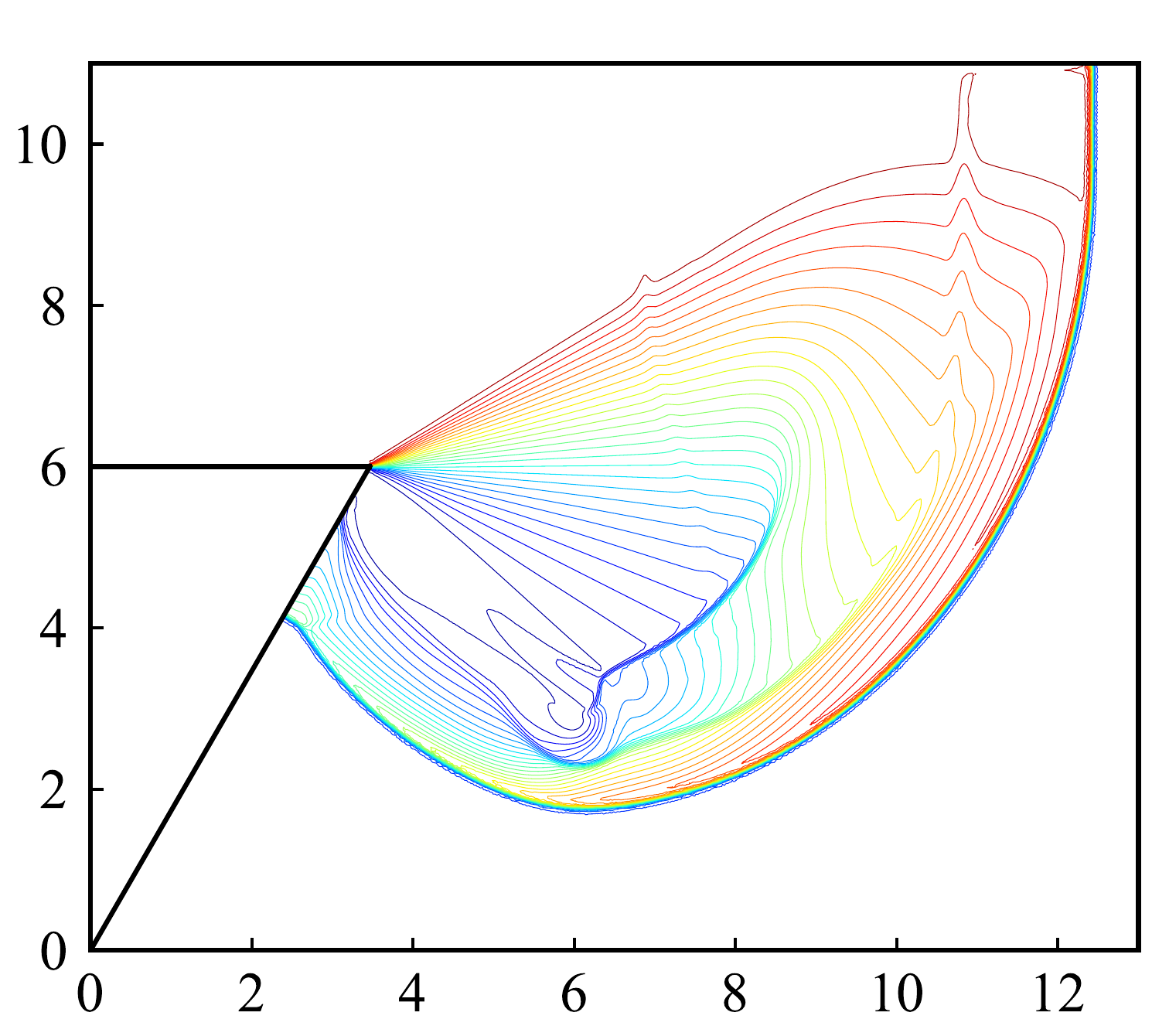}
		\caption{Point values.}
		\label{fig:Ex_SDW2}
	\end{subfigure}
	
	\caption{Shock diffraction: (a) sample mesh; (b--c) density contours (30 levels from 0.12 to 8.05).}
	\label{fig:Ex_SDW}
\end{figure}

\subsubsection{Shock reflection and diffraction with Kelvin--Helmholtz instability}\label{Ex:SRDW}  
We finally investigate the shock reflection and diffraction benchmark on the polygonal domain with vertices 
\[
(0.1,0),\ (0.1,2),\ (2.8,2),\ (2.8,0),\ (1.2,0),\ (1.2,\sqrt{3}/3),\ (0.2,0).
\]
The domain is discretized using an unstructured triangular mesh consisting of $N=303{,}116$ cells and characteristic mesh size $h=1/160$. \Cref{fig:Ex_SRDWMesh} illustrates a coarsened sample mesh ($h=0.1$). Initially, a planar right-moving shock located at $x=0.2$ separates the unperturbed states
\[
(\rho,v_1,v_2,p)=
\begin{cases}
	(8,8.25,0,116.5), & x\le 0.2,\\
	(1.4,0,0,1), & \text{otherwise}.
\end{cases}
\]
An inflow condition is imposed along $x=0.1$, an outflow condition along $x=2.8$, exact data prescribed on $\{0.1\leq x \leq 0.2, y=0\}$ and on the top boundary $y=2$, while reflective slip-wall conditions are enforced along the remaining boundaries. The simulation is carried out to the final time $t=0.245$, with the resulting density contours and discrete admissibility diagnostics reported in \Cref{fig:Ex_SRDW1,fig:Ex_SRDW2} and \Cref{tab:Ex-Euler-Min,tab:Ex-Euler-StageMin,tab:Ex-Euler-StageLimiter}, respectively.

\begin{figure}[!htb]
	\centering	
    \begin{subfigure}{0.3\textwidth}
		\includegraphics[width=\textwidth]{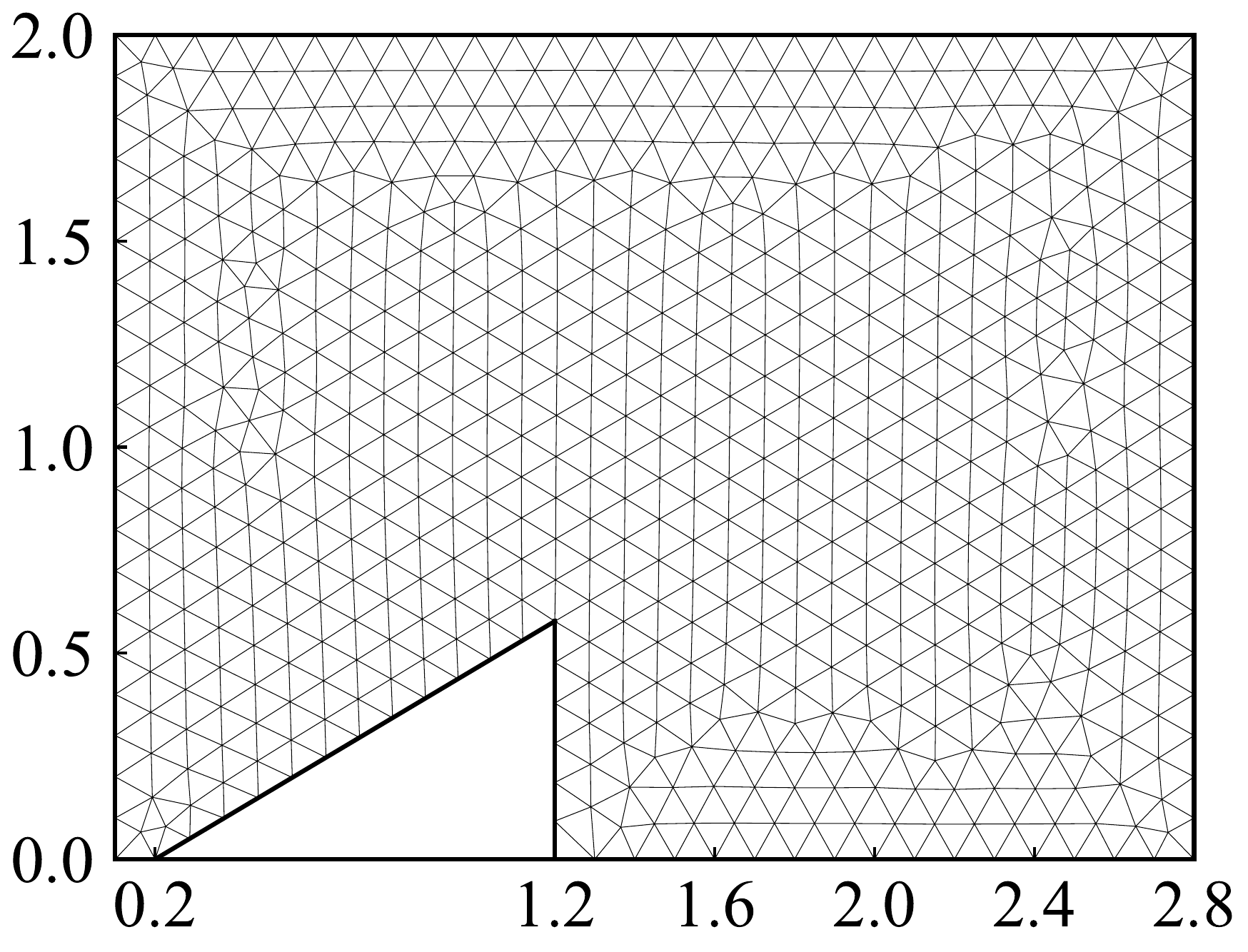}
		\caption{Coarsened sample mesh.}
		\label{fig:Ex_SRDWMesh}
	\end{subfigure}
    \hfill
	\begin{subfigure}{0.33\textwidth}
		\includegraphics[width=\textwidth]{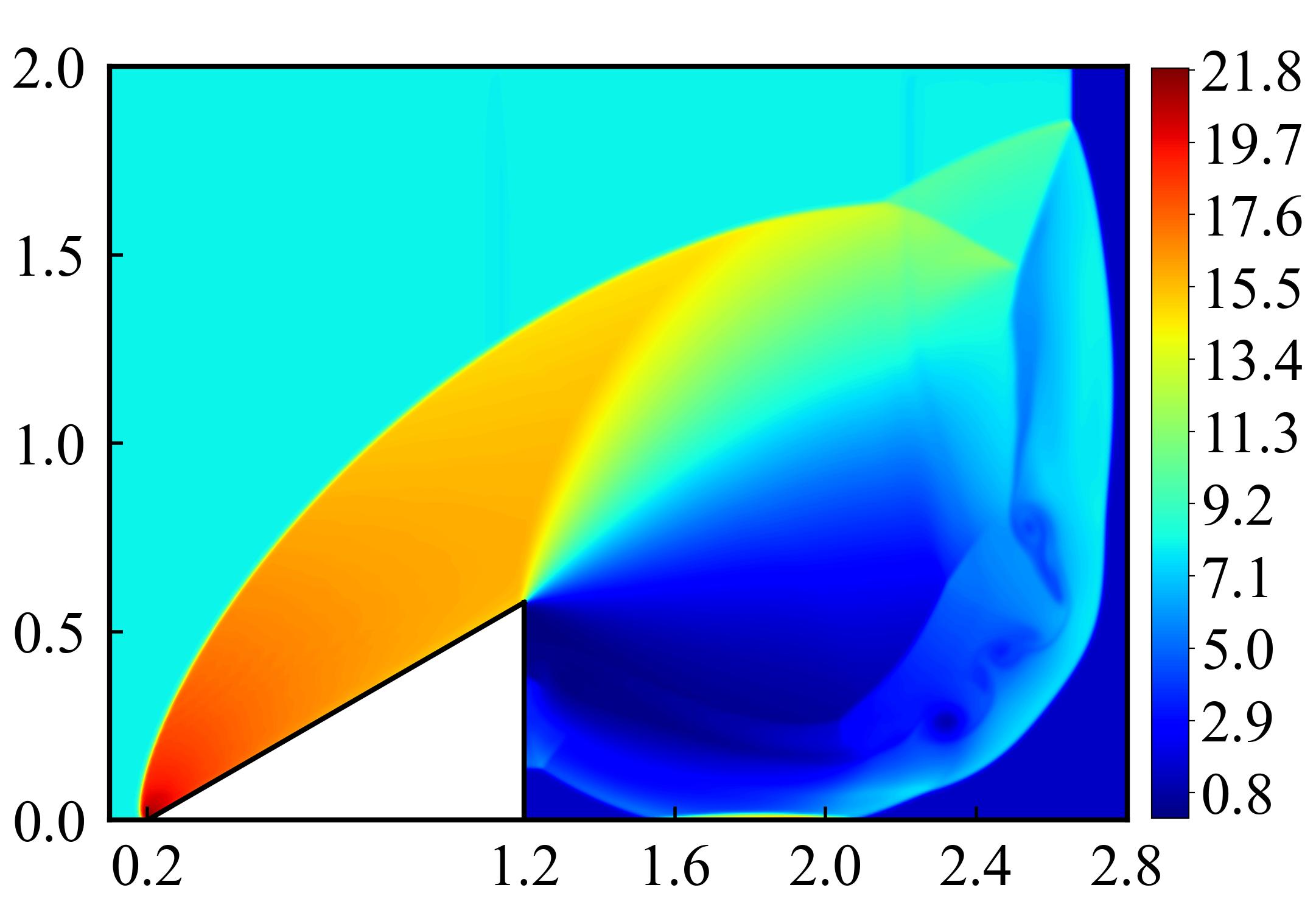}
		\caption{Cell averages.}
		\label{fig:Ex_SRDW1}
	\end{subfigure}
	\hfill
	\begin{subfigure}{0.33\textwidth}
		\includegraphics[width=\textwidth]{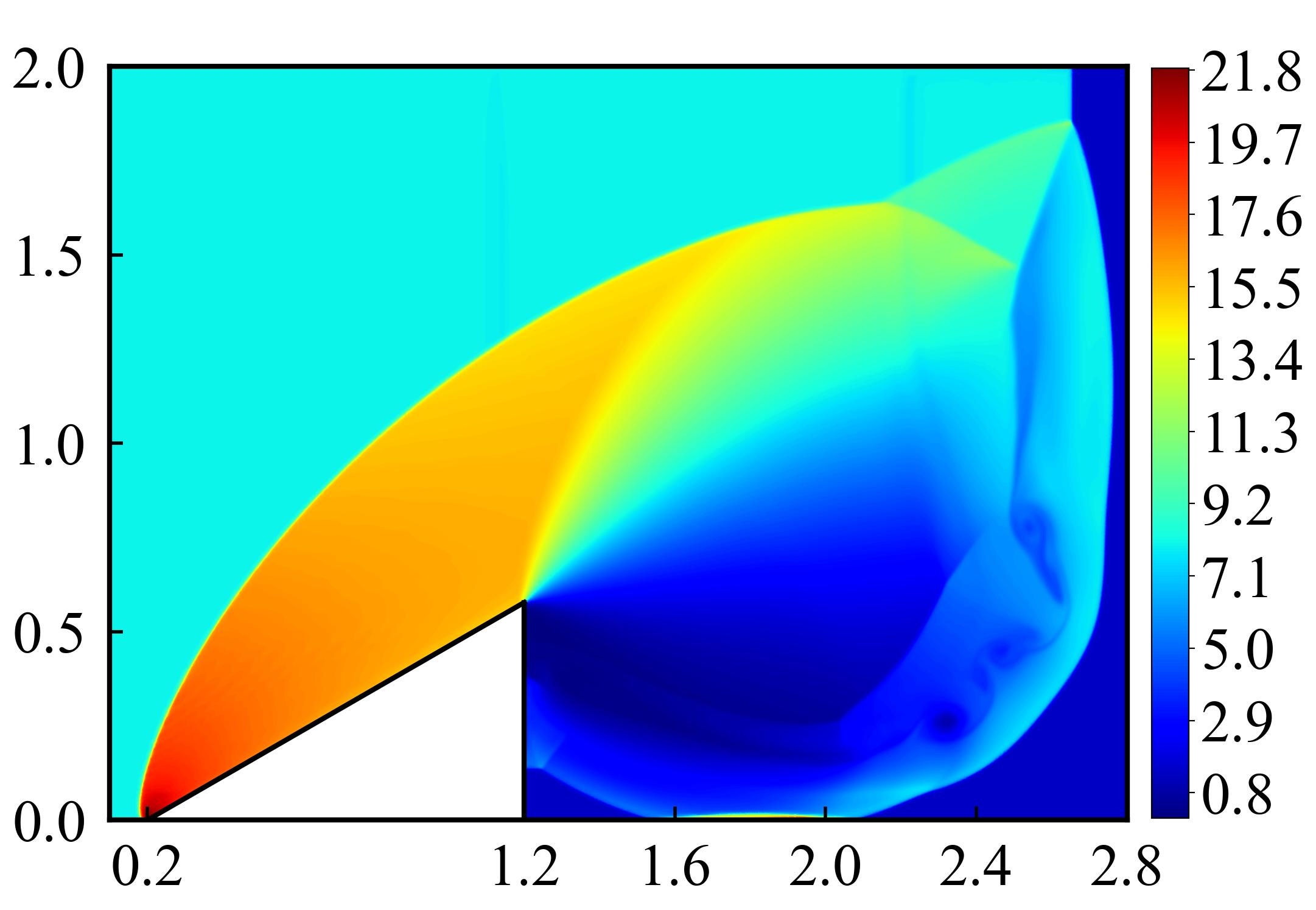}
		\caption{Point values.}
		\label{fig:Ex_SRDW2}
	\end{subfigure}
	
	\caption{Shock reflection and diffraction: (a) coarsened sample mesh; (b--c) density contours. }
	\label{fig:Ex_SRDW}
\end{figure}

\section{Conclusions}\label{sec:conclusion}

We have developed an \emph{a priori} invariant-domain-preserving framework for hybrid discretizations of hyperbolic conservation laws that couple conservative cell averages with cell-boundary point values. The analysis addresses the mismatch between two different update mechanisms within a single hybrid step: the cell average is evolved conservatively, whereas the point values may evolve under a non-conservative operator. The point-value update is handled via transformed variables, using either the globally defined barrier--Legendre inverse or a practical admissibility-preserving map. The conservative cell-average update is controlled via cell average decomposition (CAD), geometric quasilinearization, local \emph{a priori} scaling, and admissible trace inputs. Under an explicit trace-based CFL condition, these two mechanisms combine to preserve the hybrid admissible set; the extension to high-order SSP Runge--Kutta methods follows from the standard convex-stage argument.

A central part of the analysis is the obstruction--construction mechanism. We have proved that admissible continuous boundary traces, used directly via the single-state physical flux while internal reconstruction values remain uncontrolled, cannot by themselves guarantee a conservative invariant-domain update under a prescribed CFL number. Since this obstruction already appears in scalar linear advection, it reflects a geometric-algebraic limitation of continuous trace information rather than a peculiarity of nonlinear systems. It identifies the missing local one-sided control, or numerical dissipation, needed by the conservative flux. The present admissible-trace construction supplies this control explicitly by adaptively using locally admissible, generally discontinuous flux inputs before the conservative update is performed; whereas convex limiting and flux blending provide the same type of IDP numerical-flux contribution implicitly when the blending is active.

The abstract assumptions are realized constructively for third-order hybrid schemes on triangular, Cartesian, convex quadrilateral, and general convex polygonal meshes. The corresponding formulas provide explicit CAD weights, trace-state inputs, local wave-speed bounds, and CFL indicators. The local scaling preserves the cell average, is inactive on already admissible CAD data, and has a quantified high-order accuracy cost when activated. Thus, the construction provides a local pre-flux limiting mechanism rather than a post-update repair.

The numerical experiments corroborate the theoretical analysis. Smooth scalar tests verify the designed third-order accuracy, while nonsmooth scalar and compressible Euler benchmarks confirm the preservation of scalar bounds and \revblue{the strict positivity of density and pressure}. The reported final-time admissibility, all-stage minima, and limiter-activation fractions clearly distinguish the intrinsic admissibility properties from the amount of active numerical stabilization used.  
Our future work includes the study of the fully discrete multi-entropy stability with the present IDP method under the recent EPO framework in \cite{wu2026epo}.

\section*{Statements and Declarations}
\smallskip\noindent\textbf{Competing interests.} The authors declare that they have no competing interests.

\smallskip\noindent\textbf{Data availability.} {No external datasets were used. The numerical configurations needed to reproduce the reported tests are specified in Section~\ref{sec:numerics}; implementation details for the limiters and \revblue{boundary treatments} are given in Appendix~\ref{sec:appendix-implementation}.}

\vspace{8mm}

\appendix
\section{Transform construction and selected point-block proofs}\label{sec:appendix-point}

\subsection{Constructive admissibility-transform theorem}\label{sec:proof_Psi_constructible}

\begin{theorem}[Constructive admissibility transform]\label{thm:Psi_constructible}
Assume that $\Gint$ and $G$ are given by \eqref{eq:G_open_strict}--\eqref{eq:G_def_g} and that each $g_j$ is concave and $C^2$ on an open neighborhood of $G$. For any $\mu>0$, define the barrier--Legendre potential
\begin{equation}\label{eq:Phi_barrier}
\Phi_\mu(u):=\frac12\,\|u\|_2^2-\mu\sum_{j\in\mathcal{J}}\log\bigl(g_j(u)\bigr),\qquad u\in \Gint,
\end{equation}
and the associated transform
\begin{equation}\label{eq:Psi_barrier_def}
\Psi_\mu(u):=\nabla\Phi_\mu(u)=u-\mu\sum_{j\in\mathcal{J}}\frac{\nabla g_j(u)}{g_j(u)},\qquad u\in \Gint.
\end{equation}
Then $\Psi=\Psi_\mu$ satisfies (A$\Psi$1)--(A$\Psi$3) on $\Gint$ in the following sense:
\begin{itemize}
\item[(i)] \emph{Global bijection.} $\Psi:\Gint\to\mathbb{R}^m$ is bijective. Moreover, for each $w\in\mathbb{R}^m$, the inverse is characterized by the unique minimizer
\begin{equation}\label{eq:Psi_inv_argmin}
\Psi^{-1}(w)=\arg\min_{u\in \Gint}\ \Bigl\{\Phi_\mu(u)-w\cdot u\Bigr\}.
\end{equation}
\item[(ii)] \emph{Regularity and Lipschitz inverse.} $\Psi\in C^1(\Gint)$ and $D\Psi(u)=\nabla^2\Phi_\mu(u)$ is symmetric positive definite for all $u\in \Gint$. In particular, $\Psi$ is a $C^1$ diffeomorphism. Furthermore, the inverse $\Psi^{-1}:\mathbb{R}^m\to \Gint$ is globally $1$-Lipschitz continuous:
\begin{equation}\label{eq:Psi_inv_Lipschitz}
\|\Psi^{-1}(w_1)-\Psi^{-1}(w_2)\|_2\le \|w_1-w_2\|_2,\qquad \forall w_1,w_2\in\mathbb{R}^m.
\end{equation}
\item[(iii)] \emph{PDE compatibility.} For any smooth solution $u(x,t)\in \Gint$ of \eqref{eq:cl}, the transformed variable $w=\Psi(u)$ satisfies \eqref{eq:w_quasilinear}.
\end{itemize}
\end{theorem}

\begin{proof}
		We divide the proof into five steps.
		
		\smallskip
		\noindent\emph{Step 1: convexity, strong convexity, and barrier blow-up of
			$\Phi_\mu$.}
		Since $g_j\in C^2$ and $g_j(u)>0$ on $\Gint$, the function $-\log(g_j(u))$ is well
		defined and $C^2$ on $\Gint$.  Direct differentiation gives
		\begin{equation}\label{eq:hess_loggj}
			\nabla^2\bigl(-\log g_j(u)\bigr)=\frac{\nabla g_j(u)\nabla g_j(u)^\top}{g_j(u)^2}-\frac{\nabla^2 g_j(u)}{g_j(u)}.
		\end{equation}
		Because $g_j$ is concave, its Hessian satisfies $\nabla^2 g_j(u)\preceq 0$ on $\Gint$.
		Hence the second term in \eqref{eq:hess_loggj} is positive semidefinite, and the
		first term is also positive semidefinite.  Therefore,
		$\nabla^2(-\log g_j(u))\succeq 0$ for all $u\in \Gint$, i.e., $-\log g_j$ is convex.
		Consequently, $\Phi_\mu$ in \eqref{eq:Phi_barrier} is convex as the sum of the
		convex quadratic $\frac12\|u\|_2^2$ and convex barrier terms.
		
		Moreover, combining \eqref{eq:Phi_barrier} and \eqref{eq:hess_loggj} yields the
		Hessian of $\Phi_\mu$:
		\begin{equation}\label{eq:hess_Phi}
			\nabla^2\Phi_\mu(u)=I+\mu\sum_{j\in\mathcal{J}}\Bigl(\frac{\nabla g_j(u)\nabla g_j(u)^\top}{g_j(u)^2}-\frac{\nabla^2 g_j(u)}{g_j(u)}\Bigr)\succeq I.
		\end{equation}
		Thus $\Phi_\mu$ is \emph{$1$-strongly convex} on $\Gint$ (in the Euclidean norm).
		
		Next, we establish the barrier blow-up property needed below.
		Let $\{u_n\}\subset \Gint$ be any sequence that converges to a boundary point
		$u_\ast\in\partial \Gint$.
		Because $u_\ast \in \overline{\Gint} \setminus \Gint$,
		it follows from the representation \eqref{eq:G_open_strict} that $g_j(u_\ast)=0$ for at least one index $j\in\mathcal{J}$.
		By the continuity of $g_j$, we have
		$g_j(u_n)\to 0^+$, hence $-\log g_j(u_n)\to +\infty$, and therefore
		\begin{equation}\label{eq:Phi_blowup_boundary}
			\Phi_\mu(u_n)\to +\infty\qquad\text{whenever }u_n\to u_\ast\in\partial \Gint.
		\end{equation}

		\smallskip
		\noindent\emph{Step 2: surjectivity of $\Psi$ via a strictly convex minimization
			problem.}
		Fix an arbitrary $w\in\mathbb{R}^m$ and consider the function
		\begin{equation}\label{eq:Hw_def}
			H_w(u):=\Phi_\mu(u)-w\cdot u,\qquad u\in \Gint.
		\end{equation}
		Since $\Phi_\mu$ is $1$-strongly convex, $H_w$ is also $1$-strongly convex.
		We claim that $H_w$ admits a unique minimizer in $\Gint$.
		
		To prove existence, let $\{u_n\}\subset \Gint$ be any minimizing sequence for $H_w$.
		We first show that $H_w$ is coercive on $\Gint$, which implies that every minimizing sequence is
		bounded.
		
		Fix any reference point $u_0\in \Gint$.  By the concavity of $g_j$ and its differentiability,
		we have the global affine upper bound
		\begin{equation}\label{eq:gj_affine_upper}
			g_j(u)\le g_j(u_0)+\nabla g_j(u_0)\cdot(u-u_0),\qquad \forall u\in \Gint.
		\end{equation}
		In particular, for all $u\in \Gint$,
		\[
		g_j(u)\le g_j(u_0)+\|\nabla g_j(u_0)\|_2\,\|u-u_0\|_2
		\le C_j\,(1+\|u\|_2)
		\]
		for a constant $C_j>0$ depending only on $u_0$ and $g_j$.
		Therefore
		\[
		\log g_j(u)\le \log C_j + \log(1+\|u\|_2),\qquad u\in \Gint,
		\]
		and hence the barrier term admits the lower bound
		\[
		-\mu\sum_{j\in\mathcal{J}}\log g_j(u)
		\ge -\mu\sum_{j\in\mathcal{J}}\log C_j-\mu\,|\mathcal{J}|\,\log(1+\|u\|_2)
		=:-C_0-C_1\log(1+\|u\|_2).
		\]
		Combining with the quadratic term and using $-w\cdot u\ge -\|w\|_2\,\|u\|_2$,
		we obtain
		\[
		H_w(u)\ge \frac12\|u\|_2^2-\|w\|_2\,\|u\|_2 -C_0-C_1\log(1+\|u\|_2) \xrightarrow[]{\|u\|_2\to\infty}+\infty.
		\]
		Thus $H_w$ is coercive on $\Gint$, so any minimizing sequence is bounded.
		Since $\mathbb R^m$ is finite-dimensional, we may extract a convergent
		subsequence (still denoted by $u_n$) with $u_n\to u^*\in \overline{\Gint}$.
		If $u^*\in\partial \Gint$, then \eqref{eq:Phi_blowup_boundary} implies
		$H_w(u_n)\to +\infty$, contradicting the minimizing property.
		Hence $u^*\in \Gint$.
		By the continuity of $H_w$ on $\Gint$, we conclude that $u^*$ is a minimizer of $H_w$.
		
		Uniqueness follows from strict (in fact strong) convexity: if both $u^*$ and
		$v^*$ minimize $H_w$, then by strict convexity
		$H_w\bigl(\tfrac12(u^*+v^*)\bigr)<\tfrac12H_w(u^*)+\tfrac12H_w(v^*)$
		unless $u^*=v^*$.  Therefore, the minimizer is unique; we denote it by
		$u(w)\in \Gint$.
		
		Finally, the first-order optimality condition for the unconstrained minimization
		over the open set $\Gint$ is
		\[
		\nabla H_w\bigl(u(w)\bigr)=\nabla\Phi_\mu\bigl(u(w)\bigr)-w=0,
		\]
		which is equivalent to $\Psi\bigl(u(w)\bigr)=w$ by \eqref{eq:Psi_barrier_def}.
		Since $w\in\mathbb{R}^m$ was arbitrary, $\Psi$ is surjective and
		\eqref{eq:Psi_inv_argmin} holds.

		\smallskip
		\noindent\emph{Step 3: injectivity of $\Psi$.}
		Let $u_1,u_2\in \Gint$ satisfy $\Psi(u_1)=\Psi(u_2)$.  By the mean-value theorem,
		\[
		\bigl(\Psi(u_1)-\Psi(u_2)\bigr)\cdot (u_1-u_2)=\int_0^1 (u_1-u_2)^\top\nabla^2\Phi_\mu\bigl(u_2+\theta(u_1-u_2)\bigr)(u_1-u_2)\,d\theta.
		\]
		Using \eqref{eq:hess_Phi}, the integrand is bounded below by $\|u_1-u_2\|_2^2$.
		Therefore
		\[
		0=\bigl(\Psi(u_1)-\Psi(u_2)\bigr)\cdot (u_1-u_2)\ge \|u_1-u_2\|_2^2,
		\]
		which implies $u_1=u_2$.  Hence $\Psi$ is injective.
		
		\smallskip
		\noindent\emph{Step 4: $C^1$ diffeomorphism and Lipschitz inverse.}
		Since each $g_j$ is $C^2$ on a neighborhood of $\overline{\Gint}$, we have
		$\Psi=\nabla\Phi_\mu\in C^1(\Gint)$.  The Jacobian is
		$D\Psi(u)=\nabla^2\Phi_\mu(u)$, which is symmetric positive definite by
		\eqref{eq:hess_Phi}.  Hence $D\Psi(u)$ is invertible for all $u\in \Gint$.
		By Steps 2--3, $\Psi$ is a bijection between open sets.  The inverse function
		theorem then yields that $\Psi$ is a $C^1$ diffeomorphism and $\Psi^{-1}\in C^1$
		with $D\Psi^{-1}(w)=\bigl(D\Psi(\Psi^{-1}(w))\bigr)^{-1}$.
		
		To prove the global Lipschitz bound \eqref{eq:Psi_inv_Lipschitz}, take any
		$w_1,w_2\in\mathbb{R}^m$ and set $u_i=\Psi^{-1}(w_i)\in \Gint$.  Then
		$w_i=\Psi(u_i)=\nabla\Phi_\mu(u_i)$.  Using again the mean-value theorem and
		\eqref{eq:hess_Phi}, we obtain the \emph{strong monotonicity} estimate
		\[
		(w_1-w_2)\cdot (u_1-u_2)
		=\int_0^1 (u_1-u_2)^\top\nabla^2\Phi_\mu\bigl(u_2+\theta(u_1-u_2)\bigr)(u_1-u_2)\,d\theta
		\ge \|u_1-u_2\|_2^2.
		\]
		Applying the Cauchy--Schwarz inequality, $(w_1-w_2)\cdot(u_1-u_2)\le\|w_1-w_2\|_2\,\|u_1-u_2\|_2$,
		we conclude
		$\|u_1-u_2\|_2\le \|w_1-w_2\|_2$, which is \eqref{eq:Psi_inv_Lipschitz}.
		
		\smallskip
		\noindent\emph{Step 5: PDE compatibility (A$\Psi$3).}
		Let $u(x,t)\in \Gint$ be a smooth solution of \eqref{eq:cl} and set $w=\Psi(u)$.  By
		the chain rule,
		\begin{equation}\label{eq:chain_rule_wt}
			\partial_t w = D\Psi(u)\,\partial_t u,\qquad \partial_{x_i}w=D\Psi(u)\,\partial_{x_i}u.
		\end{equation}
		Since $u_t+\sum_i A_i(u)u_{x_i}=0$, multiplying by $D\Psi(u)$ gives
		$D\Psi(u)u_t+\sum_i D\Psi(u)A_i(u)u_{x_i}=0$.
		Substituting \eqref{eq:chain_rule_wt} and inserting
		$I=(D\Psi(u))^{-1}D\Psi(u)$ yields
		\[
		\partial_t w+\sum_{i=1}^d D\Psi(u)A_i(u)(D\Psi(u))^{-1}\,\partial_{x_i}w=0,
		\]
		which is exactly \eqref{eq:w_quasilinear}.
		This completes the proof.
	\end{proof}

\subsection{Point-layer transform and admissibility-preserving map pairs}\label{sec:transform_pairs}

This appendix details the concrete point-layer transform and admissibility-preserving map pairs used in theoretical proofs and numerical implementations; see \Cref{tab:transform-summary}.

\begin{table}[!htb]
\centering
\caption{Point-layer maps and targets.}
\label{tab:transform-summary}
\scriptsize
\setlength{\tabcolsep}{4pt}
\begin{tabularx}{\textwidth}{>{\raggedright\arraybackslash}p{0.16\textwidth}>{\raggedright\arraybackslash}p{0.26\textwidth}>{\raggedright\arraybackslash}p{0.23\textwidth}>{\raggedright\arraybackslash}p{0.14\textwidth}>{\raggedright\arraybackslash}X}
\toprule
Setting & Forward map or transformed variables & Inverse or admissibility-preserving map & Point target & Layer \\
\midrule
Exact smooth setting & $w=\Psi(u)$ on $\Gint_{\mathrm{sc}}$ & exact inverse $\Psi^{-1}:\mathbb R^m\to \Gint_{\mathrm{sc}}$ & $\Gint_{\mathrm{sc}}$ & exact smooth layer \\
Practical scalar setting (clipped) & affine map \eqref{u:scalar} & clipped admissibility-preserving map \eqref{w:scalar} & $G_{\mathrm{sc}}$ & scalar closed-target layer \\
Practical scalar setting (smooth) & smooth sigmoid and logit pair \eqref{u:scalar2} on an enlarged interval & corresponding smooth inverse \eqref{w:scalar2} & slightly enlarged convex interval & scalar enlarged-target layer \\
Practical Euler setting & smooth transform \eqref{w:euler1} & corresponding same smooth inverse \eqref{w:euler2} & open admissible interior $\Gint_{\rm Euler}$ & Euler point layer \\
\bottomrule
\end{tabularx}
\end{table}

For scalar conservation laws with invariant interval $G_{\mathrm{sc}}=[U_{\min},U_{\max}]$, we consider the following point-layer formulas:
\begin{itemize}
			\item \emph{Smooth diffeomorphism.}
			A smooth diffeomorphism is given by
			\begin{equation}\label{w:scalar2}
				u=\Psi^{-1}(w)=(U_{\max}-U_{\min}) \sigma(w) + U_{\min} \qquad \forall w\in \mathbb{R},
			\end{equation}
			where $\sigma(w)=(1+e^{-w})^{-1}$, with its inverse
			\begin{equation}\label{u:scalar2}
				w = \Psi(u) = \log\left( \frac{u - U_{\min}}{U_{\max} - u} \right).
			\end{equation}
			\item \emph{Affine and clipped pair.}
			We define the affine forward map
			\begin{equation}\label{u:scalar}
				w=\mathcal T(u)=\frac{u-U_{\min}}{U_{\max}-U_{\min}}, \qquad u\in G,
			\end{equation}
			paired with the clipped admissibility-preserving map
			\begin{equation}\label{w:scalar}
				u=\widetilde\Psi^{-1}(w)=(U_{\max}-U_{\min})\,\max\{0,\min\{1,w\}\}+U_{\min}, \qquad \forall w\in\mathbb{R}.
			\end{equation}
					\end{itemize}

		For the compressible Euler equations, the point-layer transform is given by
\begin{equation}\label{w:euler1}
	w=(q,v,s)^\top, ~~
	q=\log\Big(\exp\Big\{\frac{\rho}{\rho_{\mathrm{ref}}}\Big\}-1\Big), ~~ 
	s=\log \frac{p}{p_{\mathrm{ref}}} - \gamma \log \frac{\rho}{\rho_{\mathrm{ref}}},
\end{equation}
    with the velocity $v\in \mathbb{R}^d$.
	The corresponding inverse map yields
	\begin{equation}\label{w:euler2}
		\rho = \rho_{\mathrm{ref}} \log(e^q+1) >0, ~~~ p=p_{\mathrm{ref}} \Big(\frac{\rho}{\rho_{\mathrm{ref}}}\Big)^\gamma e^s >0,
	\end{equation}
	for any $w \in \mathbb{R}^m$. 
    In the numerical simulations, we set $p_{\mathrm{ref}} = \rho_{\mathrm{ref}} = 1$.
	
	The softplus-based mapping is $C^\infty$ and defines a global diffeomorphism between the open admissible interior $\Gint_{\mathrm{Euler}}=\{u=(\rho,\rho v,E)^\top:\ \rho>0,\ p(u)>0\}$ and $\mathbb{R}^m$.


\section{Geometric and CAD supporting material}\label{sec:appendix-geometry}

\subsection{Algebraic verification of the quadrilateral CAD}\label{sec:quad_cad_appendix}

Let $v_i$ denote the vertex shared by the two consecutive edges $e_i$ and $e_{i+1}$ (with the cyclic convention $e_5=e_1$ and $\Delta_5=\Delta_1$), and let $m_i:=x_K^{*,i}$ denote the midpoint of $e_i$. For every $\phi$ in the target third-order polynomial space, the restriction of $\phi$ to each edge is a univariate quadratic polynomial. Hence, Simpson's rule yields
\begin{equation}\label{eq:quad-simpson}
\bar \phi_K^{(i)}=\frac16\bigl(\phi(v_{i-1})+4\phi(m_i)+\phi(v_i)\bigr),\qquad i=1,\dots,4,
\end{equation}
with the cyclic convention $v_0=v_4$. Substituting the edge averages \eqref{eq:quad-simpson} into the CAD ansatz yields
\begin{equation}\label{eq:quad-expanded-cad}
\begin{aligned}
	&\sum_{i=1}^4 \lambda_{K,i}\,\bar \phi_K^{(i)}+\sum_{i=1}^4 \beta_{K,i}\,\phi(m_i)+\frac49 \phi(x_K^{\mathrm{ctr}})
	=\\
	& \qquad \qquad
	\sum_{i=1}^4 \frac{\lambda_{K,i}+\lambda_{K,i+1}}{6}\,\phi(v_i)
	+
	\sum_{i=1}^4 \left(\frac23\lambda_{K,i}+\beta_{K,i}\right)\phi(m_i)
	+
	\frac49 \phi(x_K^{\mathrm{ctr}}).
\end{aligned}
\end{equation}
Using the explicit definitions of $\lambda_{K,i}$ and $\beta_{K,i}$ along with the geometric relations $\Delta_1+\Delta_3=\Delta_2+\Delta_4=2|K|$, we obtain the coefficient identities
\begin{equation*}\label{eq:quad-coeff-identities}
\frac{\lambda_{K,i}+\lambda_{K,i+1}}{6}=\frac{\Delta_{i+1}}{9S_\Delta},\qquad
\frac23\lambda_{K,i}+\beta_{K,i}=\frac{2(\Delta_i+\Delta_{i+1})}{9S_\Delta},\qquad i=1,\dots,4,
\end{equation*}
where the indices are taken cyclically. Consequently, \eqref{eq:quad-expanded-cad} matches identically the point-value decomposition established in Section~\ref{subsec:quad-tensor}, up to the cyclic relabeling of the vertex terms induced by the edge-based numbering. This proves the exactness property of \Cref{prop:quad-A3}, completing the algebraic verification.

\section{Auxiliary proofs}

\subsection{Proof of the obstruction theorem}\label{sec:proof_cont_flux}

\begin{proof}[Proof of \Cref{thm:necessity_discont}]
	Consider an arbitrary cell $K$. 
    The space $\mathbb{W}_K$ possesses sufficient internal degrees of freedom to represent the state $u_K$ constructed below, independently of the prescribed boundary data and cell-average values.
    Since $\bm a\neq 0$, there exists at least one edge
	$e_{\mathrm{in}}\subset\partial K$ satisfying the inflow condition $\,\bm a\cdot n_{K,e_{\mathrm{in}}}<0$.
    Select an edge quadrature point $x_{e_{\mathrm{in}},\nu_{\mathrm{in}}}$ (such as the edge midpoint).
	We specify the boundary data according to
	\[
	u_{K,e_{\mathrm{in}},\nu_{\mathrm{in}}}^n=1,
	\qquad
	u_{K,e,\nu}^n=0 \quad \text{for all } (e,\nu)\neq(e_{\mathrm{in}},\nu_{\mathrm{in}}),
	\]
	ensuring that all boundary quadrature values belong to $G$.
	For an arbitrarily small $\varepsilon>0$, we prescribe the cell average
	\[
	\bar u_K^n = 1 - \beta_{K,s_0}\varepsilon \in (0,1)\subset G,
	\]
	and set all internal control values to zero except $u_{K,s_0}^{n,\ast}:={u}_K^n(x_{K,s_0}^\ast)$, which is uniquely determined by the CAD identity \eqref{eq:gcd_point}:
	\[
	\bar u_K^n
	=
	\sum_{e\in \mathcal{E}_K} \lambda_{K,e}\sum_{\nu=1}^{N_q}\omega_{e,\nu} u_{K,e,\nu}^n
	+
	\sum_{s=1}^{S_K} \beta_{K,s} u_{K,s}^{n,\ast},
	\]
	where ${u}_{K,e,\nu}^n:={u}_K^n(x_{e,\nu})$ and
	${u}_{K,s}^{n,\ast}:={u}_K^n(x_{K,s}^\ast)$.
	Substituting these prescribed values into the CAD identity yields
	\[
	u_{K,s_0}^{n,\ast}
	=
	\frac{\bar u_K^n-\sum_{e \in \mathcal{E}_K} \lambda_{K,e}\sum_{\nu}\omega_{e,\nu} u_{K,e,\nu}^n}{\beta_{K,s_0}}
	=
	\frac{1-\lambda_{K,e_{\mathrm{in}}}\omega_{e_{\mathrm{in}},\nu_{\mathrm{in}}}}{\beta_{K,s_0}} - \varepsilon.
	\]
	Since the CAD weights are strictly positive, it follows that $\lambda_{K,e_{\mathrm{in}}} < \sum_{e\in\mathcal{E}_K}\lambda_{K,e} = 1-\sum_{s=1}^{S_K}\beta_{K,s} \le 1-\beta_{K,s_0}$.
	Using the property that $\omega_{e_{\mathrm{in}},\nu_{\mathrm{in}}} \le 1$, we obtain $1-\lambda_{K,e_{\mathrm{in}}}\omega_{e_{\mathrm{in}},\nu_{\mathrm{in}}} > \beta_{K,s_0}$.
	Consequently, choosing $\varepsilon>0$ sufficiently small guarantees that $u_{K,s_0}^{n,\ast}>1$, and hence $u_{K,s_0}^{n,\ast}\notin G$, establishing the interior admissibility violation.

	Next, we consider the continuous-flux update \eqref{eq:avg_cont_FE}.
	For the linear flux $F(u)=\bm a\,u$, we have
	\[
	\bar u_{K}^{\mathrm{cont}}
	=
	\bar u_K^n
	-
	\frac{\Delta t}{|K|}
	\sum_{e\in \mathcal{E}_K} |e|
	\sum_{\nu=1}^{N_q}\omega_{e,\nu} \, u_{K,e,\nu}^n \,(\bm a\cdot n_{K,e}).
	\]
	Since all boundary values vanish except at $(e_{\mathrm{in}},\nu_{\mathrm{in}})$, we obtain
	\[
	\bar u_{K}^{\mathrm{cont}}
	=
	1-\beta_{K,s_0}\varepsilon
	+
	\Delta t\,
	\frac{|e_{\mathrm{in}}|}{|K|}\,
	\omega_{e_{\mathrm{in}},\nu_{\mathrm{in}}}\,
	\big|\bm a\cdot n_{K,e_{\mathrm{in}}}\big|.
	\]
	Therefore, $\bar u_{K}^{\mathrm{cont}}\le 1$ is satisfied if and only if 
	$
	\Delta t
	\le
	\beta_{K,s_0}\varepsilon\,
	\frac{|K|}{|e_{\mathrm{in}}|\omega_{e_{\mathrm{in}},\nu_{\mathrm{in}}}|\bm a\cdot n_{K,e_{\mathrm{in}}}|}.
	$ 
	Since $\varepsilon>0$ can be chosen arbitrarily small, there exists no fixed positive CFL number that guarantees $\bar u_{K}^{\mathrm{cont}}\in[0,1]$, which completes the proof.
\end{proof}

\subsection{Accuracy part of the classic limiter lemma}\label{sec:PiK_classic_accuracy}

\begin{proof}[Proof of \Cref{lem:PiK_props1}(iii)]
We define 
$
E_K:=\|u_K^n-u(\cdot,t^n)\|_{L^\infty(K)}.
$ 
Since $u(\cdot,t^n)$ is smooth and takes values within $\Gint$, and $u_K^n \to u(\cdot,t^n)$ uniformly, for sufficiently small $h$, both the cell average $\bar{u}_K^n$ and all control values of $u_K^n$ evaluated on $\mathcal X_K$ belong to a bounded convex set $\mathcal U\subset\mathbb R^m$ containing the range of the exact solution, on which each $g_j$ is Lipschitz continuous. Let $L_g:=\max_{j\in\mathcal J}L_{g,j}$ denote the common Lipschitz constant on $\mathcal U$.

For each $x\in\mathcal{X}_K$ and $j\in\mathcal{J}$, we introduce the auxiliary scaling parameter 
$\theta_{j,x}:=\sup\{\theta\in[0,1]:(1-\theta)s_0+\theta g_j(u_K^n(x))\ge 0\}.$
A direct computation yields $\theta_{j,x}=1$ if $g_j(u_K^n(x))\ge 0$, and $\theta_{j,x}=\frac{s_0}{s_0-g_j(u_K^n(x))}$ otherwise. 
In either case, the deviation from unity satisfies
$$1-\theta_{j,x}\le\frac{(-g_j(u_K^n(x)))_+}{s_0},$$
where $(y)_+:=\max(0,y)$. 
Set $\theta^*:=\min_{x\in\mathcal{X}_K}\min_{j\in\mathcal{J}}\theta_{j,x}$. By the concavity of each $g_j$ and the slack condition $g_j(\bar{u}_K^n)\ge s_0$, we have for all $x\in\mathcal{X}_K$ and $j\in\mathcal{J}$,
$$g_j\bigl(\bar{u}_K^n+\theta^*(u_K^n(x)-\bar{u}_K^n)\bigr)\ge(1-\theta^*)s_0+\theta^*g_j(u_K^n(x))\ge 0,$$
which implies that $\bar{u}_K^n+\theta^*(u_K^n(x)-\bar{u}_K^n)\in G$ for all $x\in\mathcal{X}_K$. 
Consequently, the definition of $\theta_K^\texttt{c}$ in \eqref{eq:theta_def1} guarantees that $\theta_K^\texttt{c}\ge\theta^*$, which yields
$$1-\theta_K^\texttt{c}\le 1-\theta^*=\max_{x\in\mathcal{X}_K}\max_{j\in\mathcal{J}}(1-\theta_{j,x})\le\frac{1}{s_0}\max_{x\in\mathcal{X}_K}\max_{j\in\mathcal{J}}(-g_j(u_K^n(x)))_+.$$

Since $u(x,t^n)\in G$ for all $x\in K$, it holds that $g_j(u(x,t^n))\ge 0$. Hence, for all $x\in\mathcal X_K$ and $j\in\mathcal J$, 
$
(-g_j(u_K^n(x)))_+
\le \bigl|g_j(u_K^n(x))-g_j(u(x,t^n))\bigr|
\le L_g E_K.
$ 
Consequently, 
$
1-\theta_K^{\mathrm c}\le \frac{L_g}{s_0}E_K.
$ 
Using the relation 
$
\hat u_K^{\mathrm c}-u_K^n=(1-\theta_K^{\mathrm c})(\bar u_K^n-u_K^n).
$ 
Because the exact solution $u(\cdot,t^n)$ is smooth, both $\|u(\cdot,t^n)\|_{L^\infty(K)}$ and $|\bar u_K^n|$ are uniformly bounded; for sufficiently small $h$, the approximation $u_K^n$ remains in the same bounded neighborhood. Hence, $\|u_K^n-\bar u_K^n\|_{L^\infty(K)}\le C_0$ for some constant $C_0$ independent of $h$. Combining this bound with the preceding estimate yields
\[
\|\hat u_K^{\mathrm c}-u_K^n\|_{L^\infty(K)}
\le (1-\theta_K^{\mathrm c})\|u_K^n-\bar u_K^n\|_{L^\infty(K)}
\le \frac{C}{s_0}E_K,
\]
which completes the proof.
\end{proof}

\subsection{Lipschitz continuity of the classic limiter}\label{sec:PiK_Lipschitz1}

\begin{proof}[Proof of \Cref{lem:PiK_Lipschitz1}]
	We first focus on the estimate for a single-constraint, single-point scaling parameter.
	For a fixed constraint index $j\in\mathcal{J}$ and a given control point $x\in\mathcal{X}_K$, we consider the concave scalar function
$
	h(\theta):=g_j\bigl(a+\theta(b_x-a)\bigr),$ with 
	$\theta\in[0,1].
	$
	By \eqref{eq:PiK_slack_assumption}, $h(0)=g_j(a)\ge s_0>0$.
	We define the maximal admissible scaling parameter associated with this constraint and point by
	\begin{equation}\label{eq:theta_jx_def}
		\theta_{j,x}(a,b_x):=\sup\Bigl\{\theta\in[0,1]:\ h(\theta)\ge 0\Bigr\}\in(0,1].
	\end{equation}
	Since $h$ is concave and continuous, the set $\{\theta:\ h(\theta)\ge 0\}$ is an interval; hence,
	$\theta_{j,x}(a,b_x)$ is well-defined.
	Moreover, if $h(1)=g_j(b_x)\ge 0$, then $\theta_{j,x}(a,b_x)=1$; otherwise, $h(1)<0$ and the
	concavity of $h$ ensures the existence of a unique $\theta_{j,x}(a,b_x)\in(0,1)$ satisfying
	$h(\theta_{j,x})=0$.
	
	\smallskip
	\noindent\emph{Step 1: Concave-root perturbation bound}.
	Let $h:[0,1]\to\mathbb{R}$ be a concave function with $h(0)=g_j(a)=c>0$, and set
	$\theta^\ast:=\sup\{\theta\in[0,1]:h(\theta)\ge 0\}$.
	Then, for all $s\in[0,1]$,
	\begin{equation}\label{eq:concave_root_bound}
		|\theta^\ast-s|
		\le \frac{|h(s)|}{c} = \frac{|h(s)|}{g_j(a)}.
	\end{equation}
	Indeed, if $s\le \theta^\ast$, then $h(s)\ge 0$ and the concavity yields
	$h(s)\ge (1-s/\theta^\ast)\,h(0) = c(\theta^\ast-s)/\theta^\ast \ge c(\theta^\ast-s)$, and thus
	$\theta^\ast-s\le h(s)/c$.
	Conversely, if $s\ge \theta^\ast$, then $h(s)\le 0$, and the secant inequality for concave functions yields
	\[
	\frac{h(s)-h(\theta^\ast)}{s-\theta^\ast}
	\le
	\frac{h(\theta^\ast)-h(0)}{\theta^\ast-0}
	=
	-\frac{c}{\theta^\ast}\le -c,
	\]
	which implies that $-h(s)\ge c(s-\theta^\ast)$, thereby establishing \eqref{eq:concave_root_bound}.
	
	\smallskip
	\noindent\emph{Step 2: Lipschitz bound for $\theta_{j,x}$.}
	Let $\theta_{j,x}:=\theta_{j,x}(a,b_x)$ and $\theta'_{j,x}:=\theta_{j,x}(a',b'_x)$.
	We establish the estimate
	\begin{equation}\label{eq:theta_jx_Lipschitz}
		|\theta_{j,x}-\theta'_{j,x}|
		\le
		\frac{L_{g,j}}{s_0}\bigl(\|a-a'\|_2+\|b_x-b'_x\|_2\bigr).
	\end{equation}
	To prove this, we distinguish three cases.
	
	\smallskip
	\noindent\underline{Case~1: Both constraints are inactive}.
	If $g_j(b_x)\ge 0$ and $g_j(b'_x)\ge 0$, then $\theta_{j,x}=\theta'_{j,x}=1$ and
	\eqref{eq:theta_jx_Lipschitz} holds trivially.
	
	\smallskip
	\noindent\underline{Case~2: Both constraints are active}.
	If $g_j(b_x)<0$ and $g_j(b'_x)<0$, then $\theta_{j,x},\theta'_{j,x}\in(0,1)$ and
	\[
	g_j\bigl(a+\theta_{j,x}(b_x-a)\bigr)=0,
	\qquad
	g_j\bigl(a'+\theta'_{j,x}(b'_x-a')\bigr)=0.
	\]
	Applying \eqref{eq:concave_root_bound} to $h$ with $s=\theta'_{j,x}$ yields
	\[
	\bigl|\theta_{j,x}(a,b_x)-\theta_{j,x}(a',b'_x)\bigr|
	\le
	\frac{\bigl|g_j\bigl((1-\theta'_{j,x})a+\theta'_{j,x}b_x\bigr)\bigr|}{g_j(a)}.
	\]
	Subtracting the vanishing value at the primed state and applying the Lipschitz continuity of $g_j$ on~$\mathcal{U}$,
	it follows that
	\begin{align*}
		\bigl|g_j\bigl((1-\theta'_{j,x})a+\theta'_{j,x}b_x\bigr)\bigr|
		&=
		\bigl|g_j\bigl((1-\theta'_{j,x})a+\theta'_{j,x}b_x\bigr)
		-
		g_j\bigl((1-\theta'_{j,x})a'+\theta'_{j,x}b'_x\bigr)\bigr|\\
		&\le
		L_{g,j}\Bigl(\|(1-\theta'_{j,x})(a-a')\|_2+\|\theta'_{j,x}(b_x-b'_x)\|_2\Bigr)\\
		&\le L_{g,j}\bigl(\|a-a'\|_2+\|b_x-b'_x\|_2\bigr).
	\end{align*}
	Combining this with the slack bound $g_j(a)\ge s_0$ establishes \eqref{eq:theta_jx_Lipschitz}.
	
	\smallskip
	\noindent\underline{Case~3: Mixed activity}.
	Assume without loss of generality that $g_j(b_x)<0$ and $g_j(b'_x)\ge 0$
	(the reverse case is completely symmetric). Then $\theta'_{j,x}=1$.
	Since $h$ is concave and $h(\theta_{j,x})=0$, we have
	\[
	0=h(\theta_{j,x})
	\ge (1-\theta_{j,x})h(0)+\theta_{j,x}h(1)
	\ge (1-\theta_{j,x})\,s_0+\theta_{j,x}g_j(b_x),
	\]
	which implies
	\[
	1-\theta_{j,x}
	\le
	\frac{|g_j(b_x)|}{s_0}
	\le
	\frac{|g_j(b_x)-g_j(b'_x)|}{s_0}
	\le
	\frac{L_{g,j}}{s_0}\,\|b_x-b'_x\|_2.
	\]
	Thus, $|\theta_{j,x}-\theta'_{j,x}|=1-\theta_{j,x}\le \frac{L_{g,j}}{s_0}\|b_x-b'_x\|_2$, which conforms to \eqref{eq:theta_jx_Lipschitz}.
	
	\smallskip
	\noindent\emph{Step 3: Lipschitz bound for the classic scaling parameter $\theta^\texttt{c}_K$.}
	Because $G=\bigcap_{j\in\mathcal{J}}G_j$ and the control set $\mathcal{X}_K$ is finite, the definition \eqref{eq:theta_def1}
	can be equivalently written as
	\[
	\theta_K=\min\Bigl\{\,1,\ \min_{j\in\mathcal{J}}\min_{x\in\mathcal{X}_K}\theta_{j,x}(a,b_x)\Bigr\}.
	\]
	Since the pointwise minimum of finitely many Lipschitz continuous functions is Lipschitz continuous, combining \eqref{eq:theta_jx_Lipschitz}
	over all $j\in\mathcal J$ and $x\in\mathcal X_K$ yields
	\begin{equation}\label{eq:thetaK_Lipschitz}
		|\theta_K-\theta'_K|
		\le
		\Bigl(\max_{j\in\mathcal{J}}\frac{L_{g,j}}{s_0}\Bigr)
		\Bigl(\|a-a'\|_2+\max_{x\in\mathcal{X}_K}\|b_x-b'_x\|_2\Bigr).
	\end{equation}
	
	\smallskip
	\noindent\emph{Step 4: Lipschitz bound for the limited control values.}
	For each control point $x\in\mathcal{X}_K$,
	\[
	(\Pi_K u_K)(x)=a+\theta_K(b_x-a),\qquad
	(\Pi_K v_K)(x)=a'+\theta'_K(b'_x-a').
	\]
	Using the identity
	\[
	a+\theta_K(b_x-a)-\bigl(a'+\theta'_K(b'_x-a')\bigr)
	=
	(1-\theta_K)(a-a')+\theta_K(b_x-b'_x)+(\theta_K-\theta'_K)(b'_x-a'),
	\]
	and the bound $\|b'_x-a'\|_2\le 2M = 2 \max\{\|a\|_2,\|a'\|_2,\max_{x\in\mathcal{X}_K}\|b_x\|_2,\max_{x\in\mathcal{X}_K}\|b'_x\|_2\}$, we obtain
	\begin{align*}
		\|(\Pi_K u_K)(x)-(\Pi_K v_K)(x)\|_2
		&\le
		\|a-a'\|_2+\|b_x-b'_x\|_2 + 2M\,|\theta_K-\theta'_K|\\
		&\le
		\Bigl(1+2M\,\max_{j\in\mathcal{J}}\frac{L_{g,j}}{s_0}\Bigr)
		\Bigl(\|a-a'\|_2+\max_{x\in\mathcal{X}_K}\|b_x-b'_x\|_2\Bigr),
	\end{align*}
	where the final inequality follows from \eqref{eq:thetaK_Lipschitz}.
	Taking the maximum over all $x\in\mathcal{X}_K$ establishes \eqref{eq:PiK_Lipschitz_est}, which completes the proof.
\end{proof}

\subsection{Lipschitz continuity of the alternative limiter}\label{sec:PiK_Lipschitz2}

\begin{proof}[Proof of \Cref{lem:PiK_Lipschitz2}]
	We define
	\[
	d:=\max_{e\in\mathcal E_K}\max_{1\le \nu\le N_q}\|b_{e,\nu}-b'_{e,\nu}\|_2.
	\]
    The proof follows the lines of \Cref{lem:PiK_Lipschitz1}, with the additional control required for the auxiliary state $u_K^*$ in \eqref{eq:theta_def2} when $\beta_K > 0$.
	
	\smallskip
	\noindent\emph{Step 0: Lipschitz bound for the auxiliary state $u_K^*$.}
	When $\beta_K=0$, no auxiliary constraint is present in \eqref{eq:theta_def2}, and this step is vacuous.
	For $\beta_K>0$,
    by the edge quadrature exactness condition in Assumption~(A3) and the positivity of the quadrature weights,
	\[
	\bar u_{K,e}=\sum_{\nu=1}^{N_q}\omega_{e,\nu} b_{e,\nu},
	\qquad
	\bar v_{K,e}=\sum_{\nu=1}^{N_q}\omega_{e,\nu} b'_{e,\nu},
	\]
	which yields
	\[
	\|\bar u_{K,e}-\bar v_{K,e}\|_2
	\le
	\sum_{\nu=1}^{N_q}\omega_{e,\nu}\|b_{e,\nu}-b'_{e,\nu}\|_2
	\le d.
	\]
	Using the definition of \eqref{eq:PiKa_ustar_def_u}, we obtain
	\begin{align}
		\|u_K^*-v_K^*\|_2
		&\le
		\frac{1}{\beta_K}
		\left(
		\|a-a'\|_2+\sum_{e\in\mathcal E_K}\lambda_{K,e}\|\bar u_{K,e}-\bar v_{K,e}\|_2
		\right)\notag\\
		&\le
		\frac{1}{\beta_K}\|a-a'\|_2+\frac{\sum_e\lambda_{K,e}}{\beta_K}d
		=
		\frac{1}{\beta_K}\|a-a'\|_2+\frac{1-\beta_K}{\beta_K}d\notag\\
		&\le
		\beta_K^{-1}\bigl(\|a-a'\|_2+d\bigr).
		\label{eq:PiKa_ustar_Lipschitz}
	\end{align}
	
	\smallskip
	\noindent\emph{Step 1: Single-constraint, single-candidate scaling parameter.}
	For each $j\in\mathcal J$, we introduce the candidate index set
	\[
	\mathcal I_K:=
	\begin{cases}
		\{(e,\nu): e\in\mathcal E_K,\ 1\le \nu\le N_q\}, & \beta_K=0,\\
		\{(e,\nu): e\in\mathcal E_K,\ 1\le \nu\le N_q\}\cup\{\ast\}, & \beta_K>0.
	\end{cases}
	\]
	For each $\alpha\in\mathcal I_K$, we define
	\[
	z_\alpha:=
	\begin{cases}
		b_{e,\nu}, & \alpha=(e,\nu),\\
		u_K^*, & \alpha=\ast,
	\end{cases}
	\qquad
	z'_\alpha:=
	\begin{cases}
		b'_{e,\nu}, & \alpha=(e,\nu),\\
		v_K^*, & \alpha=\ast.
	\end{cases}
	\]
	For a fixed pair $(j,\alpha)$, we define
	\[
	h(\theta):=g_j\bigl(a+\theta(z_\alpha-a)\bigr),\qquad \theta\in[0,1],
	\]
	and the maximal admissible scaling parameter
	\begin{equation}\label{eq:PiKa_theta_ja_def}
		\theta_{j,\alpha}(a,z_\alpha)
		:=
		\sup\{\theta\in[0,1]: h(\theta)\ge 0\}.
	\end{equation}
	As in the proof of \Cref{lem:PiK_Lipschitz1}, $h$ is concave and continuous, which ensures that $\theta_{j,\alpha}$ is well-defined.
	
	We recall the concave-root perturbation bound from \eqref{eq:concave_root_bound}: if $h$ is concave on $[0,1]$ with $h(0)=g_j(a)=c>0$ and $\theta^\ast:=\sup\{\theta\in[0,1]:h(\theta)\ge0\}$, then for any $s\in[0,1]$,
	\begin{equation}\label{eq:PiKa_concave_root_bound}
		|\theta^\ast-s|\le \frac{|h(s)|}{c}.
	\end{equation}
	
	\smallskip
	\noindent\emph{Step 2: Lipschitz bound for $\theta_{j,\alpha}$.}
	Set
$
	\theta_{j,\alpha}:=\theta_{j,\alpha}(a,z_\alpha)$ and $
	\theta'_{j,\alpha}:=\theta_{j,\alpha}(a',z'_\alpha).
	$ 
	Following the same argument as in the proof of \Cref{lem:PiK_Lipschitz1} (three cases: both inactive, both active, and mixed activity), we obtain
	\begin{equation}\label{eq:PiKa_theta_ja_local}
		|\theta_{j,\alpha}-\theta'_{j,\alpha}|
		\le
		\frac{L_{g,j}}{s_0}\bigl(\|a-a'\|_2+\|z_\alpha-z'_\alpha\|_2\bigr).
	\end{equation}
	We now bound $\|z_\alpha-z'_\alpha\|_2$ uniformly with respect to $\alpha$:
	\[
	\|z_\alpha-z'_\alpha\|_2\le d
	\quad\text{if }\alpha=(e,\nu),
	\]
	and, for $\beta_K>0$, \eqref{eq:PiKa_ustar_Lipschitz} yields
	\[
	\|z_\ast-z'_\ast\|_2=\|u_K^*-v_K^*\|_2
	\le \beta_K^{-1}(\|a-a'\|_2+d).
	\]
	Therefore, for all $\alpha\in\mathcal I_K$,
	\begin{equation}\label{eq:PiKa_zalpha_uniform}
		\|a-a'\|_2+\|z_\alpha-z'_\alpha\|_2
		\le
		\Gamma_K\,(\|a-a'\|_2+d),
	\end{equation}
	with $\Gamma_K$ as defined in the statement of \Cref{lem:PiK_Lipschitz2}. Combining \eqref{eq:PiKa_theta_ja_local} and
	\eqref{eq:PiKa_zalpha_uniform} yields
	\begin{equation}\label{eq:PiKa_theta_ja_Lipschitz}
		|\theta_{j,\alpha}-\theta'_{j,\alpha}|
		\le
		\frac{L_{g,j}}{s_0}\,\Gamma_K\,(\|a-a'\|_2+d).
	\end{equation}
	
	\smallskip
	\noindent\emph{Step 3: Lipschitz bound for the full scaling parameter $\theta_K^\texttt{a}$.}
	From the definition \eqref{eq:theta_def2}, we can write
	\[
	\theta_K^\texttt{a}
	=
	\min\Bigl\{1,\ \min_{j\in\mathcal J}\min_{\alpha\in\mathcal I_K}\theta_{j,\alpha}(a,z_\alpha)\Bigr\},
	\]
	and analogously for the primed reconstruction. Since the minimum of finitely many Lipschitz continuous functions preserves the Lipschitz property with the maximum Lipschitz constant, it follows from \eqref{eq:PiKa_theta_ja_Lipschitz} that
	\begin{equation}\label{eq:PiKa_thetaK_Lipschitz}
		|\theta_K^\texttt{a}-(\theta_K^\texttt{a})'|
		\le
		\Bigl(\max_{j\in\mathcal J}\frac{L_{g,j}}{s_0}\Bigr)\Gamma_K\,(\|a-a'\|_2+d).
	\end{equation}
	
	\smallskip
	\noindent\emph{Step 4: Lipschitz bound for the limited edge control values.}
	Let $x=x_{e,\nu}$ be an arbitrary edge quadrature point. Then, the limited reconstructions at $x$ satisfy
	\[
	(\Pi_K^\texttt{a}u_K)(x)=a+\theta_K^\texttt{a}(b_x-a),
	\qquad
	(\Pi_K^\texttt{a}v_K)(x)=a'+(\theta_K^\texttt{a})'(b'_x-a').
	\]
	Using the algebraic identity
	\[
	a+\theta(b_x-a)-\bigl(a'+\theta'(b'_x-a')\bigr)
	=
	(1-\theta)(a-a')+\theta(b_x-b'_x)+(\theta-\theta')(b'_x-a'),
	\]
	evaluated at $\theta=\theta_K^\texttt{a}$ and $\theta'=(\theta_K^\texttt{a})'$, together with the bound $\|b'_x-a'\|_2\le 2M_{\texttt{a}}$, it holds that
	\begin{align*}
		\bigl\|(\Pi_K^\texttt{a}u_K)(x)-(\Pi_K^\texttt{a}v_K)(x)\bigr\|_2
		&\le
		\|a-a'\|_2+\|b_x-b'_x\|_2+2M_{\texttt{a}}\,|\theta_K^\texttt{a}-(\theta_K^\texttt{a})'|\\
		&\le
		\Bigl[
		1+2M_{\texttt{a}}
		\Bigl(\max_{j\in\mathcal J}\frac{L_{g,j}}{s_0}\Bigr)\Gamma_K
		\Bigr](\|a-a'\|_2+d),
	\end{align*}
	where the final inequality follows from \eqref{eq:PiKa_thetaK_Lipschitz}. Taking the maximum over all edge quadrature points $x_{e,\nu}\in\mathcal X_K^{\mathrm{edg}}$ establishes \eqref{eq:PiKa_Lipschitz_est}, which completes the proof.
\end{proof}

\subsection{Comparison between trace speeds and proxy CFL surrogates for smooth solutions}\label{sec:proxy_bridge}

\begin{proposition}[Smooth-regime bridge between trace speeds and cell-average proxies]\label{prop:proxy_smooth_bridge}
Assume the hypotheses of \Cref{cor:limiter_small_smooth}. In addition, suppose that, on the time interval of interest, the exact solution takes values in a compact set $\mathcal K$ \revblue{such that $\mathcal K\subset\Gint$ and $\operatorname{dist}(\mathcal K,\partial G)>0$}, and that for every unit normal $n$ the wave-speed bound $u\mapsto \alpha(u,n)$ from Assumption~(A5) is locally Lipschitz on $\mathcal K$. Then, for every cell $K$, edge $e\in\mathcal E_K$, and edge quadrature node $x_{e,\nu}$,
\begin{align*}
\bigl|\alpha(\hat u_K^n(x_{e,\nu}),n_{K,e})-\alpha(u(x_{e,\nu},t^n),n_{K,e})\bigr| &\le C h^{k+1},\\
\bigl|\alpha(\hat u_K^n(x_{e,\nu}),n_{K,e})-\alpha(\bar u_K^n,n_{K,e})\bigr| &\le C h,
\end{align*}
where $C$ is independent of $h$ for sufficiently fine meshes. 
As a consequence, for solutions that are smooth in the interior of the domain, 
the practical pre-stage cell-average wave-speed proxies and the trace-based speeds exhibit the same asymptotic scaling.
\end{proposition}

\begin{proof}
By the local Lipschitz continuity of $\alpha(\cdot,n_{K,e})$ on $\mathcal K$ and \Cref{cor:limiter_small_smooth},
\[
\bigl|\alpha(\hat u_K^n(x_{e,\nu}),n_{K,e})-\alpha(u_K^n(x_{e,\nu}),n_{K,e})\bigr|
\le L\,\bigl|\hat u_K^n(x_{e,\nu})-u_K^n(x_{e,\nu})\bigr|
\le C h^{k+1}.
\]
The assumed high-order reconstruction accuracy on smooth data yields
\[
\bigl|\alpha(u_K^n(x_{e,\nu}),n_{K,e})-\alpha(u(x_{e,\nu},t^n),n_{K,e})\bigr|\le C h^{k+1},
\]
and the first estimate follows from the triangle inequality.

For the second estimate, the local Lipschitz continuity yields
\[
\bigl|\alpha(\hat u_K^n(x_{e,\nu}),n_{K,e})-\alpha(\bar u_K^n,n_{K,e})\bigr|
\le L\,\bigl|\hat u_K^n(x_{e,\nu})-\bar u_K^n\bigr|.
\]
Adding and subtracting the exact state gives
\[
\bigl|\hat u_K^n(x_{e,\nu})-\bar u_K^n\bigr|
\le
\bigl|\hat u_K^n(x_{e,\nu})-u(x_{e,\nu},t^n)\bigr|
+
\bigl|u(x_{e,\nu},t^n)-\bar u_K^n\bigr|.
\]
The first term is $O(h^{k+1})$ owing to the first estimate. Since $u(\cdot,t^n)$ is smooth, the difference between a point value on $\partial K$ and the cell average over $K$ is $O(h)$ on a shape-regular mesh. Hence, the second estimate is $O(h)$, which completes the proof.
\end{proof}

\begin{corollary}[Uniform comparison of pre-stage proxies on smooth interior solutions]\label{cor:proxy_global_bridge}
Assume the hypotheses of \Cref{prop:proxy_smooth_bridge}. Let $\{a_i^h\}_{i\in I_h}$ and $\{b_i^h\}_{i\in I_h}$ be two finite families indexed by the same index set $I_h$, where $a_i^h$ denotes a trace-based speed and $b_i^h$ denotes the corresponding pre-stage cell-average proxy, and suppose that
\[
|a_i^h-b_i^h|\le Ch
\qquad \forall i\in I_h
\]
for sufficiently fine meshes. Then
\[
\Bigl|\max_{i\in I_h} a_i^h-\max_{i\in I_h} b_i^h\Bigr|\le Ch.
\]
In particular, if one defines in the triangular setting
\[
\alpha_{\max}^{\mathrm{tr},n}:=\max_{K,e,\nu}\alpha\bigl(\hat u_K^n(x_{e,\nu}),n_{K,e}\bigr),
\qquad
\tilde\alpha_{\max}^{\,n}:=\max_{K,e}\alpha(\bar u_K^n,n_{K,e}),
\]
and in the Cartesian setting
\[
\alpha_{\ell,\max}^{\mathrm{tr},n}:=\max_{K,\nu}\alpha_\ell\bigl(\hat u_{K,\ell,\nu}^n\bigr),
\qquad
\tilde\alpha_{\ell,\max}^{\,n}:=\max_K\alpha_\ell(\bar u_K^n),\qquad \ell=1,2,
\]
then the pre-stage proxies in \eqref{eq:CFL-used} exhibit the same asymptotic scaling as the corresponding trace-based maxima on smooth interior solutions.
\end{corollary}

\begin{proof}
For any two finite families, there holds
\[
\Bigl|\max_{i\in I_h} a_i^h-\max_{i\in I_h} b_i^h\Bigr|\le \max_{i\in I_h}|a_i^h-b_i^h|.
\]
The assumed $O(h)$ bound therefore directly extends from the local trace and proxy pairs to their globalized maxima. Applying this observation to the families of trace speeds and cell-average proxies arising in the triangular and Cartesian discretizations yields the stated result.
\end{proof}
\label{sec:appendix-conservative}

\subsection{Verification of Assumption~(A5) for the scalar invariant domain}\label{sec:weakLF_scalar}
	To prove \Cref{prop:A5_scalar}, we verify condition \eqref{eq:weakLF2} for any state $u \in G_{\mathrm{sc}}$ and boundary states $u_j^* \in \{U_{\min}, U_{\max}\}$. 

    For any unit vector $n\in\mathbb S^{d-1}$ and any $\alpha\ge\tilde\alpha(u,n)$, we have
	\begin{equation*}
		\alpha |u - u_j^*| \ge \left| \left(F(u) - F(u_j^*)\right) \cdot n \right| \quad \text{for all } u \in G_{\mathrm{sc}},
	\end{equation*}
	which holds trivially when $u = u_j^*$ and follows directly from \eqref{eq:Roe_speed} otherwise. 
	
	Due to the interval structure of $G_{\mathrm{sc}}$, the inward normals $n_1^* = 1$ and $n_2^* = -1$ yield the identity $(u - u_j^*)n_j^* = |u - u_j^*|$ for any $u \in G_{\mathrm{sc}}$. 
    Hence, 
	\begin{equation*}
		\alpha (u - u_j^*)n_j^* \ge \left| \left(F(u) - F(u_j^*)\right) \cdot n \right| \ge \mp n_j^* \left( F(u) - F(u_j^*) \right) \cdot n .
	\end{equation*}
	Rearranging the terms and substituting $\zeta(u_j^*) = -n_j^*F(u_j^*)$ recovers the condition \eqref{eq:weakLF2}, which completes the proof.

\subsection{Verification of Assumption~(A5) for the Euler invariant domain}\label{sec:weakLF_Euler}

We now prove \Cref{prop:A5_Euler} for the compressible Euler equations, i.e., we verify condition \eqref{eq:weakLF2} of Assumption~(A5) in \Cref{sec:assumptions}.
Denote $v=(v_1,\dots, v_d)^\top$ and let $u=(\rho,\rho v,E)^\top\in \Gint_{\mathrm{Euler}}$.

The 2D case ($d=2$) is detailed below; the 3D case follows analogously.
For $j=1$, we take $n_1^* = (1,0,0,0)^\top$ and, for definiteness, $u_1^*=(0,0,0,0)^\top$. 
For any unit vector  $n\in\mathbb R^2$ and any $\alpha\ge\tilde\alpha(u,n)$, 
\[
\alpha (u-u_1^*)\cdot n_1^* \pm (F(u)\cdot n)\cdot n_1^*
= \rho\,(\alpha \pm v\cdot n)
\ge \rho\,(\alpha-|v\cdot n|)
\ge \rho\,c_s \,>\, 0 \,=\, \mp \zeta(u_1^*) \cdot n
\]
which follows from \eqref{eq:Euler_speed}.
Hence, \eqref{eq:weakLF2} holds for $j=1$.

Similarly, 
for $j=2$, let 
\[
u_2^*=\Bigl(\rho^*,\rho^*v^*,\rho^*\tfrac{|v^*|^2}{2}\Bigr)^\top,\qquad
n_2^*=\Bigl(\tfrac{|v^*|^2}{2},-v^*,1\Bigr)^\top,
\]
with $\rho^*>0$ and $v^*\in\mathbb R^d$.
Therefore, we have
\begin{align*}
	 \alpha\,(u-u^\ast_2)\cdot n^\ast_2 \pm \bigl(F(u)\cdot n\bigr)\cdot n^\ast_2 
	=&(\alpha \pm v\cdot n)\Big(\frac{p}{\gamma-1}+\frac{\rho}{2}|v^\ast-v|^2\Big) \mp p\,(v^\ast-v)\cdot n \\
	\ge &
	c_s\Big(\frac{p}{\gamma-1}+\frac{\rho}{2}|v^\ast-v|^2\Big)-p\,|v^\ast-v| \\
	\ge &
	2 \, c_s\sqrt{\frac{p}{\gamma-1}\cdot\frac{\rho}{2}|v^\ast-v|^2}-p\,|v^\ast-v|\\
	= &
	\left( \sqrt{\frac{2\gamma}{\gamma-1}}-1\right) p |v^\ast-v| \ge 0 = \mp \zeta(u_2^*) \cdot n .
\end{align*}
The inequality is strict unless $v^*=v$. In the case $v^*=v$, the above expression reduces to
\[
\alpha\,(u-u_2^\ast)\cdot n_2^\ast \pm \bigl(F(u)\cdot n\bigr)\cdot n_2^\ast = (\alpha \pm v\cdot n)\frac{p}{\gamma-1} \ge \frac{c_s p}{\gamma-1} > 0 = \mp \zeta(u_2^*) \cdot n .
\]
Hence, \eqref{eq:weakLF2} holds for $j=2$ as well, with $\tilde\alpha(u,n)=|v\cdot n|+c_s$.

\section{Implementation details and boundary treatments}\label{sec:appendix-implementation}

\subsection{Implementation details for the limiters}\label{sec:IDP_imple}

In the numerical implementation, the limiter is evaluated on a slightly tightened closed numerical admissible set. For scalar invariant intervals, no modification is needed, so the numerical admissible set coincides with $G$. For the compressible Euler simulations, the limiter checks
\[
\rho \ge \varepsilon_\rho,\qquad \mathcal E(u)\ge \varepsilon_{\mathcal E,K},
\]
with
\[
\varepsilon_\rho=10^{-13},\qquad \varepsilon_{\mathcal E,K}=\max\{1,\bar E_K\}\times 10^{-13},
\]
where $\bar E_K$ denotes the cell average of total energy in cell $K$. All other limiter formulations remain identical to those in the main text.

\begin{proposition}[Eventual coincidence of limiter decisions under numerical admissibility cutoffs]\label{prop:cutoff_coincide}
Assume that, on the time interval of interest, the exact solution and the exact states sampled by the limiter remain in a compact subset $\mathcal K$ \revblue{such that $\mathcal K\subset\Gint$ and $\operatorname{dist}(\mathcal K,\partial G)>0$}. Suppose further that the finite collection of cell averages, edge traces, and auxiliary states entering the limiter formulas converges uniformly to the corresponding exact interior states on a shape-regular mesh family. Let the implementation cutoffs be either global constants or cellwise values, provided they are all bounded above by a threshold smaller than the interior margin associated with $\mathcal K$. Then, for all sufficiently fine meshes, the practical admissibility checks with the numerical cutoffs yield identical outcomes to the untightened checks on the untightened closed target. In particular, the practical scaling coefficients computed in floating-point arithmetic eventually coincide with the untightened scaling coefficients in this smooth interior regime.
\end{proposition}

\begin{proof}
Because $\mathcal K$ \revblue{is compact, lies in $\Gint$, and has positive distance from $\partial G$}, every exact state inspected by the limiter maintains a strictly positive distance from $\partial G$. Hence there exist positive margins $c_\rho,c_{\mathcal E}$ (or their scalar analogues) such that all exact limiter-check states remain at least these distances inside the admissible set. Uniform convergence of the finitely many sampled numerical states to the exact ones implies that, for sufficiently fine meshes, every cell average, trace value, and auxiliary state entering the limiter remains separated by at least half of that interior margin from $\partial G$. If all implementation cutoffs, whether global or cellwise, are bounded by thresholds smaller than these margins, then the inequalities defining the practical numerical target are equivalent to the untightened closed-target inequalities on all states actually inspected by the limiter. Therefore, the admissibility status of every sampled state is unchanged, and the maximizing scaling parameters in the practical and untightened limiter definitions coincide.
\end{proof}

The compressible Euler implementation in Section~\ref{sec:numerics} uses the global density cutoff $\varepsilon_\rho=10^{-13}$ together with the cellwise internal-energy cutoff $\varepsilon_{\mathcal E,K}=\max\{1,\bar E_K\}10^{-13}$. Whenever the sampled states remain in a compact interior subset of $G$, \Cref{prop:cutoff_coincide} applies to this choice.

\subsection{Boundary treatments used in the numerical experiments}\label{sec:boundary_closure}

{
On each boundary edge, the point operator requires an exterior trace and the directional derivatives entering \eqref{eq:pt_chain_rule}. The formulas below specify the boundary treatments employed in Section~\ref{sec:numerics}, providing the exterior states and boundary directional derivatives assumed in the point-value analysis.

\revblue{
Throughout this subsection, the exterior states denoted by \(u_e^{\mathrm{ext}}\) are point-operator extensions. They are distinct from the exterior states used in the conservative numerical fluxes. For the conservative update, the interior argument of the boundary flux is the limited interior trace provided by the a priori limiter in Section~\ref{sec:avg:limiter}, while the exterior argument is a boundary flux state prescribed by the corresponding physical boundary condition, or by an admissible approximation thereof, and is required to belong to \(G\). Hence, the states entering the conservative boundary flux are admissible, although the ghost extensions used for the point operator need not themselves be admissible.
}

Fix a boundary edge $e\subset\partial\Omega$ with outward unit normal $n_e$ and unit tangent $\tau_e$. Let $d_e(x)$ denote the signed distance to the supporting line of $e$, let $\pi_e(x):=x-d_e(x)n_e$ be the orthogonal projection onto that line, and let $x_e^{\mathrm{r}}:=x-2d_e(x)n_e$ be the reflected point. We write $u_e^{\mathrm{int}}$ for the interior reconstruction attached to the unique cell adjacent to $e$. \revblue{For this edge, the boundary treatment supplies a function $u_e^{\mathrm{ext}}$ in a one-sided exterior neighborhood. At corner points, the point operator is evaluated edge by edge; each incident boundary edge therefore contributes its own local treatment in its own frame $(n_e,\tau_e)$.}

\begin{enumerate}[(i)]
	\item \emph{Periodic boundary.} If $e$ is paired with an opposite boundary edge $e^\sharp$ through a translation $T_e$, we define
	\[
	u_e^{\mathrm{ext}}(x):=u_{e^\sharp}^{\mathrm{int}}(x+T_e),
	\qquad
	\nabla u_e^{\mathrm{ext}}(x):=\nabla u_{e^\sharp}^{\mathrm{int}}(x+T_e).
	\]
	This is exact for periodic data and locally Lipschitz continuous as a composition with a translation.

    \item \emph{Inflow and exact boundary data.} Let $g_e(x,t)$ be the prescribed boundary state and let $\widetilde g_e$ be a $C^1$ extension to a neighborhood of $e$. \revblue{We define the exterior state by the reflected extension}
	\[
	u_e^{\mathrm{ext}}(x):=2\widetilde g_e(x,t)-u_e^{\mathrm{int}}(x_e^{\mathrm{r}}).
	\]
	At a boundary point $x_\sigma\in e$, this yields the explicit formulas
	\[
	u_e^{\mathrm{ext}}(x_\sigma)=2g_e(x_\sigma,t)-u_e^{\mathrm{int}}(x_\sigma),
	\]
	\[
\partial_{\tau_e}u_e^{\mathrm{ext}}(x_\sigma)=2\partial_{\tau_e}\widetilde g_e(x_\sigma,t)-\partial_{\tau_e}u_e^{\mathrm{int}}(x_\sigma),\quad
\partial_{n_e}u_e^{\mathrm{ext}}(x_\sigma)=2\partial_{n_e}\widetilde g_e(x_\sigma,t)+\partial_{n_e}u_e^{\mathrm{int}}(x_\sigma).
\]
	\revblue{This treatment is consistent with the imposed boundary data and depends in a locally Lipschitz manner on the interior trace and gradient data.}
	
	\item \emph{Outflow boundary.} We impose zero normal extrapolation by
$
	u_e^{\mathrm{ext}}(x):=u_e^{\mathrm{int}}\bigl(\pi_e(x)\bigr).
	$ 
	At $x_\sigma\in e$, this yields
	\[
	u_e^{\mathrm{ext}}(x_\sigma)=u_e^{\mathrm{int}}(x_\sigma),
	\qquad
	\partial_{\tau_e}u_e^{\mathrm{ext}}(x_\sigma)=\partial_{\tau_e}u_e^{\mathrm{int}}(x_\sigma),
	\qquad
	\partial_{n_e}u_e^{\mathrm{ext}}(x_\sigma)=0.
	\]
	The map is again locally Lipschitz continuous as a composition with the projection $\pi_e$.
	
	\item \emph{Reflective boundary (Euler slip wall).} For the compressible Euler state $u=(\rho,m,E)^\top$ with momentum vector $m\in\mathbb{R}^2$, we define the linear reflection operator
	\[
	\mathcal R_e(\rho,m,E)^\top := \bigl(\rho,\, m-2(m\cdot n_e)n_e,\, E\bigr)^\top.
	\]
	We then set
$
	u_e^{\mathrm{ext}}(x):=\mathcal R_e\bigl(u_e^{\mathrm{int}}(x_e^{\mathrm{r}})\bigr).
$ 
	Since $\mathcal R_e$ is linear, the boundary derivatives at $x_\sigma\in e$ satisfy
	\[
	\partial_{\tau_e}u_e^{\mathrm{ext}}(x_\sigma)=\mathcal R_e\,\partial_{\tau_e}u_e^{\mathrm{int}}(x_\sigma),
	\qquad
	\partial_{n_e}u_e^{\mathrm{ext}}(x_\sigma)=-\mathcal R_e\,\partial_{n_e}u_e^{\mathrm{int}}(x_\sigma).
	\]
	\revblue{Local Lipschitz continuity follows from the linearity of $\mathcal R_e$.}
\end{enumerate}

\begin{proposition}[Regularity and consistency of the boundary treatments employed in the numerical simulations in Section~\ref{sec:numerics}]
	Assume that the interior reconstruction is $C^1$ in a neighborhood of each boundary edge and that the prescribed inflow and exact boundary data admit a $C^1$ extension $\widetilde g_e$ near that edge. Then the four boundary treatments above provide the exterior state and the directional derivatives required by the point operator, are consistent with the corresponding boundary conditions, and depend in a locally Lipschitz manner on the interior trace and gradient data.
\end{proposition}

\begin{proof}
	Each boundary treatment is obtained by composing the interior reconstruction and the prescribed boundary data with elementary geometric maps (translation, orthogonal projection, reflection) and, in the reflective case, a linear state reflection. These operations are $C^1$ near the boundary and hence locally Lipschitz continuous. The explicit formulas above show consistency at the boundary: periodicity is enforced by translation, the inflow and exact-data treatment matches the prescribed boundary state, the outflow treatment imposes zero normal extrapolation, and the reflective treatment changes only the normal momentum component while preserving density and total energy. Thus, the boundary treatments used in the numerical simulations have the consistency and local Lipschitz properties required by the point-value analysis.
\end{proof}
}

\subsection{Optional oscillation-eliminating (OE) preprocessing}\label{sec:OE}
To suppress spurious oscillations near strong shocks, we adopt the cellwise contraction
\begin{equation}\label{eq:OE-scaling}
\mathcal M_K u_K:=\bar u_K+\theta_K^{\tt OE}(u_K-\bar u_K),\qquad 0\le \theta_K^{\tt OE}\le 1.
\end{equation}
We refer the reader to \cite{peng2025oedg} and \cite{Cao2025} for additional details and theoretical properties. 
When the OE preprocessing is applied prior to the IDP limiter,
\begin{equation}\label{eq:OE-combine-min}
\Pi_K(\mathcal M_Ku_K)=\bar u_K+\min\{\theta_K^{\tt OE},\theta_K\}(u_K-\bar u_K),
\end{equation}
where $\theta_K\in\{\theta_K^{\tt c},\theta_K^{\tt a}\}$ is the scaling parameter determined by the IDP limiter on the unscaled reconstruction $u_K$. The combined operator is again a contraction toward $\bar u_K$, thereby preserving the cell average and not expanding the admissible target enforced subsequently by $\Pi_K$.

\vspace{8mm}

\bibliographystyle{abbrv}
\bibliography{references}

\end{document}